\documentclass[12pt]{amsart}

\newtheorem{thm}{Theorem}
\newtheorem{lem}{Lemma}

\newtheorem{prop}{Proposition}
\newtheorem{ex}{Example}
\newtheorem{definition}{Definition}
\newtheorem{remark}{Remark}

\usepackage[dvipsnames]{xcolor} 

\usepackage[all]{xy}
\usepackage{tikz}
\usepackage{amsthm,enumitem} 

\usepackage[latin1]{inputenc}
\usepackage{subfigure}
\usepackage{color}
\usepackage{amsmath}
\usepackage{amsthm}
\usepackage{amstext}
\usepackage{amssymb}
\usepackage{amsfonts}
\usepackage{graphicx}
\usepackage{young}
\usepackage{multicol}
\usepackage{mathrsfs}
\usepackage{stmaryrd}
\usepackage{bbm}
\usepackage[all]{xy}
\usepackage{rotating}
\usepackage{tikz}
\usepackage{mathtools}
\usepackage{tabularx}
\usepackage{array}
\usepackage{commath}

\usepackage{tikz-cd}
\usepackage[new]{old-arrows}

\usepackage{lipsum}
\usepackage{caption}
\usepackage{subfigure}
\usepackage{breqn}

\title{Combinatorics of type $D$ exceptional sequences}

\author{Emily Carrick}
\author{Alexander Garver}

\begin{document}


\begin{abstract}{Exceptional sequences are important sequences of quiver representations in the study of representation theory of algebras. They are also closely related to the theory of cluster algebras and the combinatorics of Coxeter groups. We combinatorially classify exceptional sequences of a family of type $D$ Dynkin quivers, and we show how our model for exceptional sequences connects to the combinatorics of type $D$ noncrossing partitions.}\end{abstract}





\maketitle

\section{Introduction}

Exceptional sequences are certain homologically-defined sequences of quiver representations that are useful in understanding the structure of the associated bounded derived category of quiver representations. They were first studied by Crawley-Boevey \cite{c93},  who showed that the braid group acts transitively on the set of exceptional sequences of maximal length.  This result was then generalized by Ringel \cite{r94}. Exceptional sequences have also been connected to other areas of mathematics since their invention, including the combinatorics of Coxeter groups \cite{b03,hk13,ingalls2009noncrossing} and cluster algebras \cite{is10, st13}. In particular, exceptional sequences of representations of \textit{Dynkin quivers} (i.e., quivers whose underlying graph is a simply-laced Dynkin diagram) are in bijection with saturated chains in the corresponding lattice of noncrossing partitions that contain the minimal element \cite{ingalls2009noncrossing}.

Although exceptional sequences are very well studied, they have only been combinatorially classified when the quiver is a type $A$ Dynkin quiver \cite{a13, garver2015combinatorics}. In this paper, we present a combinatorial classification of the exceptional sequences of representations of the type $D_n$ Dynkin quivers $Q^n$ appearing in Figure \ref{fig:1}.
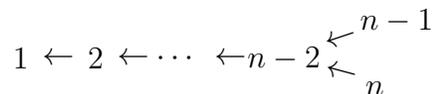
\begin{figure}[!htbp]
$$\begin{tikzpicture}
	\node[inner sep=0.5, fill=none, label=center:{1}](n1) at (0,0) {};
	\node[inner sep=0.5, fill=none, label=center:{2}](n2) at (1,0) {};
	\node[inner sep=0.5, fill=none](n6) at (2,0) {};
	\node[inner sep=0.5, fill=none](n7) at (2.1,0) {};
	\node[inner sep=1, fill=none, label=center:{$n-2$}](n3) at (3.5,0) {};
	\node[inner sep=0.5, fill=none, label=center:{$n-1$}](n4) at (5,0.5) {};
	\node[inner sep=0.5, fill=none, label=left:{$n$}](n5) at (5,-0.4) {};
	\path (n1)--(n2) node [midway, sloped] {$\leftarrow$};
	\path (n2)--(n6) node [midway] {$\leftarrow$};
	\path (n6)--(n7) node [midway] {$\cdots$};
	\path (n7)--(n3) node [midway, sloped] {$\leftarrow$};
	\path (n3)--(n4) node [midway, sloped] {$\leftarrow$};
	\path (n3)--(n5) node [midway, sloped] {$\leftarrow$};
\end{tikzpicture}$$
\caption{The type $D_n$ Dynkin quiver $Q^n$.}
\label{fig:1}
\end{figure}

For our purposes, we will define exceptional sequences as sequences of quiver representations, but they can equivalently be defined as sequences of objects in the bounded derived category of quiver representations. Since any two Dynkin quivers with the same underlying Dynkin diagram have triangle equivalent bounded derived categories of representations \cite{dieter1988triangulated}, to classify the exceptional sequences of one orientation of a Dynkin diagram is to classify them for them all orientations of the same Dynkin diagram (see Remark~\ref{der_cat_defn}). Therefore, our main result (Theorem~\ref{thm:2}) gives a combinatorial description of all exceptional sequences of objects in the bounded derived category of a type $D$ Dynkin quiver.

We classify exceptional sequences by showing that they are equivalent to certain sequences of curves on a punctured disk that we call \textit{exceptional sequences of curves}; see Section~\ref{sec_proof_idea} for the precise definition. After explaining the necessary background on exceptional sequences in Section~\ref{sec:exc_sequences}, we show how indecomposable representations of $Q^n$ may be thought of as pairs of curves on this punctured disk in Section~\ref{sec:geometric_model}.  We remark that the model we present Section~\ref{sec:geometric_model} is similar to the geometric model for cluster categories of type $D$ constructed by Schiffler \cite{schiffler2008geometric}. We present our classification of exceptional sequences in Section~\ref{sec_proof_idea}, and we interpret our model in terms of noncrossing partitions of type $D$ in Section~\ref{sec:noncrossing}. We prove our main results in Sections~\ref{sec_prop_1_proof}, \ref{sec_thm_1_2_proof}, and \ref{sec_thm_3_proof}.


\section{Exceptional sequences}\label{sec:exc_sequences}

A \textit{quiver} $Q$ is a 4-tuple $(Q_0,Q_1,s,t)$, where $Q_0$ is a set of \textit{vertices}, $Q_1$ is a set of \textit{arrows}, and $s, t:Q_1 \to Q_0$ are two functions defined so that for every $a \in Q_1$, we have $s(a) \xrightarrow{a} t(a)$. A \textit{representation} $V = ((V_k)_{k \in Q_0}, (f_a)_{a \in Q_1})$ of a quiver $Q$  is an assignment of a finite dimensional $\mathbb{K}$-vector space $V_k$ to each vertex $k$ and a $\mathbb{K}$-linear map $f_a: V_{s(a)} \rightarrow V_{t(a)}$ to each arrow $a$ where $\mathbb{K}$ is a field.  The \textit{dimension vector} of $V$ is the vector $\textbf{dim}(V):=(\dim V_k)_{k\in Q_0}$. The \textit{dimension} of $V$ is defined as $\text{dim}(V) = \sum_{k \in Q_0}\dim V_k.$

Given representations $V = ((V_k)_{k}, (f_a)_{a})$ and $W  = ((W_k)_{k}, (g_a)_{a})$, a \textit{morphism} $\theta : V \rightarrow W$ is a collection of linear maps $\theta_k : V_k \rightarrow W_k$ where $\theta_{t(a)} \circ f_a = g_a \circ \theta_{s(a)}$ for all $a \in Q_1$. We let $\text{Hom}(V,W)$ denote the $\mathbb{K}$-vector space of all morphisms from $V$ to $W$.

We remark that the category of all representations of $Q$ is equivalent to the category of finitely generated left modules over the path algebra of $Q$. This implies that there is a well-defined $\text{Ext}$-functor, denoted $\text{Ext}^k(-,-)$, for each $k \ge 1$ on the category of representations of $Q$.

In Figure~\ref{fig_first_repn_ex}, we show two examples of representations of the quiver $Q^5.$ We now define exceptional sequences of representations, which are our main object of study.

\begin{figure}[!htbp]
$$\includegraphics[scale=1]{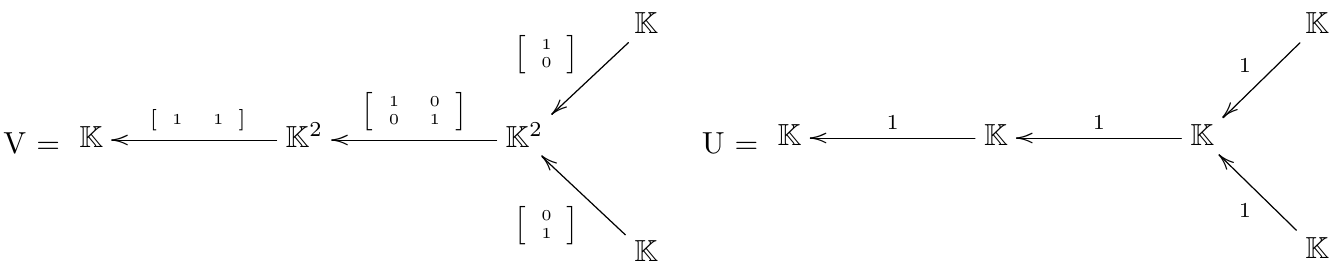}$$
\caption{Two examples of indecomposable representations of $Q^5.$}
\label{fig_first_repn_ex}
\end{figure}

\begin{definition}
An \textit{exceptional sequence} is a sequence of representations $(V^1,\ldots,V^\ell)$ of a quiver $Q$ that satisfies the following:
\begin{enumerate}
	\item[a)] $\text{Hom}(V^i,V^i)$ is a division algebra and $\text{Ext}^k(V^i,V^i)=0$ for all $k\geq1$ and all $i$, and
	\item[b)] $\text{Hom}(V^i,V^j)=0$ and $\text{Ext}^k(V^i,V^j)=0$ for all $k \ge 1$ and all pairs $i>j$.
\end{enumerate}
We say that a set of representations $\{V^1, \ldots, V^\ell\}$ is an \textit{exceptional collection} of representations if they may be totally ordered in some way so that they form an exceptional sequence.
\end{definition}

If $(V^1,\ldots,V^\ell)$ is an exceptional sequence, then property a) implies that each representation $V^i$ is indecomposable. Furthermore, in this paper, since we will only work with the type $D$ quivers $Q^n$, any representations $U$ and $V$ of $Q^n$ satisfy
\begin{itemize}
\item $\text{Ext}^k(U,V) = 0$ for all $k \ge 2,$ and 
\item $\text{Ext}^1(U,V) \simeq \text{D}\text{Hom}(V,\tau U)$
\end{itemize}
where $\tau$ is the \textit{Auslander--Reiten translation} and $\text{D}:= \text{Hom}_{\mathbb{K}}(-,\mathbb{K})$. We omit the homological definition of $\tau$, but we give a geometric description of it in Section~\ref{sec:geometric_model}. From these observations, to understand the exceptional sequences of $Q^n$, it is enough to understand the spaces of morphisms between all indecomposable representations of $Q^n$. 

We henceforth assume that $\mathbb{K}$ is an algebraically closed field. Therefore, Gabriel's Theorem implies that the indecomposable representations of $Q^n$ are in one-to-one correspondence with the positive roots of the root system corresponding to $Q^n$, under the map sending a representation to its dimension vector \cite{gabriel1972unzerlegbare}. To understand the morphisms between the indecomposables, we use the Auslander--Reiten quiver of $Q^n.$ We show the Auslander--Reiten quiver of $Q^5$ in Figure~\ref{D5_ar_fig}.

By definition, the \textit{Auslander--Reiten quiver} of $Q$ is the quiver with vertices indexed by the isomorphism classes of indecomposable representations of $Q$ and arrows indexed by a basis of the space of irreducible morphisms between the corresponding representations.  Since there are only finitely many isomorphism classes of indecomposable representations of $Q^n$, any morphism $\theta: U \to V$ between two indecomposable representations may be expressed as a sum of compositions of irreducible morphisms where the compositions correspond to paths  from $U$ to $V$ in the Auslander--Reiten quiver. 

It is also possible to calculate the Auslander--Reiten translation of a representation using the Auslander--Reiten quiver.we show how the Auslander--Reiten translation acts on indecomposable representations in Figure~\ref{D5_ar_fig}.



\begin{figure}
$$\includegraphics[scale=1]{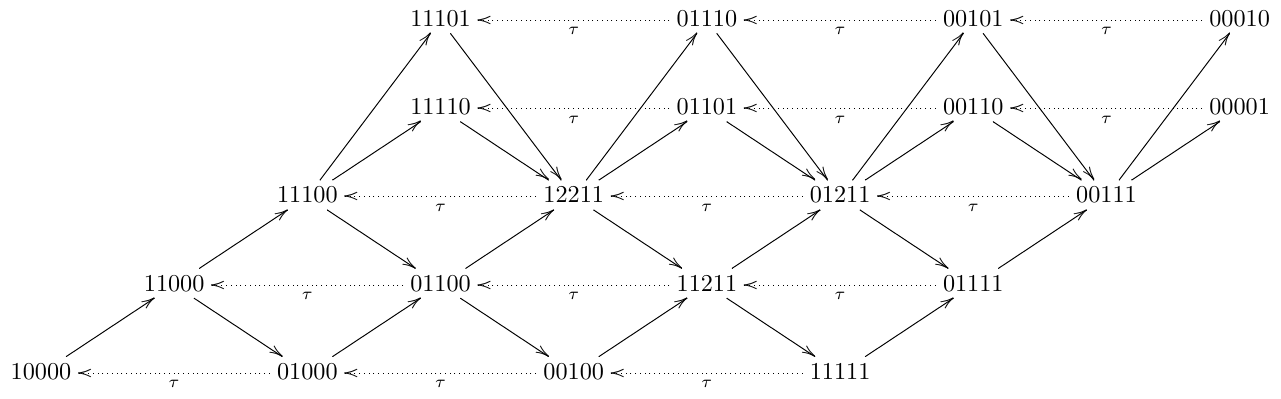}$$
\caption{The Auslander--Reiten quiver of $Q^5$. Its arrows are the solid black arrows. The dotted arrows show the action of $\tau$ on indecomposables. The leftmost indecomposables are sent to zero by $\tau$.}
\label{D5_ar_fig}
\end{figure}

\begin{ex}
The sequence $(V,U)$ is an exceptional sequence, but $(U,V)$ is not where $V$ and $U$ are the representations in Figure~\ref{fig_first_repn_ex}. One verifies this using the Auslander--Reiten quiver of $Q^5$. 
\end{ex}


\section{Curves on a punctured disk}\label{sec:geometric_model}

In this section, we introduce the geometric model that we use to classify type $D$ exceptional sequences. 

Consider the collection of $2n-1$ points in the plane such that the convex hull $2n-2$ of these points is a regular $(2n-2)$-gon and the remaining point is the centroid of this $(2n-2)$-gon. Label the vertices of the $(2n-2)$-gon with $1,2,\ldots,n-1,-1,-2,\ldots,-(n-1)$ in clockwise order, and the centroid with both $n$ and $-n$. Let $\Sigma_n$ denote the data of these labeled points. Equivalently, we think of $\Sigma_n$ as a disk with $2n-2$ marked points on its boundary and one interior marked point thought of as a puncture.


Let $\gamma:[0,1]\rightarrow \Sigma_n$ denote a geodesic curve whose endpoints are $i$ and $j$ such that $j\neq -i$ and $\gamma'$ the geodesic curve whose endpoints are $-i$ and $-j$. Define an equivalence relation on all geodesic curves on $\Sigma_n$ where $\gamma$ and $\gamma'$ are equivalent, and let $[\gamma]=\{\gamma,\gamma'\}$ denote the equivalence class containg $\gamma$. Let $\text{Eq}(\Sigma_n)$ denote the collection of all such equivalence classes of curves on $\Sigma_n$.  Given $[\gamma] \in \text{Eq}(\Sigma_n)$, we say $\gamma\in[\gamma]$ is the \textit{essential curve} of $[\gamma]$ if it has the smallest positive endpoint of either of the curves in $[\gamma].$ {In Figure~\ref{fig_gamma_gamma_bar}, we show two examples of elements of $\text{Eq}(\Sigma_n)$. The essential curves of these classes appear in blue.}


When $[\gamma] \in \text{Eq}(\Sigma_n)$ contains a curve connecting $i$ and $-n$, we will draw this curve with a solid dot in its interior. In this case, we let $[\overline{\gamma}] = \{\overline{\gamma}, \overline{\gamma}^\prime\}$ denote the equivalence class whose curves connect $i$ to $n$ and $-i$ to $-n$. {We set $ [\overline{\overline{\gamma}}] := [\gamma]$, and $[\overline{\delta}] := [\delta]$ for any $[\delta] \in \text{Eq}(\Sigma_n)$ where neither $\delta$ nor $\delta^\prime$ have $\pm n$ as an endpoint.} 
In Figure~\ref{fig_gamma_gamma_bar}, we show an example of two such an equivalence classes $[\gamma]$ and $[\overline{\gamma}]$. 

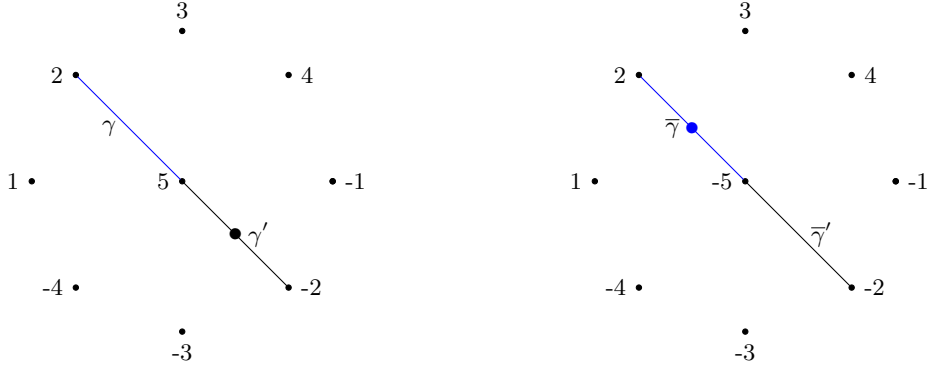
\begin{figure}
\centering
\begin{subfigure}
\centering
\hfill
\begin{tikzpicture}
\draw [fill] circle (1pt);
	
	\foreach \a in {0,45,...,360} {
		\draw[fill] (\a:2cm) circle (1pt); 
		}
		
\node[inner sep=0.5,fill=none,label=left:{\scriptsize{5}}]at (0:0) {};
\node[inner sep=0.5,fill=none,label=left:{\scriptsize{1}}]at (180:2) {};
\node[inner sep=0.5,fill=none,label=right:{\scriptsize{4}}]at (45:2) {};
\node[inner sep=0.5,fill=none,label=above:{\scriptsize{3}}]at (90:2) {};
\node[inner sep=0.5,fill=none,label=left:{\scriptsize{2}}]at (135:2) {};
\node[inner sep=0.5,fill=none,label=right:{\scriptsize{-1}}]at (180:-2) {};
\node[inner sep=0.5,fill=none,label=left:{\scriptsize{-4}}]at (45:-2) {};
\node[inner sep=0.5,fill=none,label=below:{\scriptsize{-3}}]at (90:-2) {};
\node[inner sep=0.5,fill=none,label=right:{\scriptsize{-2}}]at (135:-2) {};

	\node[inner sep=0.5, fill=none, label=right:{}] (a) at (135:2) {};
	\node[inner sep=0.5, fill=none, label=below:{}] (b) at (0:0) {};
	\path[every node/.style={font=\sffamily\small}]
(a) edge[color=blue] node [right] {} (b);	
	
	\node[inner sep=0.5, fill=none, label=right:{}] (-a) at (135:-2) {};
	\node[inner sep=0.5, fill=none, label=below:{}] (-b) at (0:0) {};
	\path[every node/.style={font=\sffamily\small}]
(-a) edge[color=black] node [midway] {$\bullet$} (-b);
\node[inner sep=0.5,fill=none,label=left:{\scriptsize{$\gamma$}}]at (135:1) {};
\node[inner sep=0.5,fill=none,label=right:{\scriptsize{$\gamma'$}}]at (135:-1) {};
\end{tikzpicture}
\end{subfigure}\hfill
\begin{subfigure}
\centering
\begin{tikzpicture}
\draw [fill] circle (1pt);
	
	\foreach \a in {0,45,...,360} {
		\draw[fill] (\a:2cm) circle (1pt); 
		}
\node[inner sep=0.5,fill=none,label=left:{\scriptsize{-5}}]at (0:0) {};		
\node[inner sep=0.5,fill=none,label=left:{\scriptsize{1}}]at (180:2) {};
\node[inner sep=0.5,fill=none,label=right:{\scriptsize{4}}]at (45:2) {};
\node[inner sep=0.5,fill=none,label=above:{\scriptsize{3}}]at (90:2) {};
\node[inner sep=0.5,fill=none,label=left:{\scriptsize{2}}]at (135:2) {};
\node[inner sep=0.5,fill=none,label=right:{\scriptsize{-1}}]at (180:-2) {};
\node[inner sep=0.5,fill=none,label=left:{\scriptsize{-4}}]at (45:-2) {};
\node[inner sep=0.5,fill=none,label=below:{\scriptsize{-3}}]at (90:-2) {};
\node[inner sep=0.5,fill=none,label=right:{\scriptsize{-2}}]at (135:-2) {};

	\node[inner sep=0.5, fill=none, label=right:{}] (f) at (135:2) {};
	\node[inner sep=0.5, fill=none, label=below:{}] (g) at (0:0) {};
	\path[every node/.style={font=\sffamily\small}]
	 (f) edge[color=blue] node [midway] {$\bullet$} (g);

	\node[inner sep=0.5, fill=none, label=right:{}] (-a) at (135:-2) {};
	\node[inner sep=0.5, fill=none, label=below:{}] (-b) at (0:0) {};
	\path[every node/.style={font=\sffamily\small}]
	(-a) edge[color=black] node [right] {} (-b);
\node[inner sep=0.5,fill=none,label=left:{\scriptsize{$\overline{\gamma}$}}]at (135:1) {};
\node[inner sep=0.5,fill=none,label=right:{\scriptsize{$\overline{\gamma}'$}}]at (135:-1) {};	

\end{tikzpicture}
\hfill
\end{subfigure}
\caption{Examples of two equivalence classes $[\gamma], [\overline{\gamma}] \in \text{Eq}(\Sigma_n).$ }
\label{fig_gamma_gamma_bar}
\end{figure}

\begin{definition}\label{def:2}
Let $[\gamma] \in \text{Eq}(\Sigma_n)$ with essential curve $\gamma$. Assume that $\gamma$ has endpoints $i$ and $j$ and $i$ is its smallest positive endpoint. Define $V(\gamma)= ((V(\gamma)_k)_k, (V(\gamma)_a)_a) = ((V_k)_{k}, (f_{a})_{a})$ to be the representation of $Q^n$ whose vector spaces are as follows:
\begin{enumerate}
\item[a)] if $i \in \{1,\ldots, n-2\}$ and $j \in \{i+1,\ldots, n-1\}$, then $$\begin{array}{ccccc}V_k & := & \left\{\begin{array}{lcl} \mathbb{K} & : & k \in \{i, \ldots, j-1\} \\ 0 & : & \text{otherwise} \end{array}\right.\end{array}$$ and $$\begin{array}{ccccc}f_a & := & \left\{\begin{array}{lcl} id & : & s(a), t(a) \in \{i, \ldots, j-1\} \\ 0 & : & \text{otherwise;} \end{array}\right.\end{array}$$
\item[b)] if $i \in \{1,\ldots, n-2\}$ and $j \in \{-(i+1),\ldots, -(n-1)\}$, then $$\begin{array}{ccccc}V_k & := & \left\{\begin{array}{lcl} \mathbb{K}^2 & : & k \in \{-j, \ldots, n-2\} \text{ for $-j \neq n-1$} \\ \mathbb{K} & : & k \in \{i, \ldots, -j-1\}\cup\{n-1,n\} \\ 0 & : & \text{otherwise} \end{array}\right.\end{array}$$ and $$\begin{array}{ccccc}f_a & := & \left\{\begin{array}{lcl} id & : & \dim(V_{s(a)}) = \dim(V_{t(a)}) \neq 0 \\ \left[\begin{array}{c} 1 \\ 0 \end{array}\right] & : & s(a) = n-1, t(a) = n-2 \\ \\ \left[\begin{array}{c} 0 \\ 1 \end{array}\right] & : & s(a) = n, t(a) = n-2 \\ \\ \left[\begin{array}{c} 1 \\ 1 \end{array}\right] & : & \dim(V_{s(a)}) = 2, \dim(V_{t(a)}) = 1 \\ 0 & : & \text{otherwise;} \end{array}\right.\end{array}$$ 
\item[c)] if $i \in \{1,\ldots, n-1\}$ and $j= n$ (resp., $j = -n$), then $$\begin{array}{ccccc}V_k & := & \left\{\begin{array}{lcl} \mathbb{K} & : & k \in \{i, i + 1, \ldots, n-2, n\} \text{ for $i \neq n-1$}\\ & & \text{(resp., $k \in \{i, i+1, \ldots, n-2, n-1\}$)}  \\ 0 & : & \text{otherwise} \end{array}\right.\end{array}$$ and $$\begin{array}{ccccc}f_a & := & \left\{\begin{array}{lcl} id & : & \dim(V_{s(a)}) = \dim(V_{t(a)}) = 1 \\ 0 & : & \text{otherwise.}\end{array}\right.\end{array}$$

\end{enumerate} The representation $V(\gamma)$ is defined above using the information of the essential curve $\gamma$ from the equivalence class $[\gamma] = \{\gamma, \gamma^\prime\}.$ For convenience, we also define $V(\gamma^\prime) := V(\gamma)$. 
\end{definition}

\begin{lem}
The map $[\gamma] \mapsto V(\gamma)$ defines a bijection between equivalence classes of curves on $\Sigma_n$ and the indecomposable representations of $Q^n.$
\end{lem}
\begin{proof}
It straightforward to check that every indecomposable representation of $Q^n$ is of the form $V(\gamma)$ for some $[\gamma] \in \text{Eq}(\Sigma_n)$. It therefore suffices to show that the map $[\gamma] \mapsto V(\gamma)$ is injective. A case-by-case check shows that given two distinct equivalences $[\gamma], [\delta] \in \text{Eq}(\Sigma_n)$, the representations $V(\gamma)$ and $V(\delta)$ have distinct dimension vectors. Therefore, $V(\gamma) \not \simeq V(\delta)$.
\end{proof}




\begin{ex}\label{commuting_curves_fig}
Pairs of equivalence classes $[\gamma]$ and $[\overline{\gamma}]$ such as the two appearing in Figure~\ref{fig_gamma_gamma_bar} are special in that both of the sequences $(V(\gamma), V(\overline{\gamma}))$ and $(V(\overline{\gamma}), V({\gamma}))$ are exceptional. This can be verified using the Auslander--Reiten quiver in Figure~\ref{D5_ar_fig}. This phenomenon differs from the second author's previous work where only curves that did not intersect corresponded to representations forming an exceptional sequence in either order \cite{garver2015combinatorics}.
\end{ex}

\begin{ex}\label{proj_and_inj_repns}
Figure~\ref{fig_proj_inj} shows the equivalence classes of curves corresponding to the indecomposable projective representations of $Q^5$ on the left and the equivalence classes of curves corresponding to the indecomposable injective representations of $Q^5$ on the right. We have not drawn the curves connected to the $-5$ as geodesics.
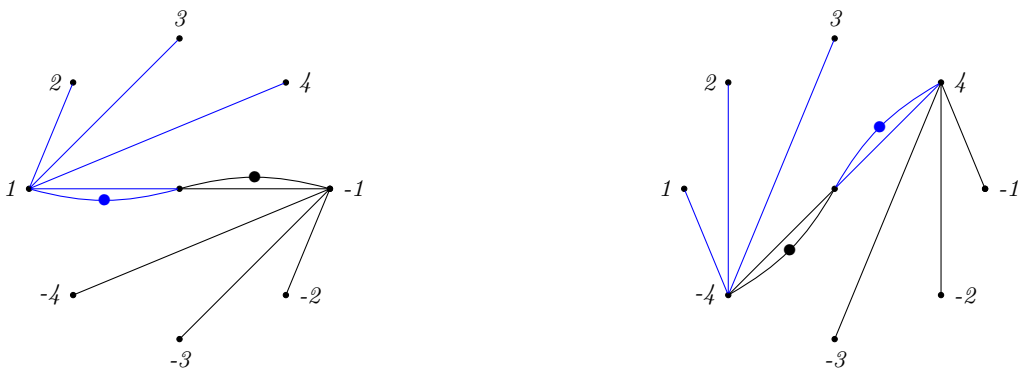
\begin{figure}[!htbp]
\centering
\begin{subfigure}%
\centering
\hfill
\begin{tikzpicture}
	\draw [fill] circle (1pt);
	
	\foreach \a in {0,45,...,360} {
		\draw[fill] (\a:2cm) circle (1pt); 
		}
\node[inner sep=0.5,fill=none,label=left:{\scriptsize{1}}](a)at (180:2) {};
\node[inner sep=0.5,fill=none,label=right:{\scriptsize{4}}](d)at (45:2) {};
\node[inner sep=0.5,fill=none,label=above:{\scriptsize{3}}](c)at (90:2) {};
\node[inner sep=0.5,fill=none,label=left:{\scriptsize{2}}](b)at (135:2) {};
\node[inner sep=0.5,fill=none,label=right:{\scriptsize{-1}}](e)at (180:-2) {};
\node[inner sep=0.5,fill=none,label=left:{\scriptsize{-4}}](h)at (45:-2) {};
\node[inner sep=0.5,fill=none,label=below:{\scriptsize{-3}}](g)at (90:-2) {};
\node[inner sep=0.5,fill=none,label=right:{\scriptsize{-2}}](f)at (135:-2) {};
\node[inner sep=0.5,fill=none](i)at (0:0) {};
	
	\path[every node/.style={font=\sffamily\small}]
(a) edge[color=blue] node [midway] {} (b);	
\path[every node/.style={font=\sffamily\small}]
(a) edge[color=blue] node [midway] {} (c);
\path[every node/.style={font=\sffamily\small}]
(a) edge[color=blue] node [midway] {} (d);
\path[every node/.style={font=\sffamily\small}]
(a) edge[color=blue] node [midway] {} (i);	
\path[every node/.style={font=\sffamily\small}]
(a) edge[bend right=15, color=blue] node [midway] {$\bullet$} (i);			
	\path[every node/.style={font=\sffamily\small}]
(e) edge[color=black] node [midway] {} (f);	
\path[every node/.style={font=\sffamily\small}]
(e) edge[color=black] node [midway] {} (g);
\path[every node/.style={font=\sffamily\small}]
(e) edge[color=black] node [midway] {} (h);
\path[every node/.style={font=\sffamily\small}]
(e) edge[color=black] node [midway] {} (i);	
\path[every node/.style={font=\sffamily\small}]
(e) edge[bend right=15, color=black] node [midway] {$\bullet$} (i);
	\end{tikzpicture}
	\hfill
\end{subfigure}
\centering
\begin{subfigure}%
\centering\hfill
\begin{tikzpicture}
	\draw [fill] circle (1pt);
	
	\foreach \a in {0,45,...,360} {
		\draw[fill] (\a:2cm) circle (1pt); 
		}
\node[inner sep=0.5,fill=none,label=left:{\scriptsize{1}}](a)at (180:2) {};
\node[inner sep=0.5,fill=none,label=right:{\scriptsize{4}}](d)at (45:2) {};
\node[inner sep=0.5,fill=none,label=above:{\scriptsize{3}}](c)at (90:2) {};
\node[inner sep=0.5,fill=none,label=left:{\scriptsize{2}}](b)at (135:2) {};
\node[inner sep=0.5,fill=none,label=right:{\scriptsize{-1}}](e)at (180:-2) {};
\node[inner sep=0.5,fill=none,label=left:{\scriptsize{-4}}](h)at (45:-2) {};
\node[inner sep=0.5,fill=none,label=below:{\scriptsize{-3}}](g)at (90:-2) {};
\node[inner sep=0.5,fill=none,label=right:{\scriptsize{-2}}](f)at (135:-2) {};
\node[inner sep=0.5,fill=none](i)at (0:0) {};
	
	\path[every node/.style={font=\sffamily\small}]
(h) edge[color=blue] node [midway] {} (a);	
\path[every node/.style={font=\sffamily\small}]
(h) edge[color=blue] node [midway] {} (b);
\path[every node/.style={font=\sffamily\small}]
(h) edge[color=blue] node [midway] {} (c);
\path[every node/.style={font=\sffamily\small}]
(i) edge[color=blue] node [midway] {} (d);	
\path[every node/.style={font=\sffamily\small}]
(i) edge[bend left=15, color=blue] node [midway] {$\bullet$} (d);			
	\path[every node/.style={font=\sffamily\small}]
(d) edge[color=black] node [midway] {} (e);	
\path[every node/.style={font=\sffamily\small}]
(d) edge[color=black] node [midway] {} (f);
\path[every node/.style={font=\sffamily\small}]
(d) edge[color=black] node [midway] {} (g);
\path[every node/.style={font=\sffamily\small}]
(h) edge[color=black] node [midway] {} (i);	
\path[every node/.style={font=\sffamily\small}]
(h) edge[bend right=15, color=black] node [midway] {$\bullet$} (i);
	\end{tikzpicture}
\hfill
\end{subfigure}
\caption{The essential curves in each equivalence class appear in blue.}
\label{fig_proj_inj}
\end{figure}

\end{ex}



Using our geometric model for indecomposable representations of $Q^n$, we can describe other representation theoretic objects related to $Q^n$. Given a curve $\gamma$ on $\Sigma_n$ that connects $i, j \in \{\pm 1, \ldots, \pm(n-1)\}$, define $\varrho(\gamma)$ to be the curve whose endpoints are immediately counterclockwise along the boundary of $\Sigma_n$ from $i$ and $j$. If $\gamma$ is a curve on $\Sigma_n$ that connects $i \in \{\pm 1, \ldots, \pm(n-1)\}$ and $j \in \{\pm n\}$, define $\varrho(\gamma)$ to be the connecting $-j$ and the point on the boundary of $\Sigma_n$ immediately counterclockwise from $i$. 

We show examples of how $\varrho$ acts on elements of $\text{Eq}(\Sigma_n)$ in Figure~\ref{fig:varrho}. We use the map $\varrho$ to describe how the Auslander--Reiten translation acts on indecomposable  representations and, therefore, on all representations of $Q^n$.

\begin{figure}
    \centering
    \hfill
\begin{subfigure}
\centering
\begin{tikzpicture}
\draw [fill] circle (1pt);
	
	\foreach \a in {0,45,...,360} {
		\draw[fill] (\a:2cm) circle (1pt); 
		}
		\node[inner sep=0.5,fill=none,label=left:{\scriptsize{1}}]at (180:2) {};	
		
		\node[inner sep=0.5,fill=none,label=left:{\scriptsize{1}}]at (180:2) {};
\node[inner sep=0.5,fill=none,label=right:{\scriptsize{4}}]at (45:2) {};
\node[inner sep=0.5,fill=none,label=above:{\scriptsize{3}}]at (90:2) {};
\node[inner sep=0.5,fill=none,label=left:{\scriptsize{2}}]at (135:2) {};
\node[inner sep=0.5,fill=none,label=right:{\scriptsize{-1}}]at (180:-2) {};
\node[inner sep=0.5,fill=none,label=left:{\scriptsize{-4}}]at (45:-2) {};
\node[inner sep=0.5,fill=none,label=below:{\scriptsize{-3}}]at (90:-2) {};
\node[inner sep=0.5,fill=none,label=right:{\scriptsize{-2}}]at (135:-2) {};
		
	\node[inner sep=0.5, fill=none, label=right:{}] (a) at (135:2) {};
	\node[inner sep=0.5, fill=none, label=right:{}] (b) at (225:2) {};
	\node[inner sep=0.5, fill=none, label=right:{}] (c) at (180:2) {};
	\node[inner sep=0.5, fill=none, label=right:{}] (d) at (270:2) {};
	\node[inner sep=0.5, fill=none, label=below:{}] (e) at (0:0) {};
			
	\path[every node/.style={font=\sffamily\small}] 
	(a) edge[color=blue] node [right] {} (b);
	\path[every node/.style={font=\sffamily\small}] 
	(c) edge[densely dotted, color=red] node [right] {} (d);
	
	\node[inner sep=0.5, fill=none, label=right:{}] (a) at (135:-2) {};
	\node[inner sep=0.5, fill=none, label=right:{}] (b) at (225:-2) {};
	\node[inner sep=0.5, fill=none, label=right:{}] (c) at (180:-2) {};
	\node[inner sep=0.5, fill=none, label=right:{}] (d) at (270:-2) {};
	
	\path[every node/.style={font=\sffamily\small}] 
	(a) edge[color=blue] node [right] {} (b);
	\path[every node/.style={font=\sffamily\small}] 
	(c) edge[densely dotted, color=red] node [right] {} (d);
	
	\node[inner sep=0.5, fill=none, label=left:{$\scriptsize{\gamma}$}] (a) at (180:0.9) {};
	\node[inner sep=0.5, fill=none, label=right:{$\scriptsize{\gamma'}$}] (a) at (180:-0.8) {};
	\node[inner sep=0.5, fill=none, label=right:{$\scriptsize{\varrho(\gamma')}$}] (a) at (80:1.8) {};
	\node[inner sep=0.5, fill=none, label=left:{$\scriptsize{\varrho(\gamma)}$}] (a) at (80:-1.8) {};

\end{tikzpicture}
\end{subfigure}\hfill
\begin{subfigure}
\centering
\begin{tikzpicture}
\draw [fill] circle (1pt);
	
	\foreach \a in {0,45,...,360} {
		\draw[fill] (\a:2cm) circle (1pt); 
		}
		\node[inner sep=0.5,fill=none,label=left:{\scriptsize{1}}]at (180:2) {};
		
		\node[inner sep=0.5,fill=none,label=left:{\scriptsize{1}}]at (180:2) {};
\node[inner sep=0.5,fill=none,label=right:{\scriptsize{4}}]at (45:2) {};
\node[inner sep=0.5,fill=none,label=above:{\scriptsize{3}}]at (90:2) {};
\node[inner sep=0.5,fill=none,label=left:{\scriptsize{2}}]at (135:2) {};
\node[inner sep=0.5,fill=none,label=right:{\scriptsize{-1}}]at (180:-2) {};
\node[inner sep=0.5,fill=none,label=left:{\scriptsize{-4}}]at (45:-2) {};
\node[inner sep=0.5,fill=none,label=below:{\scriptsize{-3}}]at (90:-2) {};
\node[inner sep=0.5,fill=none,label=right:{\scriptsize{-2}}]at (135:-2) {};	
		
	\node[inner sep=0.5, fill=none, label=right:{}] (a) at (45:2) {};
	\node[inner sep=0.5, fill=none, label=right:{}] (b) at (45:-2) {};
	\node[inner sep=0.5, fill=none, label=right:{}] (c) at (90:2) {};
	\node[inner sep=0.5, fill=none, label=right:{}] (d) at (90:-2) {};
	\node[inner sep=0.5, fill=none, label=below:{}] (e) at (0:0) {};
			
	\path[every node/.style={font=\sffamily\small}] 
	(a) edge[color=blue] node [right] {} (e);
	\path[every node/.style={font=\sffamily\small}] 
	(b) edge[color=blue] node [midway] {$\bullet$} (e);
	\path[every node/.style={font=\sffamily\small}] 
	(c) edge[densely dotted,color=red] node [midway] {$\bullet$} (e);
	\path[every node/.style={font=\sffamily\small}] 
	(d) edge[densely dotted, color=red] node [right] {} (e);
	
	\node[inner sep=0.5, fill=none, label=right:{$\scriptsize{\gamma}$}] (a) at (45:0.9) {};
	\node[inner sep=0.5, fill=none, label=left:{$\scriptsize{\gamma'}$}] (a) at (45:-0.9) {};
	\node[inner sep=0.5, fill=none, label=left:{$\scriptsize{\varrho(\gamma)}$}] (a) at (90:1) {};
	\node[inner sep=0.5, fill=none, label=right:{$\scriptsize{\varrho(\gamma')}$}] (a) at (90:-1) {};
\end{tikzpicture}\hfill
\end{subfigure}
    \caption{The action of $\varrho$ on elements of $\text{Eq}(\Sigma_n).$}
    \label{fig:varrho}
\end{figure}
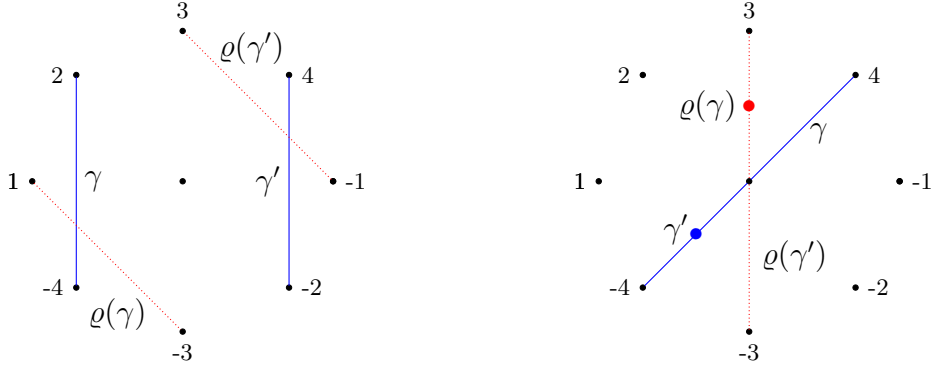

\begin{prop}\label{prop_tau_action}
The action of $\varrho$ on $\text{Eq}(\Sigma_n)$ induces an action on the indecomposable representations of $Q^n$ as follows: $$\begin{array}{ccccc}\varrho: V(\gamma) & \mapsto & V(\varrho(\gamma)) \simeq \left\{\begin{array}{cccc}
\tau V(\gamma) \text{ if } \tau V(\gamma) \neq 0\\
\nu V(\gamma) \text{ if } \tau V(\gamma) = 0\end{array}. \right. \end{array}$$
Here, the map $\nu$ is the \textit{Nakayama functor}.
\end{prop}

\section{Classification of exceptional sequences}\label{sec_proof_idea}


In this section, we state the main theorems of our work. We first explain the conditions under which a collection of representations form an exceptional sequence in some order. We do this by giving conditions on the corresponding collection of equivalence classes of curves.

Let $[\gamma]$ and $[\delta]$ be distinct elements of $\text{Eq}(\Sigma_n)$. If $\gamma$ and $\delta$ or $\gamma$ and $\delta^\prime$ intersect in their interiors, we say that $[\gamma]$ and $[\delta]$ are \textit{crossing}. Otherwise, we say that $[\gamma]$ and $[\delta]$ are \textit{noncrossing}.


\begin{definition}
We say that two distinct equivalence classes $[\gamma], [\delta] \in \text{Eq}(\Sigma_n)$ where $[\delta] \neq [\overline{\gamma}]$ are a \textit{bad pair} if one of the following holds: 
\begin{itemize}
    \item the classes $[\gamma]$ and $[\delta]$ are crossing, or
    \item one has that $\{\text{endpoints of } [\gamma]\} = \{\text{endpoints of } [\delta]\}$.
\end{itemize}
\end{definition}
\begin{figure}[!htbp]
\centering
\begin{subfigure}%
\centering
\hfill
\begin{tikzpicture}
	\draw [fill] circle (1pt);
	
	\foreach \a in {0,45,...,360} {
		\draw[fill] (\a:2cm) circle (1pt); 
		}
\node[inner sep=0.5,fill=none,label=left:{\scriptsize{1}}]at (180:2) {};
\node[inner sep=0.5,fill=none,label=right:{\scriptsize{4}}]at (45:2) {};
\node[inner sep=0.5,fill=none,label=above:{\scriptsize{3}}]at (90:2) {};
\node[inner sep=0.5,fill=none,label=left:{\scriptsize{2}}]at (135:2) {};
\node[inner sep=0.5,fill=none,label=right:{\scriptsize{-1}}]at (180:-2) {};
\node[inner sep=0.5,fill=none,label=left:{\scriptsize{-4}}]at (45:-2) {};
\node[inner sep=0.5,fill=none,label=below:{\scriptsize{-3}}]at (90:-2) {};
\node[inner sep=0.5,fill=none,label=right:{\scriptsize{-2}}]at (135:-2) {};
\node[inner sep=0.5,fill=none,label=below:{\scriptsize{5,-5}}]at (0:0) {};
		
	\node[inner sep=0.5, fill=none, label=right:{}] (n-1) at (0:2) {};
	\node[inner sep=0.5, fill=none, label=left:{}] (-n-1) at (0:-2) {};
	\node[inner sep=0.5, fill=none, label=left:{}] at (165:2) {};
	\node[inner sep=0.5, fill=none, label=right:{}] at (-15:2) {};
	
	
	\node[inner sep=0.5, fill=none, label=right:{}] (i) at (315:2) {};
	\node[inner sep=0.5, fill=none, label=left:{}] (j) at (180:2) {};
	\draw (i)--(j) [color=blue] {};
	
	\node[inner sep=0.5, fill=none, label=left:{}] (-i) at (180:-2) {};
	\node[inner sep=0.5, fill=none, label=right:{}] (-j) at (315:-2) {};
	\draw (-i)--(-j) [color=blue] {};

	
	\node[inner sep=0.5, fill=none, label=right:{}] (k) at (180:2) {};
	\node[inner sep=0.5, fill=none, label=above:{}] (l) at (90:2) {};
	\draw (k)--(l) [color=black] {};

	
	\node[inner sep=0.5, fill=none,label={$\scriptsize{V}$}] (-l) at (225:1.4) {};
	
	\node[inner sep=0.5, fill=none, label=right:{}] (-k) at (180:-2) {};
	\node[inner sep=0.5, fill=none, label=below:{}] (-l) at (90:-2) {};
	\draw (-k)--(-l) [color=black] {};
	\end{tikzpicture}
	\hfill
\end{subfigure}
\centering
\begin{subfigure}%
\centering\hfill
\begin{tikzpicture}
	\draw [fill] circle (1pt);
	
	\foreach \a in {0,45,...,360} {
		\draw[fill] (\a:2cm) circle (1pt); 
		}
\node[inner sep=0.5,fill=none,label=left:{\scriptsize{1}}]at (180:2) {};
\node[inner sep=0.5,fill=none,label=right:{\scriptsize{4}}]at (45:2) {};
\node[inner sep=0.5,fill=none,label=above:{\scriptsize{3}}]at (90:2) {};
\node[inner sep=0.5,fill=none,label=left:{\scriptsize{2}}]at (135:2) {};
\node[inner sep=0.5,fill=none,label=right:{\scriptsize{-1}}]at (180:-2) {};
\node[inner sep=0.5,fill=none,label=left:{\scriptsize{-4}}]at (45:-2) {};
\node[inner sep=0.5,fill=none,label=below:{\scriptsize{-3}}]at (90:-2) {};
\node[inner sep=0.5,fill=none,label=right:{\scriptsize{-2}}]at (135:-2) {};
\node[inner sep=0.5,fill=none,label=below:{\scriptsize{5,-5}}]at (0:0) {};

	\node[inner sep=0.5, fill=none, label=right:{}] (n-1) at (0:2) {};
	\node[inner sep=0.5, fill=none, label=left:{}] (-n-1) at (0:-2) {};
	\node[inner sep=0.5, fill=none, label=left:{}] at (165:2) {};
	\node[inner sep=0.5, fill=none, label=right:{}] at (-15:2) {};
	
	
	\node[inner sep=0.5, fill=none, label=right:{}] (i) at (225:2) {};
	\node[inner sep=0.5, fill=none, label=left:{}] (j) at (180:2) {};
	\draw (i)--(j) [color=blue] {};
	
	\node[inner sep=0.5, fill=none, label=left:{}] (-i) at (180:-2) {};
	\node[inner sep=0.5, fill=none, label=right:{}] (-j) at (225:-2) {};
	\draw (-i)--(-j) [color=blue] {};

	
	\node[inner sep=0.5, fill=none, label=right:{}] (k) at (180:2) {};
	\node[inner sep=0.5, fill=none, label=above:{}] (l) at (45:2) {};
	\draw (k)--(l) [color=black] {};

	
	\node[inner sep=0.5, fill=none, label=right:{}] (-k) at (180:-2) {};
	\node[inner sep=0.5, fill=none, label=below:{}] (-l) at (45:-2) {};
	\draw (-k)--(-l) [color=black] {};
	
	 \node[inner sep=0.5, fill=none,label=below:{$\scriptsize{U}$}] (-l) at (190:1.5) {};
	
\end{tikzpicture}
\hfill
\end{subfigure}
\caption{Two examples of bad pairs. In blue, we show the equivalence class of curves corresponding to the representations $U$ and $V$ from Figure \ref{fig_first_repn_ex}.}
\end{figure}
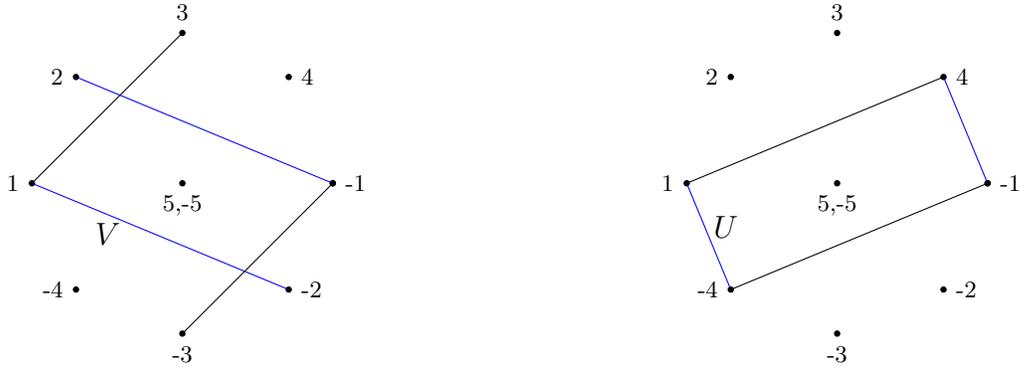






Now, let $\{[\gamma_i]\}_{i = 1}^k \subset \text{Eq}(\Sigma_n)$ with $k \le n$ be a collection equivalence classes where $\gamma_i$ is the essential curve of $[\gamma_i]$ for each $i \in \{1,\ldots, k\}$. Such a collection naturally defines a graph whose vertices are the $2n-1$ unlabeled vertices of $\Sigma_n$ and whose edges are all of the curves appearing in the equivalence classes $\{[\gamma_i]\}_{i = 1}^k$. We will refer to this as the \textit{graph determined} by $\{[\gamma_i]\}_{i = 1}^k.$

We say that $\mathcal{E} =\{[\gamma_i]\}_{i = 1}^k$ is an \textit{exceptional collection of curves} if 
\begin{itemize}
\item[(a)] classes $[\gamma_i]$ and $[\gamma_j]$ are not a bad pair for any distinct $i, j \in \{1, \ldots, k\}$,
\item[(b)] the only cycles in the graph determined by $\{[\gamma_i]\}_{i = 1}^k$ are those induced by a configuration of curves in $\{[\gamma_i]\}_{i = 1}^k$ of the form shown in Figure~\ref{fig_special_configuration} and at most one such configuration belongs to $\{[\gamma_i]\}_{i = 1}^k$, and
\item[(c)] among the curves $\gamma_1,  \ldots, \gamma_k$ there are no configurations of the form appearing in Figure~\ref{fig:forbidden_cconfiguration}.
\end{itemize}
We show two examples of an exceptional collection of curves and one collection of curves that is not exceptional in Figure~\ref{fig_exc_coll_of curves}.

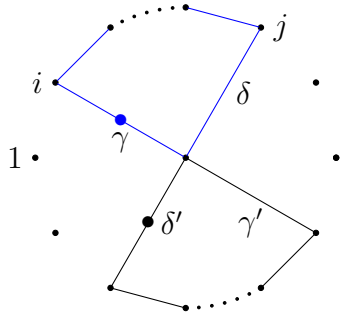
\begin{figure}[!htb]
\centering
\begin{tikzpicture}
\draw [fill] circle (1pt);
	
	\foreach \a in {0,30,...,360} {
		\draw[fill] (\a:2cm) circle (1pt); 
		}
		\foreach \a in {90,95,...,120} {
		\draw[fill] (\a:2cm) circle (0.5pt);
		}
		
		\foreach \a in {270,275,...,300} {
		\draw[fill] (\a:2cm) circle (0.5pt);
		}
\node[inner sep=0.5,fill=none,label=left:{$1$}]at (180:2) {};	

\node[inner sep=0.5,fill=none,label=left:{$i$}]at (150:2) {};	
\node[inner sep=0.5,fill=none,label=right:{$j$}]at (60:2) {};	

	\node[inner sep=0.5, fill=none, label=right:{}] (a) at (60:2) {};
	\node[inner sep=0.5, fill=none, label=right:{}] (b) at (150:2) {};
	\node[inner sep=0.5, fill=none, label=right:{}] (c) at (120:2) {};
	\node[inner sep=0.5, fill=none, label=right:{}] (d) at (90:2) {};
	\node[inner sep=0.5, fill=none, label=below:{}] (e) at (0:0) {};
	
	\path[every node/.style={font=\sffamily\small}] 
	(a) edge[color=blue] node [right] {} (e);
	\path[every node/.style={font=\sffamily\small}] 
	(b) edge[color=blue] node [midway] {$\bullet$} (e);
	\path[every node/.style={font=\sffamily\small}] 
	(c) edge[color=blue] node [right] {} (b);
	\path[every node/.style={font=\sffamily\small}] 
	(d) edge[color=blue] node [right] {} (a);
	
	\node[inner sep=0.5, fill=none, label=right:{}] (a) at (60:-2) {};
	\node[inner sep=0.5, fill=none, label=right:{}] (b) at (150:-2) {};
	\node[inner sep=0.5, fill=none, label=right:{}] (c) at (120:-2) {};
	\node[inner sep=0.5, fill=none, label=right:{}] (d) at (90:-2) {};
	\node[inner sep=0.5, fill=none, label=below:{}] (e) at (0:0) {};
	
	\path[every node/.style={font=\sffamily\small}] 
	(a) edge[color=black] node [midway] {$\bullet$} (e);
	\path[every node/.style={font=\sffamily\small}] 
	(b) edge[color=black] node [midway] {} (e);
	\path[every node/.style={font=\sffamily\small}] 
	(c) edge[color=black] node [right] {} (b);
	\path[every node/.style={font=\sffamily\small}] 
	(d) edge[color=black] node [right] {} (a);

\node[inner sep=0.5, fill=none, label=right:{$\delta$}]at (60:1) {};
	\node[inner sep=0.5, fill=none, label=below:{$\gamma$}] at (150:1) {};	
	\node[inner sep=0.5, fill=none, label=right:{$\delta'$}]at (60:-1) {};
	\node[inner sep=0.5, fill=none, label=below:{$\gamma '$}] at (150:-1) {};

\end{tikzpicture}

\caption{The only cycle configurations allowed in exceptional collections of curves, up to the action of $\varrho$ and $\overline{(\cdot)}$ on all of the curves in the configuration. Here we allow $i$ to equal $j$. If $i \neq j$, then there must be an angle strictly less than $\pi$ between $\gamma$ and $\delta$, and therefore the same must be true of $\gamma^\prime$ and $\delta^\prime$.}
    \label{fig_special_configuration}
\end{figure}

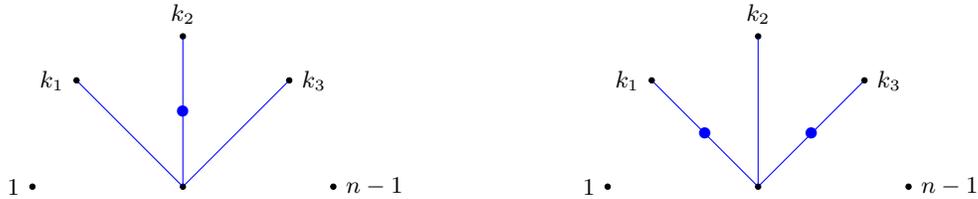
\begin{figure}[!htbp]
    \centering
    \hfill
    \begin{subfigure}
    \hfill
    \begin{tikzpicture}
    \draw [fill] circle (1pt);
	
	\foreach \a in {0,45,...,180} {
		\draw[fill] (\a:2cm) circle (1pt); 
		}
\node[inner sep=0.5,fill=none,label=left:{\scriptsize{$k_1$}}]at (135:2) {};
\node[inner sep=0.5,fill=none,label=right:{\scriptsize{$k_3$}}]at (45:2) {};
\node[inner sep=0.5,fill=none,label=above:{\scriptsize{$k_2$}}]at (90:2) {};
\node[inner sep=0.5,fill=none,label=left:{\scriptsize{$1$}}]at (180:2) {};
\node[inner sep=0.5,fill=none,label=right:{\scriptsize{$n-1$}}]at (180:-2) {};

\node[inner sep=0.5,fill=none](i)at (90:2) {};
\node[inner sep=0.5,fill=none,label=right:{}](j)at (45:2) {};
\node[inner sep=0.5,fill=none,label=below:{}](n) at (0:0) {};
\node[inner sep=0.5, fill=none, label=left:{}](k) at (135:2) {};

\path[every node/.style={font=\sffamily\small}]
(n) edge[color=blue] node [midway] {$\bullet$} (i);
\draw (j)--(n) [color=blue] {};
\draw (k)--(n) [color=blue] {};

    \end{tikzpicture}
    \end{subfigure}\hfill
    \begin{subfigure}
    \centering
    \begin{tikzpicture}
    \draw [fill] circle (1pt);
	
	\foreach \a in {0,45,...,180} {
		\draw[fill] (\a:2cm) circle (1pt); 
		}
\node[inner sep=0.5,fill=none,label=left:{\scriptsize{$k_1$}}]at (135:2) {};
\node[inner sep=0.5,fill=none,label=right:{\scriptsize{$k_3$}}]at (45:2) {};
\node[inner sep=0.5,fill=none,label=above:{\scriptsize{$k_2$}}]at (90:2) {};
\node[inner sep=0.5,fill=none,label=left:{\scriptsize{$1$}}]at (180:2) {};
\node[inner sep=0.5,fill=none,label=right:{\scriptsize{$n-1$}}]at (180:-2) {};

\node[inner sep=0.5,fill=none](i)at (90:2) {};
\node[inner sep=0.5,fill=none,label=right:{}](j)at (45:2) {};
\node[inner sep=0.5,fill=none,label=below:{}](n) at (0:0) {};
\node[inner sep=0.5, fill=none, label=left:{}](k) at (135:2) {};

\path[every node/.style={font=\sffamily\small}]
(n) edge[color=blue] node [midway] {$\bullet$} (j);
\path[every node/.style={font=\sffamily\small}]
(n) edge[color=blue] node [midway] {$\bullet$} (k);
\draw (i)--(n) [color=blue] {};
    \end{tikzpicture}
    \hfill
    \end{subfigure}
    \caption{The two types of configurations of three essential curves that may not appear in an exceptional collection of curves. Here $k_1, k_2, k_3 \in \{1, \ldots, n-1\}$ and $k_1 < k_2 < k_3$.}
    \label{fig:forbidden_cconfiguration}
\end{figure}

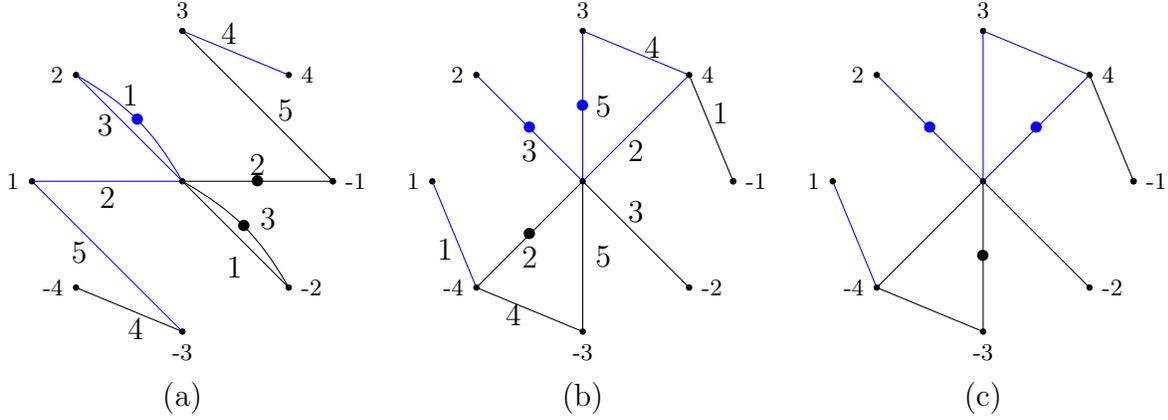
\begin{figure}[!htb]
\centering
\begin{subfigure}
\centering
\begin{tikzpicture}
\draw [fill] circle (1pt);
	
	\foreach \a in {0,45,...,360} {
		\draw[fill] (\a:2cm) circle (1pt); 
		}
\node[inner sep=0.5,fill=none,label=left:{\scriptsize{1}}]at (180:2) {};
\node[inner sep=0.5,fill=none,label=right:{\scriptsize{4}}]at (45:2) {};
\node[inner sep=0.5,fill=none,label=above:{\scriptsize{3}}]at (90:2) {};
\node[inner sep=0.5,fill=none,label=left:{\scriptsize{2}}]at (135:2) {};
\node[inner sep=0.5,fill=none,label=right:{\scriptsize{-1}}]at (180:-2) {};
\node[inner sep=0.5,fill=none,label=left:{\scriptsize{-4}}]at (45:-2) {};
\node[inner sep=0.5,fill=none,label=below:{\scriptsize{-3}}]at (90:-2) {};
\node[inner sep=0.5,fill=none,label=right:{\scriptsize{-2}}]at (135:-2) {};
	\node[inner sep=0.5, fill=none, label=right:{}] (a) at (270:2) {};
	\node[inner sep=0.5, fill=none, label=below:{}] (b) at (225:2) {};
	\draw (a)--(b) [color=black] {};
	
	\node[inner sep=0.5, fill=none, label=right:{}] (-a) at (270:-2) {};
	\node[inner sep=0.5, fill=none, label=below:{}] (-b) at (225:-2) {};
	\draw (-a)--(-b) [color=blue] {};
	\node[inner sep=0.5, fill=none, label=right:{}] (a) at (180:2) {};
	\node[inner sep=0.5, fill=none, label=below:{}] (b) at (270:2) {};
	\draw (a)--(b) [color=blue] {};
	
	\node[inner sep=0.5, fill=none, label=right:{}] (-a) at (180:-2) {};
	\node[inner sep=0.5, fill=none, label=below:{}] (-b) at (270:-2) {};
	\draw (-a)--(-b) [color=black] {};
	\node[inner sep=0.5, fill=none, label=right:{}] (a) at (180:2) {};
	\node[inner sep=0.5, fill=none, label=below:{}] (b) at (0:0) {};
	\draw (a)--(b) [color=blue] {};
	
	\node[inner sep=0.5, fill=none, label=right:{}] (-a) at (180:-2) {};
	\node[inner sep=0.5, fill=none, label=below:{}] (-b) at (0:0) {};
	\path[every node/.style={font=\sffamily\small}]
	 (-a) edge[color=black] node [midway] {$\bullet$} (-b);
	\node[inner sep=0.5, fill=none, label=right:{}] (f) at (135:2) {};
	\node[inner sep=0.5, fill=none, label=below:{}] (g) at (0:0) {};
	\path[every node/.style={font=\sffamily\small}]
	 (f) edge[bend left=15, color=blue] node [midway] {$\bullet$} (g);

	\node[inner sep=0.5, fill=none, label=right:{}] (-a) at (135:-2) {};
	\node[inner sep=0.5, fill=none, label=below:{}] (-b) at (0:0) {};
	\path[every node/.style={font=\sffamily\small}]
	(-a) edge[color=black] node [midway] {} (-b);
	
	\node[inner sep=0.5, fill=none, label=right:{}] (a) at (135:2) {};
	\node[inner sep=0.5, fill=none, label=below:{}] (b) at (0:0) {};
	\path[every node/.style={font=\sffamily\small}]
(a) edge[color=blue] node [right] {} (b);	
	
	\node[inner sep=0.5, fill=none, label=right:{}] (-a) at (135:-2) {};
	\node[inner sep=0.5, fill=none, label=below:{}] (-b) at (0:0) {};
	\path[every node/.style={font=\sffamily\small}]
(-a) edge[bend right=15] node [midway] {$\bullet$} (-b);
	
	\node[inner sep=0.5, fill=none, label=right:{5}] at (210:1.9) {};
	\node[inner sep=0.5, fill=none, label=right:{4}] at (80:2) {};
	\node[inner sep=0.5, fill=none, label=right:{1}] at (130:1.5) {};
	\node[inner sep=0.5, fill=none, label=below:{2}] at (175:1) {};
	\node[inner sep=0.5, fill=none, label=right:{3}] at (150:1.5) {};
	
	\node[inner sep=0.5, fill=none, label=left:{5}] at (210:-1.9) {};
	\node[inner sep=0.5, fill=none, label=left:{4}] at (80:-2) {};
	\node[inner sep=0.5, fill=none, label=left:{1}] at (130:-1.5) {};
	\node[inner sep=0.5, fill=none, label=above:{2}] at (175:-1) {};
	\node[inner sep=0.5, fill=none, label=right:{3}] at (150:-1) {};
	
	\node[inner sep=0.5,fill=none,label=below:{(a)}]at (90:-2.5) {};
\end{tikzpicture}
\end{subfigure}
\begin{subfigure}
\centering
\begin{tikzpicture}
\draw [fill] circle (1pt);
	
	\foreach \a in {0,45,...,360} {
		\draw[fill] (\a:2cm) circle (1pt); 
		}

\node[inner sep=0.5,fill=none,label=left:{\scriptsize{1}}]at (180:2) {};
\node[inner sep=0.5,fill=none,label=right:{\scriptsize{4}}]at (45:2) {};
\node[inner sep=0.5,fill=none,label=above:{\scriptsize{3}}]at (90:2) {};
\node[inner sep=0.5,fill=none,label=left:{\scriptsize{2}}]at (135:2) {};
\node[inner sep=0.5,fill=none,label=right:{\scriptsize{-1}}]at (180:-2) {};
\node[inner sep=0.5,fill=none,label=left:{\scriptsize{-4}}]at (45:-2) {};
\node[inner sep=0.5,fill=none,label=below:{\scriptsize{-3}}]at (90:-2) {};
\node[inner sep=0.5,fill=none,label=right:{\scriptsize{-2}}]at (135:-2) {};

	\node[inner sep=0.5, fill=none, label=right:{}] (a) at (0:0) {};
	\node[inner sep=0.5, fill=none, label=below:{}] (b) at (45:2) {};
	\node[inner sep=0.5, fill=none, label=right:{}] (c) at (90:2) {};
	\node[inner sep=0.5, fill=none, label=below:{}] (d) at (135:2) {};
	\node[inner sep=0.5, fill=none, label=below:{}] (e) at (180:2) {};
	\node[inner sep=0.5, fill=none, label=right:{}] (f) at (225:2) {};
	
	\path[every node/.style={font=\sffamily\small}]
(a) edge[color=blue] node [midway] {} (b);

\path[every node/.style={font=\sffamily\small}]
(a) edge[color=blue] node [midway] {$\bullet$} (c);
\path[every node/.style={font=\sffamily\small}]
(a) edge[color=blue] node [midway] {$\bullet$} (d);
\path[every node/.style={font=\sffamily\small}]
(b) edge[color=blue] node [midway] {} (c);
\path[every node/.style={font=\sffamily\small}]
(e) edge[color=blue] node [midway] {} (f);

	\node[inner sep=0.5, fill=none, label=below:{2}] (b) at (45:1) {};
	\node[inner sep=0.5, fill=none, label=right:{5}] (c) at (90:1) {};
	\node[inner sep=0.5, fill=none, label=below:{3}] (d) at (135:1) {};
	\node[inner sep=0.5, fill=none, label=right:{4}] (e) at (70:1.9) {};
	\node[inner sep=0.5, fill=none, label=left:{1}] (f) at (210:1.8) {};

\node[inner sep=0.5, fill=none, label=right:{}] (a) at (0:0) {};
	\node[inner sep=0.5, fill=none, label=below:{}] (b) at (45:-2) {};
	\node[inner sep=0.5, fill=none, label=right:{}] (c) at (90:-2) {};
	\node[inner sep=0.5, fill=none, label=below:{}] (d) at (135:-2) {};
	\node[inner sep=0.5, fill=none, label=below:{}] (e) at (180:-2) {};
	\node[inner sep=0.5, fill=none, label=right:{}] (f) at (225:-2) {};
	
	\path[every node/.style={font=\sffamily\small}]
(a) edge[color=black] node [midway] {$\bullet$} (b);

\path[every node/.style={font=\sffamily\small}]
(a) edge[color=black] node [midway] {} (c);
\path[every node/.style={font=\sffamily\small}]
(a) edge[color=black] node [midway] {} (d);
\path[every node/.style={font=\sffamily\small}]
(b) edge[color=black] node [midway] {} (c);
\path[every node/.style={font=\sffamily\small}]
(e) edge[color=black] node [midway] {} (f);

	\node[inner sep=0.5, fill=none, label=below:{2}] (b) at (45:-1) {};
	\node[inner sep=0.5, fill=none, label=right:{5}] (c) at (90:-1) {};
	\node[inner sep=0.5, fill=none, label=above:{3}] (d) at (135:-1) {};
	\node[inner sep=0.5, fill=none, label=left:{4}] (e) at (70:-1.9) {};
	\node[inner sep=0.5, fill=none, label=right:{1}] (f) at (210:-1.8) {};
\node[inner sep=0.5,fill=none,label=below:{(b)}]at (90:-2.5) {};
\end{tikzpicture}
\end{subfigure}
\begin{tikzpicture}
\draw [fill] circle (1pt);
	
	\foreach \a in {0,45,...,360} {
		\draw[fill] (\a:2cm) circle (1pt); 
		}
		
1
\node[inner sep=0.5,fill=none,label=left:{\scriptsize{1}}]at (180:2) {};
\node[inner sep=0.5,fill=none,label=right:{\scriptsize{4}}]at (45:2) {};
\node[inner sep=0.5,fill=none,label=above:{\scriptsize{3}}]at (90:2) {};
\node[inner sep=0.5,fill=none,label=left:{\scriptsize{2}}]at (135:2) {};
\node[inner sep=0.5,fill=none,label=right:{\scriptsize{-1}}]at (180:-2) {};
\node[inner sep=0.5,fill=none,label=left:{\scriptsize{-4}}]at (45:-2) {};
\node[inner sep=0.5,fill=none,label=below:{\scriptsize{-3}}]at (90:-2) {};
\node[inner sep=0.5,fill=none,label=right:{\scriptsize{-2}}]at (135:-2) {};

	\node[inner sep=0.5, fill=none, label=right:{}] (a) at (0:0) {};
	\node[inner sep=0.5, fill=none, label=below:{}] (b) at (45:2) {};
	\node[inner sep=0.5, fill=none, label=right:{}] (c) at (90:2) {};
	\node[inner sep=0.5, fill=none, label=below:{}] (d) at (135:2) {};
	\node[inner sep=0.5, fill=none, label=below:{}] (e) at (180:2) {};
	\node[inner sep=0.5, fill=none, label=right:{}] (f) at (225:2) {};
	
	\path[every node/.style={font=\sffamily\small}]
(a) edge[color=blue] node [midway] {$\bullet$} (b);

\path[every node/.style={font=\sffamily\small}]
(a) edge[color=blue] node [midway] {} (c);
\path[every node/.style={font=\sffamily\small}]
(a) edge[color=blue] node [midway] {$\bullet$} (d);
\path[every node/.style={font=\sffamily\small}]
(b) edge[color=blue] node [midway] {} (c);
\path[every node/.style={font=\sffamily\small}]
(e) edge[color=blue] node [midway] {} (f);

\node[inner sep=0.5, fill=none, label=right:{}] (a) at (0:0) {};
	\node[inner sep=0.5, fill=none, label=below:{}] (b) at (45:-2) {};
	\node[inner sep=0.5, fill=none, label=right:{}] (c) at (90:-2) {};
	\node[inner sep=0.5, fill=none, label=below:{}] (d) at (135:-2) {};
	\node[inner sep=0.5, fill=none, label=below:{}] (e) at (180:-2) {};
	\node[inner sep=0.5, fill=none, label=right:{}] (f) at (225:-2) {};
	
	\path[every node/.style={font=\sffamily\small}]
(a) edge[color=black] node [midway] {} (b);

\path[every node/.style={font=\sffamily\small}]
(a) edge[color=black] node [midway] {$\bullet$} (c);
\path[every node/.style={font=\sffamily\small}]
(a) edge[color=black] node [midway] {} (d);
\path[every node/.style={font=\sffamily\small}]
(b) edge[color=black] node [midway] {} (c);
\path[every node/.style={font=\sffamily\small}]
(e) edge[color=black] node [midway] {} (f);
\node[inner sep=0.5,fill=none,label=below:{(c)}]at (90:-2.5) {};
\end{tikzpicture}
\begin{subfigure}
\centering

\end{subfigure}\hfill
\caption{The collection of curves in (a) and (b) correspond to exceptional sequences. In (a) we draw the two curves connected to $-5$ not as geodesics to make it clear that four curves are incident to the centroid of $\Sigma_5$. The collection of curves in (c) is not an exceptional collection.}
\label{fig_exc_coll_of curves}
\end{figure}

\begin{thm}\label{thm1}
There is a bijection between exceptional collections of curves on $\Sigma_n$ and exceptional collections of representations of $Q^n$ given by
\[
\mathcal{E} = \{[\gamma_i]\}_{i=1}^k \mapsto \{V(\gamma_i)\}_{i = 1}^k
\]
where $\gamma_i$ is the essential curve of $[\gamma_i]$.
\end{thm}

Using our model, we now explicitly describe the ways in which we may order the representations coming from an exceptional collection of curves to produce an exceptional sequence. 

{Let $[\gamma]$ and $[\delta]$ be two distinct equivalences classes that are not a bad pair.} We say that $[\gamma]$ and $[\delta]$ \textit{share an endpoint} if the curves $\gamma$ and $\delta$ or $\gamma$ and $\delta^\prime$ share an endpoint $\ell \in \{\pm 1, \ldots, \pm n\}$. Assume that $\gamma$ and $\delta$ are the essential curves of $[\gamma]$ and $[\delta]$, respectively. We say that $[\gamma]$ is \textit{clockwise} from $[\delta]$ if one of the following holds:
\begin{enumerate}
\item[1)] the classes $[\gamma]$ and $[\delta]$ share an endpoint $\ell \in \{\pm 1, \ldots, \pm (n-1)\}$ where $\gamma$ is incident to $\ell$ and is clockwise about $\ell$ from $\delta$ or $\delta^\prime$ by a positive angle of rotation strictly less than $\pi$, 
\item[2)] the curve $\gamma$ (resp., $\delta$) connects $s > 0$ and $\pm n$ (resp., $t > 0$ and $\pm n$) and $s > t$.
\end{enumerate}
In this situation, we say that $[\delta]$ is \textit{counterclockwise} from $[\gamma]$.

Given an exceptional collection of curves $\mathcal{E}$, we will consider total orders on these equivalence classes of curves defined as follows. Let $[\gamma_i]$ and $[\gamma_j]$ be distinct equivalence classes where $\{\text{endpoints of $[\gamma_i]$}\} \neq \{\text{endpoints of $[\gamma_j]$}\}$ and which share an endpoint $\ell$. Furthermore, assume that $\gamma_i$ and $\gamma_j$ are the essential curves of $[\gamma_i]$ and $[\gamma_j]$, respectively. Now suppose that any such pair has the property that $i < j$ if the following holds:
\begin{itemize}
\item $[\gamma_{j}]$ is clockwise from $[\gamma_{i}]$ about $\ell$, or
\item if $\gamma_{j}$ has endpoints $k_j$ and  $\pm n$ and $\gamma_{i}$ has endpoints $k_i$ and $\mp n$, respectively, then $k_j < k_i.$ 
\end{itemize}
In this case, we call the sequence $\overrightarrow{\mathcal{E}} = ([\gamma_i])_{i=1}^k$ an \textit{exceptional sequence of curves}. That at least one such order on the classes in $\overrightarrow{\mathcal{E}}$ exists follows from the properties of exceptional collections of curves.

Observe that there is no prescribed order on pairs of classes $[\gamma_i]$ and $[\gamma_j] = [\overline{\gamma_i}]$. This is intentional because the corresponding representations $V(\gamma_i)$ and $V(\overline{\gamma_i})$ define exceptional sequences in either order. As the following theorem shows, exceptional sequences of curves are equivalent to exceptional sequences of representations.






\begin{thm}\label{thm:2}
There is a bijection between exceptional sequences of curves on $\Sigma_n$ and exceptional sequences of representations of $Q^n$ given by 
\[
\overrightarrow{\mathcal{E}} = ([\gamma_i])_{i=1}^k \mapsto (V(\gamma_i))_{i=1}^k
\]
where $\gamma_i$ is the essential curve of $[\gamma_i].$
\end{thm}

\begin{remark}\label{der_cat_defn}
Here, we recall the definition of exceptional sequences of objects in the bounded derived category of representations of $Q$. We use this definition to explain how Theorem~\ref{thm:2} provides a classification of exceptional sequences in the bounded derived category of $Q^n$. 

Let $\mathcal{D}$ denote the bounded derived category of representations of a quiver $Q$ over $\mathbb{K}$, and let $[1]$ denote its shift functor. We say that $X \in \mathcal{D}$ is \textit{exceptional} if $\text{Hom}_{\mathcal{D}}(X,X) \simeq \mathbb{K}$ and $\text{Hom}_{\mathcal{D}}(X[\ell], X) = 0$ for any integer $\ell \neq 0$. A sequence $(X_1, \ldots, X_k)$ of exceptional objects of $\mathcal{D}$ is called an \textit{exceptional sequence} if for each $i,j \in \{1, \ldots, k\}$ with $i < j$ one has that $\text{Hom}_{\mathcal{D}}(X_j[\ell],X_i) = 0$ for all integers $\ell$.

When $Q$ is a Dynkin quiver, it is known that for any indecomposable object $X$ in $\mathcal{D}$, there is a unique integer $\ell$ and a unique indecomposable representation $V$ of $Q$ such that $V \simeq X[\ell]$. In addition, for any exceptional sequence $(X_1, \ldots, X_k)$ in $\mathcal{D}$, the sequence $(X_1[\ell_1], \ldots, X_k[\ell_k])$ is also an exceptional sequence of objects in $\mathcal{D}$. Using these facts and that any two orientations of $Q$ give rise to triangle-equivalent derived categories \cite{dieter1988triangulated}, we see that any exceptional sequence of objects of $\mathcal{D}$ is unique identified with an exceptional sequence of representations of $Q$. 

Consequently, our classification of exceptional sequences of representations of $Q^n$ in \ref{thm:2} classifies all exceptional sequences in the bounded derived category of $Q^n$.
\end{remark}

\section{Type $D$ noncrossing partitions}\label{sec:noncrossing}

In this section, we interpret our geometric model in terms of a combinatorial model for noncrossing partitions of type $D$ introduced by Athanasiadis and Reiner \cite{athanasiadis2004noncrossing}.


A \textit{$D_n$-partition} is a partition $\pi$ of the $2n$ labeled points on $\Sigma_n$ into blocks that satisfying the following:
\begin{itemize}
\item[(1)] if $B$ is a block of $\pi$, then $-B$ is also a block of $\pi$, and 
\item[(2)] there is at most one \textit{zero-block} (if present) that contains $\{i,-i\}$ for some $i\in [n]^\pm := \{\pm 1, \ldots, \pm n\}$ and it does not consist of a single pair.
\end{itemize}
Let $\Pi^D(n)$ denote the poset of all $D_n$-partitions ordered by refinement. This poset is a geometric lattice. In fact, it is isomorphic to the intersection lattice of the type $D_n$ Coxeter hyperplane arrangement.

Now, let $\pi$ be a $D_n$-partition, $B \in \pi$ one of its blocks, and $\rho(B)$ the convex hull of the elements of $B$. We say that two blocks $B$ and $B^\prime$ of $\pi$ \textit{cross} if $\rho(B)$ and $\rho(B')$ do not coincide and one of them contains a point of the other in its relative interior. A $D_n$-partition $\pi$ is called \textit{noncrossing} if no two distinct blocks of $\pi$ are crossing. Let $NC^D(n)$ denote the subposet of $\Pi^D(n)$ of all noncrossing $D_n$-partitions. This poset is a lattice, and it is isomorphic to the lattice of noncrossing partitions of the associated Coxeter group \cite[Theorem 1.1]{athanasiadis2004noncrossing}.

Ingalls and Thomas have already connected the theory of exceptional sequences to the combinatorics of lattices of noncrossing partitions of finite Coxeter groups \cite{ingalls2009noncrossing}. More specifically, they proved that saturated chains  that contain the minimal element of these lattices are in bijection with exceptional sequences of representations of a suitable Dynkin quiver of the corresponding type. We have the following theorem, which shows that exceptional sequences of curves provide a combinatorial model for these chains in the lattice $NC^D(n)$.

\begin{thm}\label{thm:5.1}
Exceptional sequences of curves on $\Sigma_n$ are in bijection with the saturated chains of  NC$^D(n)$ containing its minimal element.
\end{thm}


\begin{ex} Here we give an example of our bijection from Theorem~\ref{thm:5.1}. Consider the exceptional sequence $(00101,00001,01000,01110,10000)$. Sequentially adding the equivalence classes of curves corresponding to these representations to $\Sigma_5$, as shown in Figure \ref{partitions}, produces the following maximal chain of noncrossing partitions in $NC^D(5)$: 
\[\begin{array}{c}
\{[n]^\pm\}\\
\{\{2,3,4,5, -2,-3,-4,-5\}\} \sqcup \{\{i\} \mid i \in [n]^\pm \backslash \{\pm 2, \pm 3,\pm 4, \pm 5\}\}  \\
\{\{2,3,4,5\}, \{-2,-3,-4,-5\}\} \sqcup \{\{i\} \mid i \in [n]^\pm \backslash \{\pm 2, \pm 3,\pm 4, \pm 5\}\}  \\
\{\{3,4,5\}, \{-3,-4,-5\}\} \sqcup \{\{i\} \mid i \in [n]^\pm \backslash \{\pm 3,\pm 4, \pm 5\}\}  \\
\{\{3,5\}, \{-3,-5\}\} \sqcup \{\{i\} \mid i \in [n]^\pm \backslash \{\pm 3, \pm 5\}\}  \\
\{\{i\} \mid i \in [n]^\pm\} \end{array}.\] 
\end{ex}





\begin{figure}[htb!]
\centering
\begin{minipage}{0.33\textwidth}
\begin{tikzpicture}


\draw [fill] circle (1pt);

	\foreach \a in {0,45,...,360} {
		\draw[fill] (\a:2cm) circle (1pt); 
		}
		
		\node[inner sep=0.5, fill=none, label=left:{\scriptsize{$1$}}](1) at (180:2) {};
		\node[inner sep=0.5, fill=none, label=left:{\scriptsize{$2$}}](2) at (135:2) {};
		\node[inner sep=0.5, fill=none, label=above:{\scriptsize{$3$}}](3) at (90:2) {};
		\node[inner sep=0.5, fill=none, label=right:{\scriptsize{$4$}}](4) at (45:2) {};
		
		\node[inner sep=0.5, fill=none, label=right:{\scriptsize{$-1$}}](-1) at (180:-2) {};
		\node[inner sep=0.5, fill=none, label=right:{\scriptsize{$-2$}}](-2) at (135:-2) {};
		\node[inner sep=0.5, fill=none, label=below:{\scriptsize{$-3$}}](-3) at (90:-2) {};
		\node[inner sep=0.5, fill=none, label=left:{\scriptsize{$-4$}}](-4) at (45:-2) {};
		\node[inner sep=0.5, fill=none, label=right:{$\longrightarrow$}](-1) at (180:-2.5) {};

\end{tikzpicture}
\end{minipage}%
\begin{minipage}{0.33\textwidth}
\begin{tikzpicture}


\draw [fill] circle (1pt);

	\foreach \a in {0,45,...,360} {
		\draw[fill] (\a:2cm) circle (1pt); 
		}
		
		\node[inner sep=0.5, fill=none, label=left:{\scriptsize{$1$}}](1) at (180:2) {};
		\node[inner sep=0.5, fill=none, label=left:{\scriptsize{$2$}}](2) at (135:2) {};
		\node[inner sep=0.5, fill=none, label=above:{\scriptsize{$3$}}](3) at (90:2) {};
		\node[inner sep=0.5, fill=none, label=right:{\scriptsize{$4$}}](4) at (45:2) {};
		
		\node[inner sep=0.5, fill=none, label=right:{\scriptsize{$-1$}}](-1) at (180:-2) {};
		\node[inner sep=0.5, fill=none, label=right:{\scriptsize{$-2$}}](-2) at (135:-2) {};
		\node[inner sep=0.5, fill=none, label=below:{\scriptsize{$-3$}}](-3) at (90:-2) {};
		\node[inner sep=0.5, fill=none, label=left:{\scriptsize{$-4$}}](-4) at (45:-2) {};
		\node[inner sep=0.5, fill=none, label=right:{$\longrightarrow$}](-1) at (180:-2.5) {};
		\node[inner sep=0.5, fill=none, label=right:{}](0) at (0:0) {};
	
	\path[every node/.style={font=\sffamily\small}]
		(3) edge[color=blue] node [midway] {} (0);	
	\path[every node/.style={font=\sffamily\small}]
		(-3) edge[color=black] node [midway] {$\bullet$} (0);

\node[inner sep=0.5, fill=none, label=right:{1}] at (90:1) {};
\node[inner sep=0.5, fill=none, label=right:{1}] at (90:-1) {};

\end{tikzpicture}
\end{minipage}%
\begin{minipage}{0.33\textwidth}
\begin{tikzpicture}
\draw [fill] circle (1pt);

	\foreach \a in {0,45,...,360} {
		\draw[fill] (\a:2cm) circle (1pt); 
		}
		
		\node[inner sep=0.5, fill=none, label=left:{\scriptsize{$1$}}](1) at (180:2) {};
		\node[inner sep=0.5, fill=none, label=left:{\scriptsize{$2$}}](2) at (135:2) {};
		\node[inner sep=0.5, fill=none, label=above:{\scriptsize{$3$}}](3) at (90:2) {};
		\node[inner sep=0.5, fill=none, label=right:{\scriptsize{$4$}}](4) at (45:2) {};
		
		\node[inner sep=0.5, fill=none, label=right:{\scriptsize{$-1$}}](-1) at (180:-2) {};
		\node[inner sep=0.5, fill=none, label=right:{\scriptsize{$-2$}}](-2) at (135:-2) {};
		\node[inner sep=0.5, fill=none, label=below:{\scriptsize{$-3$}}](-3) at (90:-2) {};
		\node[inner sep=0.5, fill=none, label=left:{\scriptsize{$-4$}}](-4) at (45:-2) {};

	\path[every node/.style={font=\sffamily\small}]
		(3) edge[color=blue] node [midway] {} (0);	
	\path[every node/.style={font=\sffamily\small}]
		(-3) edge[color=black] node [midway] {$\bullet$} (0);
	
	\path[every node/.style={font=\sffamily\small}]
		(4) edge[color=blue] node [midway] {} (0);	
	\path[every node/.style={font=\sffamily\small}]
		(-4) edge[color=black] node [midway] {$\bullet$} (0);

\draw[fill=gray!50,opacity=0.3 ] (0,0) -- (45 : 2cm) -- (90:2cm)  -- cycle;
\draw[fill=gray!50, opacity=0.3 ] (0,0) -- (45 : -2cm) -- (90:-2cm)  -- cycle;

\node[inner sep=0.5, fill=none, label=right:{$\longrightarrow$}](-1) at (180:-2.5) {};

\node[inner sep=0.5, fill=none, label=right:{1}] at (90:1) {};
\node[inner sep=0.5, fill=none, label=right:{1}] at (90:-1) {};	

\node[inner sep=0.5, fill=none, label=right:{2}] at (45:1) {};
\node[inner sep=0.5, fill=none, label=right:{2}] at (45:-1) {};

\end{tikzpicture}
\end{minipage}
\begin{minipage}{0.33\textwidth}
\begin{tikzpicture}


\draw [fill] circle (1pt);

	\foreach \a in {0,45,...,360} {
		\draw[fill] (\a:2cm) circle (1pt); 
		}
		
		\node[inner sep=0.5, fill=none, label=left:{\scriptsize{$1$}}](1) at (180:2) {};
		\node[inner sep=0.5, fill=none, label=left:{\scriptsize{$2$}}](2) at (135:2) {};
		\node[inner sep=0.5, fill=none, label=above:{\scriptsize{$3$}}](3) at (90:2) {};
		\node[inner sep=0.5, fill=none, label=right:{\scriptsize{$4$}}](4) at (45:2) {};
		
		\node[inner sep=0.5, fill=none, label=right:{\scriptsize{$-1$}}](-1) at (180:-2) {};
		\node[inner sep=0.5, fill=none, label=right:{\scriptsize{$-2$}}](-2) at (135:-2) {};
		\node[inner sep=0.5, fill=none, label=below:{\scriptsize{$-3$}}](-3) at (90:-2) {};
		\node[inner sep=0.5, fill=none, label=left:{\scriptsize{$-4$}}](-4) at (45:-2) {};
		\node[inner sep=0.5, fill=none, label=right:{$\longrightarrow$}](-1) at (180:-2.5) {};
	
	\path[every node/.style={font=\sffamily\small}]
		(3) edge[color=blue] node [midway] {} (0);	
	\path[every node/.style={font=\sffamily\small}]
		(-3) edge[color=black] node [midway] {$\bullet$} (0);
	
	\path[every node/.style={font=\sffamily\small}]
		(4) edge[color=blue] node [midway] {} (0);	
	\path[every node/.style={font=\sffamily\small}]
		(-4) edge[color=black] node [midway] {$\bullet$} (0);	
		\path[every node/.style={font=\sffamily\small}]
		(3) edge[color=blue] node [midway] {} (2);	
	\path[every node/.style={font=\sffamily\small}]
		(-3) edge[color=black] node [midway] {} (-2);

\draw[fill=gray!50,opacity=0.3 ] (0,0) -- (45 : 2cm) -- (90:2cm)  -- cycle;
\draw[fill=gray!50, opacity=0.3 ] (0,0) -- (45 : -2cm) -- (90:-2cm)  -- cycle;

\draw[fill=gray!50,opacity=0.3 ] (0,0) -- (90 : 2cm) -- (135:2cm)  -- cycle;
\draw[fill=gray!50, opacity=0.3 ] (0,0) -- (90 : -2cm) -- (135:-2cm)  -- cycle;

\node[inner sep=0.5, fill=none, label=right:{1}] at (90:1) {};
\node[inner sep=0.5, fill=none, label=right:{1}] at (90:-1) {};	

\node[inner sep=0.5, fill=none, label=right:{2}] at (45:1) {};
\node[inner sep=0.5, fill=none, label=right:{2}] at (45:-1) {};	

\node[inner sep=0.5, fill=none, label=left:{3}] at (112:2) {};
\node[inner sep=0.5, fill=none, label=right:{3}] at (112:-2) {};

\end{tikzpicture}
\end{minipage}%
\begin{minipage}{0.33\textwidth}
\begin{tikzpicture}


\draw [fill] circle (1pt);

	\foreach \a in {0,45,...,360} {
		\draw[fill] (\a:2cm) circle (1pt); 
		}
		
		\node[inner sep=0.5, fill=none, label=left:{\scriptsize{$1$}}](1) at (180:2) {};
		\node[inner sep=0.5, fill=none, label=left:{\scriptsize{$2$}}](2) at (135:2) {};
		\node[inner sep=0.5, fill=none, label=above:{\scriptsize{$3$}}](3) at (90:2) {};
		\node[inner sep=0.5, fill=none, label=right:{\scriptsize{$4$}}](4) at (45:2) {};
		
		\node[inner sep=0.5, fill=none, label=right:{\scriptsize{$-1$}}](-1) at (180:-2) {};
		\node[inner sep=0.5, fill=none, label=right:{\scriptsize{$-2$}}](-2) at (135:-2) {};
		\node[inner sep=0.5, fill=none, label=below:{\scriptsize{$-3$}}](-3) at (90:-2) {};
		\node[inner sep=0.5, fill=none, label=left:{\scriptsize{$-4$}}](-4) at (45:-2) {};
		\node[inner sep=0.5, fill=none, label=right:{$\longrightarrow$}](-1) at (180:-2.5) {};
		\node[inner sep=0.5, fill=none, label=right:{}](0) at (0:0) {};
	
	\path[every node/.style={font=\sffamily\small}]
		(3) edge[color=blue] node [midway] {} (0);	
	\path[every node/.style={font=\sffamily\small}]
		(-3) edge[color=black] node [midway] {$\bullet$} (0);
		
	\path[every node/.style={font=\sffamily\small}]
		(3) edge[color=blue] node [midway] {} (0);	
	\path[every node/.style={font=\sffamily\small}]
		(-3) edge[color=black] node [midway] {$\bullet$} (0);
	
	\path[every node/.style={font=\sffamily\small}]
		(4) edge[color=blue] node [midway] {} (0);	
	\path[every node/.style={font=\sffamily\small}]
		(-4) edge[color=black] node [midway] {$\bullet$} (0);	
		\path[every node/.style={font=\sffamily\small}]
		(3) edge[color=blue] node [midway] {} (2);	
	\path[every node/.style={font=\sffamily\small}]
		(-3) edge[color=black] node [midway] {} (-2);
	\path[every node/.style={font=\sffamily\small}]
		(2) edge[color=blue] node [midway] {$\bullet$} (0);	
	\path[every node/.style={font=\sffamily\small}]
		(-2) edge[color=black] node [midway] {} (0);

\draw[fill=gray!50,opacity=0.3 ] (0,0) -- (45 : 2cm) -- (90:2cm)  -- cycle;
\draw[fill=gray!50, opacity=0.3 ] (0,0) -- (45 : -2cm) -- (90:-2cm)  -- cycle;

\draw[fill=gray!50,opacity=0.3 ] (0,0) -- (90 : 2cm) -- (135:2cm)  -- cycle;
\draw[fill=gray!50, opacity=0.3 ] (0,0) -- (90 : -2cm) -- (135:-2cm)  -- cycle;	

\draw[fill=gray!50,opacity=0.3 ] (0,0) -- (135 : 2cm) -- (225:2cm)  -- cycle;
\draw[fill=gray!50, opacity=0.3 ] (0,0) -- (135 : -2cm) -- (225:-2cm)  -- cycle;		
		
\node[inner sep=0.5, fill=none, label=right:{1}] at (90:1) {};
\node[inner sep=0.5, fill=none, label=right:{1}] at (90:-1) {};	

\node[inner sep=0.5, fill=none, label=right:{2}] at (45:1) {};
\node[inner sep=0.5, fill=none, label=right:{2}] at (45:-1) {};	

\node[inner sep=0.5, fill=none, label=left:{3}] at (112:2) {};
\node[inner sep=0.5, fill=none, label=right:{3}] at (112:-2) {};		

\node[inner sep=0.5, fill=none, label=left:{4}] at (135:1) {};
\node[inner sep=0.5, fill=none, label=right:{4}] at (135:-1) {};

\end{tikzpicture}
\end{minipage}%
\begin{minipage}{0.33\textwidth}
\begin{tikzpicture}


\draw [fill] circle (1pt);

	\foreach \a in {0,45,...,360} {
		\draw[fill] (\a:2cm) circle (1pt); 
		}
		
		\node[inner sep=0.5, fill=none, label=left:{\scriptsize{$1$}}](1) at (180:2) {};
		\node[inner sep=0.5, fill=none, label=left:{\scriptsize{$2$}}](2) at (135:2) {};
		\node[inner sep=0.5, fill=none, label=above:{\scriptsize{$3$}}](3) at (90:2) {};
		\node[inner sep=0.5, fill=none, label=right:{\scriptsize{$4$}}](4) at (45:2) {};
		
		\node[inner sep=0.5, fill=none, label=right:{\scriptsize{$-1$}}](-1) at (180:-2) {};
		\node[inner sep=0.5, fill=none, label=right:{\scriptsize{$-2$}}](-2) at (135:-2) {};
		\node[inner sep=0.5, fill=none, label=below:{\scriptsize{$-3$}}](-3) at (90:-2) {};
		\node[inner sep=0.5, fill=none, label=left:{\scriptsize{$-4$}}](-4) at (45:-2) {};

	\path[every node/.style={font=\sffamily\small}]
		(3) edge[color=blue] node [midway] {} (0);	
	\path[every node/.style={font=\sffamily\small}]
		(-3) edge[color=black] node [midway] {$\bullet$} (0);
		
	\path[every node/.style={font=\sffamily\small}]
		(3) edge[color=blue] node [midway] {} (0);	
	\path[every node/.style={font=\sffamily\small}]
		(-3) edge[color=black] node [midway] {$\bullet$} (0);
	
	\path[every node/.style={font=\sffamily\small}]
		(4) edge[color=blue] node [midway] {} (0);	
	\path[every node/.style={font=\sffamily\small}]
		(-4) edge[color=black] node [midway] {$\bullet$} (0);	
		\path[every node/.style={font=\sffamily\small}]
		(3) edge[color=blue] node [midway] {} (2);	
	\path[every node/.style={font=\sffamily\small}]
		(-3) edge[color=black] node [midway] {} (-2);
	\path[every node/.style={font=\sffamily\small}]
		(2) edge[color=blue] node [midway] {$\bullet$} (0);	
	\path[every node/.style={font=\sffamily\small}]
		(-2) edge[color=black] node [midway] {} (0);
		\path[every node/.style={font=\sffamily\small}]
		(1) edge[color=blue] node [midway] {} (2);	
	\path[every node/.style={font=\sffamily\small}]
		(-1) edge[color=black] node [midway] {} (-2);

\draw[fill=gray!50,opacity=0.3 ] (0,0) -- (45 : 2cm) -- (90:2cm)  -- cycle;
\draw[fill=gray!50, opacity=0.3 ] (0,0) -- (45 : -2cm) -- (90:-2cm)  -- cycle;

\draw[fill=gray!50,opacity=0.3 ] (0,0) -- (90 : 2cm) -- (135:2cm)  -- cycle;
\draw[fill=gray!50, opacity=0.3 ] (0,0) -- (90 : -2cm) -- (135:-2cm)  -- cycle;	

\draw[fill=gray!50,opacity=0.3 ] (0,0) -- (135 : 2cm)--(180:2cm) -- (225:2cm)  -- cycle;
\draw[fill=gray!50, opacity=0.3 ] (0,0) -- (135 : -2cm) -- (180:-2cm) -- (225:-2cm)  -- cycle;

\node[inner sep=0.5, fill=none, label=right:{1}] at (90:1) {};
\node[inner sep=0.5, fill=none, label=right:{1}] at (90:-1) {};	

\node[inner sep=0.5, fill=none, label=right:{2}] at (45:1) {};
\node[inner sep=0.5, fill=none, label=right:{2}] at (45:-1) {};	

\node[inner sep=0.5, fill=none, label=left:{3}] at (112:2) {};
\node[inner sep=0.5, fill=none, label=right:{3}] at (112:-2) {};		

\node[inner sep=0.5, fill=none, label=left:{4}] at (135:1) {};
\node[inner sep=0.5, fill=none, label=right:{4}] at (135:-1) {};	

\node[inner sep=0.5, fill=none, label=left:{5}] at (157:1.8) {};
\node[inner sep=0.5, fill=none, label=right:{5}] at (157:-1.8) {};	
\end{tikzpicture}
\end{minipage}%
\caption{A maximal chain in $NC^D(n)$ represented by an exceptional sequence of curves.}
\label{partitions}
\end{figure}
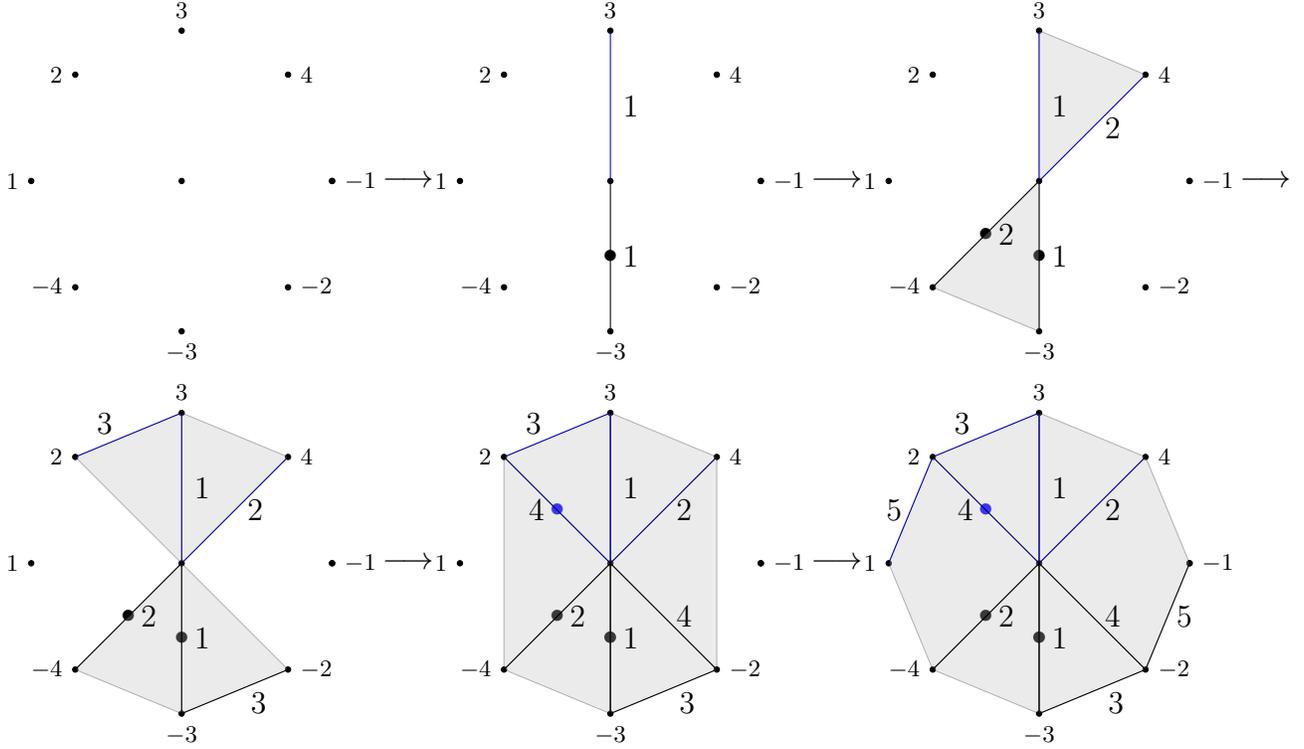


\section{Proof of Proposition~\ref{prop_tau_action}}\label{sec_prop_1_proof}

To prove Proposition~\ref{prop_tau_action}, we use the well-known fact that Auslander--Reiten translation associated may be calculated using \textit{reflection functors}. We define the \textit{reflection functor} \[R^+_k: \text{rep}(Q) \to \text{rep} (Q^\prime)\] when $k \in Q_0$ is a sink and $Q^\prime$ is the quiver obtained from $Q$ by reversing all arrows incident to $k$ as follows. Given $V = ((V_i)_{i \in Q_0}, (f_a)_{a \in Q_1}) \in \text{rep}(Q)$, we set ${R}_k^+(V) := ((V^\prime_i)_{i \in Q^\prime_0}, (f^\prime_a)_{a \in Q^\prime_1}) \in \text{rep}(Q^\prime)$ where
\begin{itemize}
\item $V_i^\prime = V_i$ for $i \neq k$ and $V^\prime_k$ is the kernel of the map $(f_a)_{a: s(a)\to k}: \left(\bigoplus_{a: s(a) \to k} V_{s(a)}  \right)\longrightarrow V_k ,$
\item $f^\prime_a = f_a$ for all arrows $a: i \to j \in Q_1$ with $j \neq k$, and for any arrows $a: i \to k \in Q_1$ the map $f^\prime_a: V_k^\prime \to V^\prime_i = V_i$ is the composition of the inclusion of $V_k^\prime$ into $\bigoplus_{a: s(a) \to k} V_{s(a)}$ with the projection onto the direct summand $V_i$.
\end{itemize}
The reflection functor $R^+_k$ sends the simple representation at vertex $k$ to the zero representation and induces an injective map from the other indecomposable representations of $Q$ to the indecomposable representations of $Q^\prime$. Additionally, if $\textbf{dim}(V) = (d_1, \ldots, d_n)$ with $V \in \text{rep}(Q)$, then \[\textbf{dim}(R^+_k(V)) = \displaystyle (d_1, \ldots, d_{k-1}, -d_k + \sum_{i: i \text{ is incident to } k} d_i, d_{k+1}, \ldots, d_n).\]

For more information about reflection functors we refer the reader to \cite{assem2006elements}.

\begin{proof}[Proof of Proposition~\ref{prop_tau_action}]
First, assume that the representation $V(\gamma)$ defined by $[\gamma] \in \text{Eq}(\Sigma_n)$ has $\tau V(\gamma) = 0$. In other words, $V(\gamma)$ is a projective representation of $Q$. By referring to Example~\ref{proj_and_inj_repns}, we see that $V(\varrho(\gamma)) \simeq \nu V(\gamma)$.

Next, assume that $\tau V(\gamma) \neq 0$. We prove that $V(\varrho(\gamma)) \simeq \nu V(\gamma)$ by considering the three types of representations that appear in Definition~\ref{def:2}. We prove the statement when $V(\gamma)$ is a representation of the form appearing in Definition~\ref{def:2} $b)$. The other two cases are handled similarly.

Recall that the Auslander--Reiten translation $\tau: \text{rep}(Q) \to \text{rep}(Q)$ may be expressed as $\tau = R^+_n\cdots R^+_1$. Let \[R^+_n\cdots R^+_1(V(\gamma)) = ((V^\prime_k)_{k \in Q_0}, (f^\prime_a)_{a \in Q_1}),\] and assume the endpoints of $\gamma$ are $i$ and $j$ where $i$ is its smallest endpoint. By assumption $V(\gamma)$ is not projective so $i \in \{2, \ldots, n-2\}$ and $j \in \{-(i+1), \ldots, -(n-1)\}$. By repeatedly using the above formula for how reflection functors act on dimension vectors, we obtain \[\begin{array}{ccccc}V^\prime_k & := & \left\{\begin{array}{lcl} \mathbb{K}^2 & : & k \in \{-j-1, \ldots, n-2\} \text{ for $-j \neq n-1$} \\ \mathbb{K} & : & k \in \{i-1, i, \ldots, -j-2\}\cup\{n-1,n\} \\ 0 & : & \text{otherwise.} \end{array}\right.\end{array}\] By inspection, we now have \[\textbf{dim}(\tau(V(\gamma))) = \textbf{dim}(V(\varrho(\gamma))).\]

Since $\tau(V(\gamma)) \neq 0$, we also know that $\tau(V(\gamma))$ is indecomposable. The fact that $Q^n$ is a Dynkin quiver implies that its indecomposable representations are determined up to isomorphism by their dimension vectors. Thus $\tau(V(\gamma)) \simeq V(\varrho(\gamma)).$
\end{proof}

\section{Proof of Theorem~\ref{thm1} and Theorem~\ref{thm:2}}\label{sec_thm_1_2_proof}

To prove Theorem~\ref{thm1} and Theorem~\ref{thm:2}, we first need to determine exactly which pairs of indecomposable representations define 0, 1, or 2 exceptional sequences. One of the ingredients we use in this classification is the \textit{Euler form} on $\mathbb{Z}^{Q^n_0} \cong \mathbb{Z}^n$. Given two vectors $\textbf{u} = (u_1, \ldots, u_n)$ and $\textbf{v} = (v_1, \ldots, v_n)$, the Euler form is given by \[\langle \textbf{u}, \textbf{v}\rangle = \displaystyle \sum_{i \in Q^n_0} u_iv_i - \sum_{\alpha \in Q^n_1} u_{s(\alpha)}v_{t(\alpha)}.\] It is well-known that for any representations $U$ and $V$ of an acyclic quiver $Q$, one can express the Euler form as \[\langle \textbf{dim}(U), \textbf{dim}(V)\rangle = \text{dim} \text{Hom}(U,V) - \text{dim} \text{Ext}^1(U,V).\]

\begin{lem}
{Let $U$ and $V$ be any indecomposable representations of $Q^n$. If $\text{Hom}(U, V) \neq 0$, then $\text{Hom}(V, U) = 0$ and $\text{Ext}^1(U, V) = 0.$ Dually, if $\text{Ext}^1(U, V) \neq 0$, then $\text{Ext}^1(V,U) = 0$ and $\text{Hom}(U,V) = 0$. Consequently, the sequence $(U,V)$ is exceptional if and only if $\langle \textbf{dim}(V), \textbf{dim}(U)\rangle = 0.$}
\end{lem}

\begin{proof}
{We prove the first assertion, and the dual assertion is proved similarly. By assumption, there is a sequence of arrows of the Auslander--Reiten quiver of $Q^n$ going from $U$ to $V$. It is well known that the Auslander--Reiten quiver of any Dynkin quiver is acyclic. Therefore, there is no such sequence of arrows going from $V$ to $U$ or from $V$ to $\tau U$. This implies that $\text{Hom}(V,U) = 0$ and $\text{Ext}^1(U,V) \simeq D\text{Hom}(V,\tau U) = 0.$}
\end{proof}


Before presenting our classification, we need one other definition. {Let $\gamma :[0,1]\rightarrow\Sigma_n$ denote a geodesic curve whose endpoints are $i$ and $j$ where $1\le i<n$. The curve $\gamma$ induces a partition of the set of vertices of $\Sigma_n$ into disjoint sets of vertices $R,S,\{i,j\}$  (see Figure \ref{fig:support_fig2}). 
Without loss of generality, assume that $|S|<|R|$ . The \textit{support} of $\gamma$, denoted supp$(\gamma)$, is defined to be $S\sqcup\{i,j\}$. If $i=\pm n$, then the support of $\gamma$ is defined as $\text{supp}(\gamma)=\{i,i+1,\ldots,n-1,\pm n\}$.}

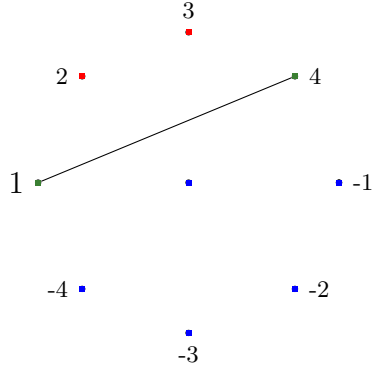
\begin{figure}
\centering
\begin{tikzpicture}
\draw [fill] circle (1pt);
	
	\foreach \a in {0,45,...,360} {
		\draw[fill] (\a:2cm) circle (1pt); 
		}
		
\node[inner sep=1, fill=OliveGreen, label=right:{\scriptsize{4}}] (b) at (45:2) {};
	\node[inner sep=1, fill=red, label=above:{\scriptsize{3}}] (c) at (90:2) {};
	\node[inner sep=1, fill=red, label=left:{\scriptsize{2}}] (d) at (135:2) {};
	\node[inner sep=1, fill=OliveGreen, label=left:{$\scriptsize{1}$}] (g) at (180:2) {};
	\node[inner sep=1, fill=blue, label=left:{\scriptsize{-4}}] (h) at (45:-2) {};
	\node[inner sep=1, fill=blue, label=below:{\scriptsize{-3}}] (i) at (90:-2) {};
	\node[inner sep=1, fill=blue, label=right:{\scriptsize{-2}}] (j) at (135:-2){};
	\node[inner sep=1, fill=blue, label=right:{\scriptsize{-1}}] (m) at (360:2) {};

	\node[inner sep=1, fill=blue, label=right:{}] (o) at (0:0) {};
\path[every node/.style={font=\sffamily\small}]
		(g) edge[color=black] node [midway] {} (b);
				
\end{tikzpicture}
\caption{Here, $\Sigma_n$ is partitioned into three sets: $\{i,j\}=\{1,4\},$ $S = \{2,3\}$, and $R=\{-1,-2,-3,-4,\pm5\}$. So $\text{supp}(\gamma) = \{1,2,3,4\}.$}
\label{fig:support_fig2}
\end{figure}

\begin{lem}\label{common_endpt_lemma}
Let $[\gamma], [\delta]$ be equivalence classes in $\text{Eq}(\Sigma_n)$ that share an endpoint $\ell$ (and therefore also share $-\ell$) and do not form a bad pair. Assume $\ell \in \{\pm 1, \ldots, \pm(n-1)\}$ or the essential curves of $[\gamma]$ and $[\delta]$ share $\ell = \pm n$. Then $(V(\gamma), V(\delta))$ is an exceptional sequence and $(V(\delta), V(\gamma))$ is not if and only if $[\delta]$ is clockwise from $[\gamma]$. On the other hand, assume that $[\gamma]$ and $[\delta]$ share $\ell = \pm n$, but $\ell$ is not shared by the essential curves. Then $(V(\gamma), V(\delta))$ is an exceptional sequence and $(V(\delta), V(\gamma))$ is not if and only if $[\delta]$ is counterclockwise from $[\gamma]$.
\end{lem}
\begin{proof}
First, assume that $\ell \in \{1, \ldots, n-1\}$, that $\ell$ is an endpoint of $\gamma$ and $\delta$, and that $\text{supp}(\gamma)\cap \text{supp}(\delta) = \{\ell\}$. Then the equivalence classes $[\gamma]$ and $[\delta]$ appear in one of the configurations in Figure~\ref{fig:two_cases}.

\begin{figure}[htb!]
\[\begin{array}{cccccc}
\begin{tikzpicture}
\draw [fill] circle (1pt);
	
	\foreach \a in {0,30,...,360} {
		\draw[fill] (\a:2cm) circle (1pt); 
		}
		
	\node[inner sep=0.5, fill=none, label=above:{$\scriptsize{-i}$}] (b) at (30:2) {};
	\node[inner sep=0.5, fill=none, label=above:{$\scriptsize{j}$}] (c) at (60:2) {};
	\node[inner sep=0.5, fill=none, label=below:{}] (d) at (90:2) {};
	\node[inner sep=0.5, fill=none, label=above:{$\scriptsize{\ell}$}] (e) at (120:2) {};
	\node[inner sep=0.5, fill=none, label=right:{}] (f) at (150:2) {};
	\node[inner sep=0.5, fill=none, label=left:{$\scriptsize{1}$}] (g) at (180:2) {};
	\node[inner sep=0.5, fill=none, label=left:{$\scriptsize{i}$}] (h) at (210:2) {};
	\node[inner sep=0.5, fill=none, label=below:{$\scriptsize{-j}$}] (i) at (240:2) {};
	\node[inner sep=0.5, fill=none, label=right:{}] (j) at (270:2){};
	\node[inner sep=0.5, fill=none, label=below:{$\scriptsize{-\ell}$}] (k) at (300:2) {};
	\node[inner sep=0.5, fill=none, label=right:{}] (l) at (330:2) {};
	\node[inner sep=0.5, fill=none, label=right:{}] (m) at (360:2) {};
\path[every node/.style={font=\sffamily\small}]
		(e) edge[color=blue] node [midway] {} (h);
\path[every node/.style={font=\sffamily\small}]
		(e) edge[color=black] node [midway] {} (c);	
\path[every node/.style={font=\sffamily\small}]
		(k) edge[color=black] node [midway] {} (i);
\path[every node/.style={font=\sffamily\small}]
		(k) edge[color=blue] node [midway] {} (b);
			
	\node[inner sep=0.5, fill=none, label=center:{$	\scriptsize{\delta}$}] (a) at (165:1.2) {};
	\node[inner sep=0.5, fill=none, label=center:{$	\scriptsize{\gamma}$}] (a) at (90:1.4) {};			
	
\end{tikzpicture}
& &
\begin{tikzpicture}
\draw [fill] circle (1pt);
	
	\foreach \a in {0,30,...,360} {
		\draw[fill] (\a:2cm) circle (1pt); 
		}
		
	\node[inner sep=0.5, fill=none, label=above:{$\scriptsize{-i}$}] (b) at (30:2) {};
	\node[inner sep=0.5, fill=none, label=above:{}] (c) at (60:2) {};
	\node[inner sep=0.5, fill=none, label=below:{}] (d) at (90:2) {};
	\node[inner sep=0.5, fill=none, label=above:{$\scriptsize{\ell}$}] (e) at (120:2) {};
	\node[inner sep=0.5, fill=none, label=right:{}] (f) at (150:2) {};
	\node[inner sep=0.5, fill=none, label=left:{$\scriptsize{1}$}] (g) at (180:2) {};
	\node[inner sep=0.5, fill=none, label=left:{$\scriptsize{i}$}] (h) at (210:2) {};
	\node[inner sep=0.5, fill=none, label=below:{}] (i) at (240:2) {};
	\node[inner sep=0.5, fill=none, label=right:{}] (j) at (270:2){};
	\node[inner sep=0.5, fill=none, label=below:{$\scriptsize{-\ell}$}] (k) at (300:2) {};
	\node[inner sep=0.5, fill=none, label=right:{}] (l) at (330:2) {};
	\node[inner sep=0.5, fill=none, label=right:{}] (m) at (360:2) {};
	\node[inner sep=0.5, fill=none, label=right:{}] (o) at (0:0) {};
\path[every node/.style={font=\sffamily\small}]
		(e) edge[color=blue] node [midway] {} (h);
\path[every node/.style={font=\sffamily\small}]
		(e) edge[color=black] node [midway] {} (o);	
\path[every node/.style={font=\sffamily\small}]
		(k) edge[color=black] node [midway] {$\bullet$} (o);
\path[every node/.style={font=\sffamily\small}]
		(k) edge[color=blue] node [midway] {} (b);

	\node[inner sep=0.5, fill=none, label=center:{$	\scriptsize{\delta}$}] (a) at (165:1.2) {};
	\node[inner sep=0.5, fill=none, label=center:{$	\scriptsize{\gamma}$}] (a) at (100:.8) {};
				
\end{tikzpicture}
\\ (a) & & (b) \\
\\
\begin{tikzpicture}
\draw [fill] circle (1pt);
	
	\foreach \a in {0,30,...,360} {
		\draw[fill] (\a:2cm) circle (1pt); 
		}
		
	\node[inner sep=0.5, fill=none, label=right:{$\scriptsize{n-1}$}] (b) at (30:2) {};
	\node[inner sep=0.5, fill=none, label=above:{}] (c) at (60:2) {};
	\node[inner sep=0.5, fill=none, label=above:{$\scriptsize{\ell}$}] (d) at (90:2) {};
	\node[inner sep=0.5, fill=none, label=left:{$\scriptsize{i}$}] (e) at (120:2) {};
	\node[inner sep=0.5, fill=none, label=left:{$\scriptsize{-j}$}] (f) at (150:2) {};
	\node[inner sep=0.5, fill=none, label=left:{$\scriptsize{1}$}] (g) at (180:2) {};
	\node[inner sep=0.5, fill=none, label=left:{}] (h) at (210:2) {};
	\node[inner sep=0.5, fill=none, label=below:{}] (i) at (240:2) {};
	\node[inner sep=0.5, fill=none, label=below:{$\scriptsize{-\ell}$}] (j) at (270:2){};
	\node[inner sep=0.5, fill=none, label=below:{$\scriptsize{-i}$}] (k) at (300:2) {};
	\node[inner sep=0.5, fill=none, label=right:{$\scriptsize{j}$}] (l) at (330:2) {};
	\node[inner sep=0.5, fill=none, label=right:{}] (m) at (360:2) {};
	
\path[every node/.style={font=\sffamily\small}]
		(d) edge[color=blue] node [midway] {} (e);
\path[every node/.style={font=\sffamily\small}]
		(d) edge[color=black] node [midway] {} (l);	
\path[every node/.style={font=\sffamily\small}]
		(j) edge[color=black] node [midway] {} (f);
\path[every node/.style={font=\sffamily\small}]
		(j) edge[color=blue] node [midway] {} (k);	
\path[every node/.style={font=\sffamily\small}]
		(e) edge[color=OliveGreen, dashed] node [midway] {} (l);	
\path[every node/.style={font=\sffamily\small}]
		(f) edge[color=OliveGreen, dashed] node [midway] {} (k);	

	\node[inner sep=0.5, fill=none, label=center:{$	\scriptsize{\delta}$}] (a) at (105:2.2) {};
	\node[inner sep=0.5, fill=none, label=right:{$	\scriptsize{\gamma}$}] (a) at (30:1) {};		
	\node[inner sep=0.5, fill=none, label=center:{$	\scriptsize{\nu}$}] (a) at (120:1) {};
	
\end{tikzpicture} & & 
\begin{tikzpicture}
\draw [fill] circle (1pt);
	
	\foreach \a in {0,30,...,360} {
		\draw[fill] (\a:2cm) circle (1pt); 
		}
		
	\node[inner sep=0.5, fill=none, label=right:{$\scriptsize{n-1}$}] (b) at (30:2) {};
	\node[inner sep=0.5, fill=none, label=above:{}] (c) at (60:2) {};
	\node[inner sep=0.5, fill=none, label=above:{$\scriptsize{\ell}$}] (d) at (90:2) {};
	\node[inner sep=0.5, fill=none, label=left:{$\scriptsize{i}$}] (e) at (120:2) {};
	\node[inner sep=0.5, fill=none, label=left:{}] (f) at (150:2) {};
	\node[inner sep=0.5, fill=none, label=left:{$\scriptsize{1}$}] (g) at (180:2) {};
	\node[inner sep=0.5, fill=none, label=left:{}] (h) at (210:2) {};
	\node[inner sep=0.5, fill=none, label=below:{}] (i) at (240:2) {};
	\node[inner sep=0.5, fill=none, label=below:{$\scriptsize{-\ell}$}] (j) at (270:2){};
	\node[inner sep=0.5, fill=none, label=below:{$\scriptsize{-i}$}] (k) at (300:2) {};
	\node[inner sep=0.5, fill=none, label=right:{}] (l) at (330:2) {};
	\node[inner sep=0.5, fill=none, label=right:{}] (m) at (360:2) {};
	\node[inner sep=0.5, fill=none, label=right:{}] (n) at (0:0) {};
	
\path[every node/.style={font=\sffamily\small}]
		(d) edge[color=blue] node [midway] {} (e);
\path[every node/.style={font=\sffamily\small}]
		(d) edge[color=black]  node [midway] {$\bullet$} (n);	
\path[every node/.style={font=\sffamily\small}]
		(j) edge[color=black] node [midway] {} (n);
\path[every node/.style={font=\sffamily\small}]
		(j) edge[color=blue] node [midway] {} (k);	
\path[every node/.style={font=\sffamily\small}]
		(e) edge[color=OliveGreen, dashed] node [midway] {$\bullet$} (n);	
\path[every node/.style={font=\sffamily\small}]
		(k) edge[color=OliveGreen, dashed] node [midway] {} (n);

	\node[inner sep=0.5, fill=none, label=center:{$	\scriptsize{\delta}$}] (a) at (105:2.25) {};
	\node[inner sep=0.5, fill=none, label=right:{$	\scriptsize{\gamma}$}] (a) at (90:1) {};	
	\node[inner sep=0.5, fill=none, label=left:{$	\scriptsize{\nu}$}] (a) at (120:1) {};	
\end{tikzpicture}\\ 
(c) & & (d)
\end{array}\]
\caption{The possible configurations formed by $[\gamma]$ and $[\delta]$ where $\text{supp}(\gamma)\cap \text{supp}(\delta) = \{\ell\}$ and $\ell \in \{1, \ldots, n-1\}$. In $(a)$, $-(n-1)\le i<-j$ and $\ell < j < -i$. In $(b)$, $-(n-1)\le i<-\ell$ and $j = \pm n$; the configuration is considered up to applying $\overline{(\cdot)}$ to $[\gamma]$ and $[\delta]$. In $(c)$, $1 \le i < \ell$ and either $-1 \le j < -i$ or $\ell < j \le n-1$. In $(d)$, $1 \le i < \ell$ and $j = \pm n$; the configuration is considered up to applying $\overline{(\cdot)}$ to $[\gamma]$, $[\delta]$, and $[\nu]$.}
\label{fig:two_cases}
\end{figure}
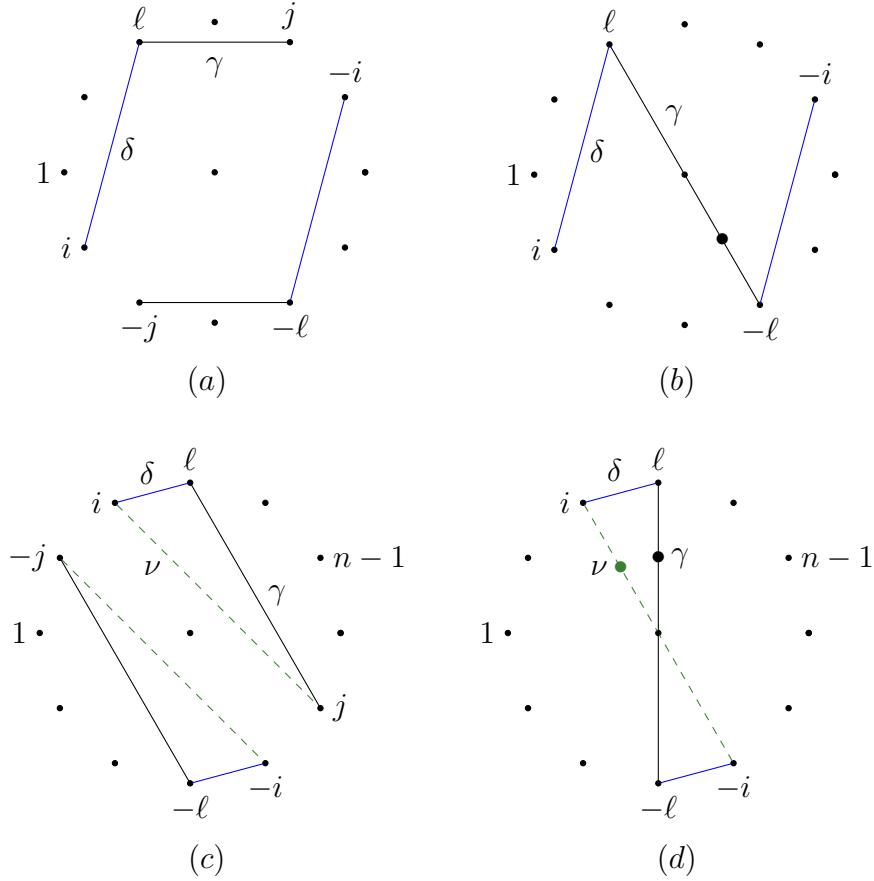

Suppose these equivalence classes appear in a configuration appearing in Figure~\ref{fig:two_cases}$(a)$ or Figure~\ref{fig:two_cases}$(b)$. Then we have that \[\text{Hom}(V(\gamma), V(\delta)) \neq 0,\] which implies that $(V(\delta), V(\gamma))$ is not an exceptional sequence. If $[\gamma]$ and $[\delta]$ appear in a configuration as in Figure~\ref{fig:two_cases}$(a)$, then we have that \[\begin{array}{cclccc} \langle \textbf{dim}(V(\delta)), \textbf{dim}(V(\gamma))\rangle & = & \dim(V(\gamma)) - \dim(V(\gamma))\\ & = & 0. \end{array}\] 
On the other hand, if $[\gamma]$ and $[\delta]$ appear in a configuration as in Figure~\ref{fig:two_cases}$(b)$, then we have that \[\begin{array}{cclccc} \langle \textbf{dim}(V(\delta)), \textbf{dim}(V(\gamma))\rangle & = & (\dim(V(\delta)) - 1) - (\dim(V(\delta)) - 1)\\ & = & 0. \end{array}\] Therefore, $(V(\gamma), V(\delta))$ is an exceptional sequence.

Now, suppose $[\gamma]$ and $[\delta]$ form a configuration appearing in Figure~\ref{fig:two_cases}$(c)$ or Figure~\ref{fig:two_cases}$(d)$. Observe that \[0 \to V(\delta) \to V(\nu) \to V(\gamma) \to 0\] is a nonsplit extension, where $[\nu]$ is the equivalence class of curves appearing in Figure~\ref{fig:two_cases}$(c)$ and Figure~\ref{fig:two_cases}$(d)$. Thus, $\text{Ext}^1(V(\gamma), V(\delta)) \neq 0$ so $(V(\delta), V(\gamma))$ is not an exceptional sequence. We also have that
\[\langle\textbf{dim}(V(\delta)),\textbf{dim}(V(\gamma))\rangle = \text{dim}(V(\delta))-\text{dim}(V(\delta))=0
\]
if $j\in \{-1,\ldots,-i+1\}$, and
\[\langle\textbf{dim}(V(\delta)),\textbf{dim}(V(\gamma))\rangle=0-0=0\]
if $j\in \{\ell +1,\ldots,n\}\cup\{-n\}$. Therefore, $(V(\gamma),V(\delta))$ is an exceptional sequence.

Next, assume that $\ell \in \{1, \ldots, n-1\}$, that $\ell$ is an endpoint of $\gamma$ and $\delta$, and that $|\text{supp}(\gamma)\cap \text{supp}(\delta)| \ge 2$. Then the equivalence classes $[\gamma]$ and $[\delta]$ appear in a configuration of the form shown in Figure~\ref{fig:four_cases_n}$(a)$, $(b)$, or $(c)$.

\begin{figure}[htb!]
\begin{subfigure}
\centering
\hfill

\begin{tikzpicture}
\draw [fill] circle (1pt);
	
	\foreach \a in {0,30,...,360} {
		\draw[fill] (\a:2cm) circle (1pt); 
		}
		
	\node[inner sep=0.5, fill=none, label=right:{$-i$}] (b) at (30:2) {};
	\node[inner sep=0.5, fill=none, label=above:{$-j$}] (c) at (60:2) {};
	\node[inner sep=0.5, fill=none, label=above:{$\scriptsize{\ell}$}] (d) at (90:2) {};
	\node[inner sep=0.5, fill=none, label=left:{}] (e) at (120:2) {};
	\node[inner sep=0.5, fill=none, label=left:{}] (f) at (150:2) {};
	\node[inner sep=0.5, fill=none, label=left:{$i$}] (h) at (210:2) {};
	\node[inner sep=0.5, fill=none, label=below:{$j$}] (i) at (240:2) {};
	\node[inner sep=0.5, fill=none, label=below:{$\scriptsize{-\ell}$}] (j) at (270:2){};
	\node[inner sep=0.5, fill=none, label=below:{\textnormal{(}a\textnormal{)}}] at (270:2.5){};
	\node[inner sep=0.5, fill=none, label=below:{}] (k) at (300:2) {};
	\node[inner sep=0.5, fill=none, label=right:{}] (l) at (330:2) {};
	\node[inner sep=0.5, fill=none, label=right:{}] (m) at (360:2) {};
	
\path[every node/.style={font=\sffamily\small}]
		(d) edge[color=blue] node [midway] {} (h);
\path[every node/.style={font=\sffamily\small}]
		(d) edge[color=black] node [midway] {} (i);	
\path[every node/.style={font=\sffamily\small}]
		(j) edge[color=black] node [midway] {} (c);
\path[every node/.style={font=\sffamily\small}]
		(j) edge[color=blue] node [midway] {} (b);	
		
	\path[every node/.style={font=\sffamily\small}]
		(h) edge[color=OliveGreen, dashed] node [midway] {} (i);
	\path[every node/.style={font=\sffamily\small}]
		(c) edge[color=OliveGreen, dashed] node [midway] {} (b);

	\node[inner sep=0.5, fill=none, label=right:{$	\scriptsize{\nu}$}] at (50:1.5) {};

	\node[inner sep=0.5, fill=none, label=center:{$	\scriptsize{\delta}$}] (a) at (150:1.2) {};
	\node[inner sep=0.5, fill=none, label=right:{$	\scriptsize{\gamma}$}] (a) at (180:1) {};		
	
\end{tikzpicture}
\end{subfigure}\hspace{10mm}
\begin{subfigure}
\centering
\begin{tikzpicture}
\draw [fill] circle (1pt);
	
	\foreach \a in {0,30,...,360} {
		\draw[fill] (\a:2cm) circle (1pt); 
		}
\node[inner sep=0.5, fill=none, label=below:{\textnormal{(}b\textnormal{)}}] at (270:2.5){};		
	\node[inner sep=0.5, fill=none, label=right:{$j$}] (b) at (30:2) {};
	\node[inner sep=0.5, fill=none, label=above:{}] (c) at (60:2) {};
	\node[inner sep=0.5, fill=none, label=above:{$\scriptsize{\ell}$}] (d) at (90:2) {};
	\node[inner sep=0.5, fill=none, label=left:{}] (e) at (120:2) {};
	\node[inner sep=0.5, fill=none, label=left:{$-i$}] (f) at (150:2) {};
	\node[inner sep=0.5, fill=none, label=left:{$-j$}] (h) at (210:2) {};
	\node[inner sep=0.5, fill=none, label=below:{}] (i) at (240:2) {};
	\node[inner sep=0.5, fill=none, label=below:{$\scriptsize{-\ell}$}] (j) at (270:2){};
	\node[inner sep=0.5, fill=none, label=below:{}] (k) at (300:2) {};
	\node[inner sep=0.5, fill=none, label=right:{$i$}] (l) at (330:2) {};
	\node[inner sep=0.5, fill=none, label=right:{}] (m) at (360:2) {};
	\node[inner sep=0.5, fill=none, label=right:{}] (n) at (0:0) {};
	
\path[every node/.style={font=\sffamily\small}]
		(d) edge[color=black] node [midway] {} (b);
\path[every node/.style={font=\sffamily\small}]
		(d) edge[color=blue]  node [midway] {} (l);	
\path[every node/.style={font=\sffamily\small}]
		(j) edge[color=blue] node [midway] {} (f);
\path[every node/.style={font=\sffamily\small}]
		(j) edge[color=black] node [midway] {} (h);

	\node[inner sep=0.5, fill=none, label=center:{$	\scriptsize{\gamma}$}] (a) at (60:1.5) {};
	\node[inner sep=0.5, fill=none, label=right:{$	\scriptsize{\delta}$}] (a) at (90:1) {};	
\end{tikzpicture}
\end{subfigure}
\vfill
\begin{subfigure}
\centering
\hfill

\begin{tikzpicture}
\draw [fill] circle (1pt);
	
	\foreach \a in {0,30,...,360} {
		\draw[fill] (\a:2cm) circle (1pt); 
		}
\node[inner sep=0.5, fill=none, label=below:{\textnormal{(}c\textnormal{)}}] at (270:2.5){};		
	\node[inner sep=0.5, fill=none, label=right:{}] (b) at (30:2) {};
	\node[inner sep=0.5, fill=none, label=above:{}] (c) at (60:2) {};
	\node[inner sep=0.5, fill=none, label=above:{$\scriptsize{\ell}$}] (d) at (90:2) {};
	\node[inner sep=0.5, fill=none, label=left:{}] (e) at (120:2) {};
	\node[inner sep=0.5, fill=none, label=left:{$-j$}] (f) at (150:2) {};
	\node[inner sep=0.5, fill=none, label=left:{}] (h) at (210:2) {};
	\node[inner sep=0.5, fill=none, label=below:{}] (i) at (240:2) {};
	\node[inner sep=0.5, fill=none, label=below:{$\scriptsize{-\ell}$}] (j) at (270:2){};
	\node[inner sep=0.5, fill=none, label=below:{}] (k) at (300:2) {};
	\node[inner sep=0.5, fill=none, label=right:{$j$}] (l) at (330:2) {};
	\node[inner sep=0.5, fill=none, label=right:{}] (m) at (360:2) {};
	\node[inner sep=0.5, fill=none, label=right:{}] (n) at (0:0) {};
	
\path[every node/.style={font=\sffamily\small}]
		(d) edge[color=blue] node [midway] {$\bullet$} (n);
\path[every node/.style={font=\sffamily\small}]
		(d) edge[color=black]  node [midway] {} (l);	
\path[every node/.style={font=\sffamily\small}]
		(j) edge[color=black] node [midway] {} (f);
\path[every node/.style={font=\sffamily\small}]
		(j) edge[color=blue] node [midway] {} (n);

	\node[inner sep=0.5, fill=none, label=center:{$	\scriptsize{\gamma}$}] (a) at (60:1.5) {};
	\node[inner sep=0.5, fill=none, label=right:{$	\scriptsize{\delta}$}] (a) at (90:1) {};
	
\end{tikzpicture}
\end{subfigure}\hspace{10mm}
\begin{subfigure}
\centering
\begin{tikzpicture}
\draw [fill] circle (1pt);
	
	\foreach \a in {0,30,...,360} {
		\draw[fill] (\a:2cm) circle (1pt); 
		}
	\node[inner sep=0.5, fill=none, label=below:{\textnormal{(}d\textnormal{)}}] at (270:2.5){};	
	\node[inner sep=0.5, fill=none, label=right:{$i$}] (b) at (30:2) {};
	\node[inner sep=0.5, fill=none, label=above:{}] (c) at (60:2) {};
	\node[inner sep=0.5, fill=none, label=above:{$j$}] (d) at (90:2) {};
	\node[inner sep=0.5, fill=none, label=left:{}] (e) at (120:2) {};
	\node[inner sep=0.5, fill=none, label=left:{}] (f) at (150:2) {};
	\node[inner sep=0.5, fill=none, label=left:{$-i$}] (h) at (210:2) {};
	\node[inner sep=0.5, fill=none, label=below:{}] (i) at (240:2) {};
	\node[inner sep=0.5, fill=none, label=below:{$-j$}] (j) at (270:2){};
	\node[inner sep=0.5, fill=none, label=below:{}] (k) at (300:2) {};
	\node[inner sep=0.5, fill=none, label=right:{}] (l) at (330:2) {};
	\node[inner sep=0.5, fill=none, label=right:{}] (m) at (360:2) {};
	\node[inner sep=0.5, fill=none, label=right:{}] (n) at (0:0) {};
	
\path[every node/.style={font=\sffamily\small}]
		(b) edge[color=blue] node [midway] {$\bullet$} (n);
\path[every node/.style={font=\sffamily\small}]
		(d) edge[color=black]  node [midway] {$\bullet$} (n);	
\path[every node/.style={font=\sffamily\small}]
		(j) edge[color=black] node [midway] {} (n);
\path[every node/.style={font=\sffamily\small}]
		(h) edge[color=blue] node [midway] {} (n);

	\node[inner sep=0.5, fill=none, label=below:{$	\scriptsize{\delta}$}] (a) at (30:1) {};
	\node[inner sep=0.5, fill=none, label=right:{$	\scriptsize{\gamma}$}] (a) at (90:1) {};

\end{tikzpicture}
\end{subfigure}
\vfill
\begin{subfigure}
\centering
\begin{tikzpicture}
	\draw [fill] circle (1pt);
	
	\foreach \a in {0,30,...,360} {
		\draw[fill] (\a:2cm) circle (1pt); 
		}
	\node[inner sep=0.5, fill=none, label=below:{\textnormal{(}e\textnormal{)}}] at (270:2.5){};	
	\node[inner sep=0.5, fill=none, label=right:{$-j$}] (b) at (30:2) {};
	\node[inner sep=0.5, fill=none, label=above:{}] (c) at (60:2) {};
	\node[inner sep=0.5, fill=none, label=above:{$i$}] (d) at (90:2) {};
	\node[inner sep=0.5, fill=none, label=left:{}] (e) at (120:2) {};
	\node[inner sep=0.5, fill=none, label=left:{}] (f) at (150:2) {};
	\node[inner sep=0.5, fill=none, label=left:{$j$}] (h) at (210:2) {};
	\node[inner sep=0.5, fill=none, label=below:{}] (i) at (240:2) {};
	\node[inner sep=0.5, fill=none, label=below:{$-i$}] (j) at (270:2){};
	\node[inner sep=0.5, fill=none, label=below:{}] (k) at (300:2) {};
	\node[inner sep=0.5, fill=none, label=right:{}] (l) at (330:2) {};
	\node[inner sep=0.5, fill=none, label=right:{}] (m) at (360:2) {};
	\node[inner sep=0.5, fill=none, label=right:{}] (n) at (0:0) {};
	
\path[every node/.style={font=\sffamily\small}]
		(b) edge[color=black] node [midway] {} (n);
\path[every node/.style={font=\sffamily\small}]
		(d) edge[color=blue]  node [midway] {$\bullet$} (n);	
\path[every node/.style={font=\sffamily\small}]
		(j) edge[color=blue] node [midway] {} (n);
\path[every node/.style={font=\sffamily\small}]
		(h) edge[color=black] node [midway] {$\bullet$} (n);	
		
	\path[every node/.style={font=\sffamily\small}]
		(h) edge[color=OliveGreen, dashed] node [midway] {} (d);
	\path[every node/.style={font=\sffamily\small}]
		(j) edge[color=OliveGreen, dashed] node [midway] {} (b);

	\node[inner sep=0.5, fill=none, label=right:{$	\scriptsize{\nu}$}] at (150:1.5) {};

	\node[inner sep=0.5, fill=none, label=below:{$	\scriptsize{\gamma}$}] (a) at (30:-1) {};

	\node[inner sep=0.5, fill=none, label=below:{$	\scriptsize{\gamma^\prime}$}] (a) at (30:1) {};
	\node[inner sep=0.5, fill=none, label=right:{$	\scriptsize{\delta}$}] (a) at (90:1) {};
\end{tikzpicture}
\end{subfigure}
\caption{In the case where $\text{supp}(\gamma)\cap \text{supp}(\delta)\neq\emptyset$, there are five possible configurations of curves. Here, $i,j$ vary and may be positive or negative in $(a)$, $(b)$, and $(c)$. In $(d)$, $1 \le j < i \le n-1$. In $(e)$, $1 \le i \le n-1$ and $-(n-1) \le j < -i$. }
\label{fig:four_cases_n}
\end{figure}
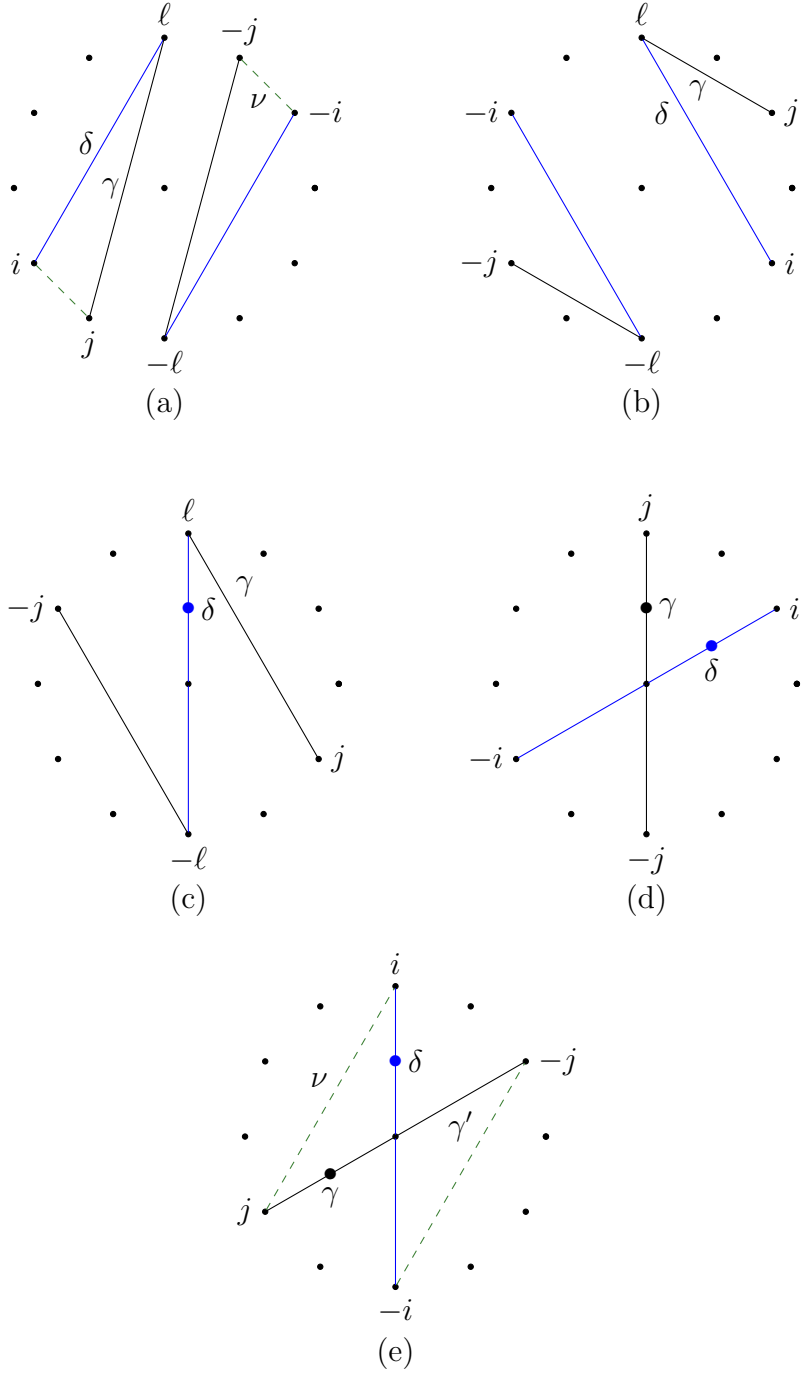

Assume that $[\gamma]$ and $[\delta]$ appear in a configuration of the form shown in Figure~\ref{fig:four_cases_n}$(a)$. If $1\le j < i < \ell$ or $-(n-1) \le i < j < -\ell$, then we have that $\text{Hom}(V(\gamma), V(\delta)) \neq 0$. In the former case, 
\[\langle\textbf{dim}(V(\delta)),\textbf{dim}(V(\gamma))\rangle = \text{dim}(V(\delta)) - \text{dim}(V(\delta)) = 0.\] In the latter case, 
\[\begin{array}{lll} 
\langle\textbf{dim}(V(\delta)),\textbf{dim}(V(\gamma))\rangle & = & \text{dim}(V(\gamma)) + 2|\{i\mid \text{dim}(V(\delta)_i) = \text{dim}(V(\gamma)_i) = 2\}| \\ & & -\text{dim}(V(\gamma)) - 2|\{a\mid \text{dim}(V(\delta)_{s(a)}) = \text{dim}(V(\gamma)_{t(a)}) = 2\}| \\ & = & 0.
\end{array}\]


If $1 \le i < \ell$ and $-(n-1) \le j < -\ell$, then there is a nonsplit extension \[0 \to V(\delta) \to V(\nu) \to V(\gamma) \to 0\] where $[\nu]$ is the equivalence class of curves appearing Figure~\ref{fig:four_cases_n}$(a)$. Thus $\text{Ext}^1(V(\gamma), V(\delta)) \neq 0$. Also, $\langle \textbf{dim}(V(\delta)),\textbf{dim}(V(\gamma))\rangle = 0$. Consequently, $(V(\delta), V(\gamma))$ is not an exceptional sequence, and $(V(\gamma), V(\delta))$ is an exceptional sequence.

Now, assume that $[\gamma]$ and $[\delta]$ appear in a configuration of the form shown in Figure~\ref{fig:four_cases_n}$(b)$. For each $i,j$, we have $\text{Hom}(V(\gamma), V(\delta)) \neq 0$. If $\ell < j < i \le n-1$, then $\langle \textbf{dim}(V(\delta)),\textbf{dim}(V(\gamma))\rangle = \text{dim}(V(\gamma)) - \text{dim}(V(\gamma)) = 0$. If $\ell < j \le n-1$ and $-\ell < i \le -1$, then \[\langle \textbf{dim}(V(\delta)),\textbf{dim}(V(\gamma))\rangle = 2\text{dim}(V(\gamma)) - 2\text{dim}(V(\gamma)) = 0.\] If $-\ell < i < j \le -1$, then 
\[\begin{array}{lll}
\langle \textbf{dim}(V(\delta)),\textbf{dim}(V(\gamma))\rangle & = & \text{dim}(V(\delta)) + 2|\{i \mid \text{dim}(V(\delta)_i) = 2\}| \\ & & -\text{dim}(V(\delta)) -2|\{a \mid \text{dim}(V(\delta)_{s(a)}) = 2\}|\\ & = & 0. 
\end{array}\] Thus, $(V(\delta), V(\gamma))$ is not an exceptional sequence, and $(V(\gamma), V(\delta))$ is an exceptional sequence.


Next, assume that $[\gamma]$ and $[\delta]$ appear in a configuration of the form shown in Figure~\ref{fig:four_cases_n}$(c).$ Here $i = \pm n$, and the configuration in Figure~\ref{fig:four_cases_n}$(c)$ is considered up to applying $\overline{(\cdot)}$ to $[\gamma]$ and $[\delta]$. For each $j$, we have $\text{Hom}(V(\gamma), V(\delta)) \neq 0.$ If $\ell < j \le n-1$, then \[\langle \textbf{dim}(V(\delta)), \textbf{dim}(V(\gamma))\rangle = \text{dim}(V(\gamma)) - \text{dim}(V(\gamma)) = 0.\] If $-\ell < j \le -1$, then 
\[\begin{array}{lllllll}\langle \textbf{dim}(V(\delta)), \textbf{dim}(V(\gamma)) \rangle & = & \text{dim}(V(\delta)) + |\{i \mid \text{dim}(V(\gamma)_i) = 2\}| \\ & & -\text{dim}(V(\delta)) -|\{a \mid \text{dim}(V(\gamma)_{t(a)})=2 \text{ and } \text{dim}(V(\delta)_{s(a)})\neq 0\}| \\ & = & 0.
\end{array}\]
We conclude that $(V(\delta), V(\gamma))$ is not an exceptional sequence, and $(V(\gamma), V(\delta))$ is an exceptional sequence.


The final case we consider is when $\gamma$ and $\delta$ share $\ell = n$ or $\ell = -n$. We complete the proof under the assumption that $\ell = -n$, and the relevant configurations appear in Figure~\ref{fig:four_cases_n}$(d)$ and $(e)$. If $1 \le j < i \le n-1$, as in Figure~\ref{fig:four_cases_n}$(d)$, then $\text{Hom}(V(\gamma), V(\delta)) \neq 0$ and \[\langle \textbf{dim}(V(\delta)), \textbf{dim}(V(\gamma)) \rangle = \text{dim}(V(\delta)) - \text{dim}(V(\delta)) = 0.\]


On the other hand, if $1 \le i \le n-1$ and $-(n-1) \le j < -i$, as in Figure~\ref{fig:four_cases_n}$(e)$, then there exists the nonsplit extension  $0 \to V(\delta) \to V(\nu) \to V(\gamma) \to 0$. Therefore, $\text{Ext}^1(V(\gamma), V(\delta) \neq 0$. We also have that \[\langle \textbf{dim}(V(\delta)), \textbf{dim}(V(\gamma))\rangle = (\text{dim}(V(\gamma)) - 1) - (\text{dim}(V(\gamma)) - 1) = 0.\] We conclude that $(V(\delta), V(\gamma))$ is not an exceptional sequence and $(V(\gamma), V(\delta))$ is an exceptional sequence.
\end{proof}

\begin{lem}\label{bad_pair_classification}
Let $[\gamma]$ and $[\delta]$ be two equivalence classes of curves that form a bad pair. Then neither $(V(\gamma),V(\delta))$ nor $(V(\delta),V(\gamma))$ are exceptional sequences. 
\end{lem}

\begin{proof}
Assume that $\gamma$ and $\delta$ are the essential curves in the equivalence classes $[\gamma]$ and $[\delta]$, respectively. We prove the result in each of the following cases:
\begin{enumerate}
\item\label{bad1} both endpoints of the curves $\gamma$ and $\delta$ are in $\{1, \ldots, n\}$ or the curve $\gamma$ has both its endpoints in $\{1,\ldots, n-1\}$ and the curve $\delta$ has one endpoint in $\{1,\ldots, n-1\}$ and one endpoint is $-n$;
\item\label{bad2} both endpoints of the curve $\gamma$ are in $\{1,\ldots, n\}$, one endpoint of $\delta$ is in $\{1, \ldots, n-1\}$, and the other is in $\{-1,\ldots, -(n-1)\}$;
\item\label{bad3} the curves $\gamma$ and $\delta$ each have one endpoint in $\{1,\ldots, n-1\}$, the curve $\gamma$ has one endpoint in $\{-1,\ldots, -(n-1)\}$, and the curve $\delta$ has one endpoint in $\{-1,\ldots, -n\}$.
\end{enumerate} 

Throughout the proof, we let $i_1$ and $j_1$ denote the endpoints of $\gamma$ and $i_2$ and $j_2$ denote the endpoints of $\delta$. 


\textit{Case~\ref{bad1}}: We may reduce to examining the configurations of curves shown in Figure~\ref{fig_bad1}. In this situation, we have the following morphisms of representations
\[\begin{array}{rccrcccccccccc}
p_1: & V(\gamma) {\twoheadrightarrow} V(\mu_1) & & \iota_1: & V(\mu_1) \hookrightarrow V(\delta)\\
{\iota_2}: & V(\gamma) {\hookrightarrow} V(\mu_2) & & p_2: & V({\mu_2}) \twoheadrightarrow V(\delta). \\
\end{array}\]
Therefore, $\iota_1\circ p_1 \neq 0$ so $\text{Hom}(V(\gamma), V(\delta)) \neq 0$. We also obtain the nonsplit extension
\[\begin{tikzcd}[ampersand replacement=\&]
0 \ar[r] \& V(\gamma) \ar[r, "{\tiny\left[\begin{array}{c} p_1 \\ {\iota_2} \end{array}\right]}"]
    \&  V(\mu_1)\oplus V({\mu_2}) \ar[r, "{\tiny\left[\begin{array}{r} -\iota_1\\ {p_2} \end{array}\right]^T}"]  \& V(\delta) \ar[r] \& 0,
\end{tikzcd}\]
which shows that $\text{Ext}^1(V(\delta),V(\gamma)) \neq 0$.

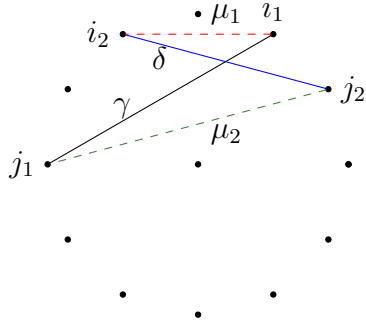
\begin{figure}
\begin{tikzpicture}
\draw [fill] circle (1pt);

	\foreach \a in {0,30,...,360} {
		\draw[fill] (\a:2cm) circle (1pt); 
		}
	
	\node[inner sep=0.5, fill=none, label=above:{}] (a) at (0:0) {};
	\node[inner sep=0.5, fill=none, label=right:{$j_2$}] (b) at (30:2) {};
	\node[inner sep=0.5, fill=none, label=above:{$i_1$}] (c) at (60:2) {};
	\node[inner sep=0.5, fill=none, label=below:{}] (d) at (90:2) {};
	\node[inner sep=0.5, fill=none, label=left:{$i_2$}] (e) at (120:2) {};
	\node[inner sep=0.5, fill=none, label=left:{}] (f) at (150:2) {};
	\node[inner sep=0.5, fill=none, label=left:{$j_1$}] (g) at (180:2) {};
	\node[inner sep=0.5, fill=none, label=left:{}] (h) at (210:2) {};
	\node[inner sep=0.5, fill=none, label=below:{}] (i) at (240:2) {};
	\node[inner sep=0.5, fill=none, label=right:{}] (j) at (270:2){};
	\node[inner sep=0.5, fill=none, label=below:{}] (k) at (300:2) {};
	\node[inner sep=0.5, fill=none, label=right:{}] (l) at (330:2) {};
	\node[inner sep=0.5, fill=none, label=right:{}] (m) at (360:2) {};

\path[every node/.style={font=\sffamily\small}]
		(c) edge[color=black] node [midway] {} (g);	
\path[every node/.style={font=\sffamily\small}]
		(e) edge[color=blue] node [midway] {} (b);

\path[every node/.style={font=\sffamily\small}]
		(b) edge[color=OliveGreen, dashed] node [midway] {} (g);	
\path[every node/.style={font=\sffamily\small}]
		(c) edge[color=red, dashed] node [midway] {} (e);

	\node[inner sep=0.5, fill=none, label=right:{$\gamma$}] at (150:1.5) {};
	\node[inner sep=0.5, fill=none, label=center:{$\delta$}] at (110:1.5) {};

	\node[inner sep=0.5, fill=none, label=right:{$\mu_1$}] at (90:2) {};	
	\node[inner sep=0.5, fill=none, label=right:{$\mu_2$}] at (90:0.4) {};
	
\end{tikzpicture}
\centering
\caption{The endpoint $j_2$ belongs to $\{i_1+1, \ldots, n\}\cup \{-n\}$. All other endpoints belong to $\{1, \ldots, n\}$ and satisfy $j_1 < i_2 < i_1$. When $j_2 \neq -n$, we also have that $i_1 < j_2$.} 
\label{fig_bad1}
\end{figure}

\textit{Case~\ref{bad2}}: Figure~\ref{fig_bad2} shows the possible configurations of curves that can appear in this case. In the situation of Figure~\ref{fig_bad2}$(a)$, we have the following morphisms of representations
\[\begin{array}{rccrcccccccccc}
p_1: & V(\gamma) {\twoheadrightarrow} V(\mu_1) & & \iota_1: & V(\mu_1) \hookrightarrow V(\delta)\\
{\iota_2}: & V(\gamma) {\hookrightarrow} V(\mu_2) & & p_2: & V({\mu_2}) \twoheadrightarrow V(\delta). \\
\end{array}\]
It follows that $\iota_1\circ p_1 \neq 0$ so $\text{Hom}(V(\gamma),V(\delta)) \neq 0$. We also obtain the following nonsplit extension
\[\begin{tikzcd}[ampersand replacement=\&]
0 \ar[r] \& V(\gamma) \ar[r, "{\tiny\left[\begin{array}{c} p_1 \\ {\iota_2} \end{array}\right]}"]
    \&  V(\mu_1)\oplus V({\mu_2}) \ar[r, "{\tiny\left[\begin{array}{r} -\iota_1\\ {p_2} \end{array}\right]^T}"]  \& V(\delta) \ar[r] \& 0,
\end{tikzcd}\]
which shows that $\text{Ext}^1(V(\delta),V(\gamma)) \neq 0$. In the situation of Figure~\ref{fig_bad2}$(b)$, we have the following morphisms of representations
\[\begin{array}{rccrccccccccc}
\iota: & V(\gamma) {\hookrightarrow} V(\mu) & & \iota^\prime: & V(\mu) \hookrightarrow V(\delta)\\
\overline{\iota}: & V(\gamma) {\hookrightarrow} V(\overline{\mu}) & & \overline{\iota}^\prime: & V(\overline{\mu}) \hookrightarrow V(\delta)   \\
q: & V(\gamma) {\twoheadrightarrow}  V(\nu) & & \epsilon: & V(\nu) \hookrightarrow V(\delta). \\
\end{array}\]
It follows that $\iota^\prime \circ \iota \neq 0$ so $\text{Hom}(V(\gamma),V(\delta)) \neq 0$. We also obtain the following nonsplit extension
\[\begin{tikzcd}[ampersand replacement=\&]
0 \ar[r] \& V(\gamma) \ar[r, "{\tiny\left[\begin{array}{c} \iota \\ \overline{\iota} \\ q  \end{array}\right]}"]
    \&  V(\mu)\oplus V(\overline{\mu})\oplus V(\nu) \ar[r, "{\tiny\left[\begin{array}{r} -\iota^\prime\\ \overline{\iota}^\prime \\ \epsilon  \end{array}\right]^T}"]  \& V(\delta) \ar[r] \& 0.
\end{tikzcd}\] In the special case that $j_1 = j_2$, removing $V(\nu)$ and the maps $q$ and $\epsilon$ from the above diagram produces a nonsplit extension. Thus, $\text{Ext}^1(V(\delta),V(\gamma)) \neq 0$.

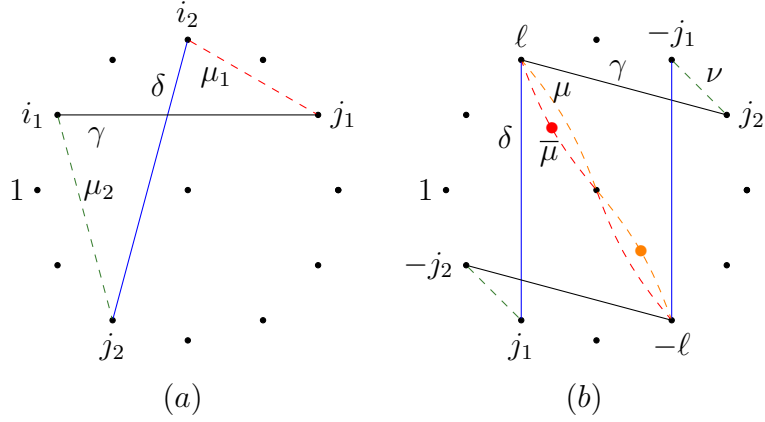
\begin{figure}
\[\begin{array}{cccc}
\begin{tikzpicture}
\draw [fill] circle (1pt);

	\foreach \a in {0,30,...,360} {
		\draw[fill] (\a:2cm) circle (1pt); 
		}
	
	\node[inner sep=0.5, fill=none, label=above:{}] (a) at (0:0) {};
	\node[inner sep=0.5, fill=none, label=right:{$j_1$}] (b) at (30:2) {};
	\node[inner sep=0.5, fill=none, label=above:{}] (c) at (60:2) {};
	\node[inner sep=0.5, fill=none, label=above:{$i_2$}] (d) at (90:2) {};
	\node[inner sep=0.5, fill=none, label=left:{}] (e) at (120:2) {};
	\node[inner sep=0.5, fill=none, label=left:{$i_1$}] (f) at (150:2) {};
	\node[inner sep=0.5, fill=none, label=left:{$1$}] (g) at (180:2) {};
	\node[inner sep=0.5, fill=none, label=left:{}] (h) at (210:2) {};
	\node[inner sep=0.5, fill=none, label=below:{$j_2$}] (i) at (240:2) {};
	\node[inner sep=0.5, fill=none, label=right:{}] (j) at (270:2){};
	\node[inner sep=0.5, fill=none, label=below:{}] (k) at (300:2) {};
	\node[inner sep=0.5, fill=none, label=right:{}] (l) at (330:2) {};
	\node[inner sep=0.5, fill=none, label=right:{}] (m) at (360:2) {};

\path[every node/.style={font=\sffamily\small}]
		(d) edge[color=blue] node [midway] {} (i);	
		\path[every node/.style={font=\sffamily\small}]
		(b) edge[color=black] node [midway] {} (f);

\path[every node/.style={font=\sffamily\small}]
		(f) edge[color=OliveGreen, dashed] node [midway] {} (i);
	\path[every node/.style={font=\sffamily\small}]
		(d) edge[color=red, dashed] node [midway] {} (b);

	\node[inner sep=0.5, fill=none, label=center:{$\gamma$}] at (150:1.4) {};
	\node[inner sep=0.5, fill=none, label=left:{$\delta$}] at (95:1.4) {};

	\node[inner sep=0.5, fill=none, label=right:{$\mu_1$}] at (90:1.5) {};	
	\node[inner sep=0.5, fill=none, label=left:{$\mu_2$}] at (180:0.8) {};

\end{tikzpicture} &
\begin{tikzpicture}
\draw [fill] circle (1pt);

	\foreach \a in {0,30,...,360} {
		\draw[fill] (\a:2cm) circle (1pt); 
		}
	
	\node[inner sep=0.5, fill=none, label=above:{}] (a) at (0:0) {};
	\node[inner sep=0.5, fill=none, label=right:{$j_2$}] (b) at (30:2) {};
	\node[inner sep=0.5, fill=none, label=above:{$-j_1$}] (c) at (60:2) {};
	\node[inner sep=0.5, fill=none, label=below:{}] (d) at (90:2) {};
	\node[inner sep=0.5, fill=none, label=above:{$\ell$}] (e) at (120:2) {};
	\node[inner sep=0.5, fill=none, label=left:{}] (f) at (150:2) {};
	\node[inner sep=0.5, fill=none, label=left:{$1$}] (g) at (180:2) {};
	\node[inner sep=0.5, fill=none, label=left:{$-j_2$}] (h) at (210:2) {};
	\node[inner sep=0.5, fill=none, label=below:{$j_1$}] (i) at (240:2) {};
	\node[inner sep=0.5, fill=none, label=right:{}] (j) at (270:2){};
	\node[inner sep=0.5, fill=none, label=below:{$-\ell$}] (k) at (300:2) {};
	\node[inner sep=0.5, fill=none, label=right:{}] (l) at (330:2) {};
	\node[inner sep=0.5, fill=none, label=right:{}] (m) at (360:2) {};

\path[every node/.style={font=\sffamily\small}]
		(e) edge[color=blue] node [midway] {} (i);	
\path[every node/.style={font=\sffamily\small}]
		(e) edge[color=black] node [midway] {} (b);
\path[every node/.style={font=\sffamily\small}]
		(c) edge[color=blue] node [midway] {} (k);
		\path[every node/.style={font=\sffamily\small}]
		(h) edge[color=black] node [midway] {} (k);

\path[every node/.style={font=\sffamily\small}]
		(h) edge[color=OliveGreen, dashed] node [midway] {} (i);
	\path[every node/.style={font=\sffamily\small}]
		(c) edge[color=OliveGreen, dashed] node [midway] {} (b);

\path[every node/.style={font=\sffamily\small}]
		(a) edge[color=red, dashed, bend left=10] node [midway] {$\bullet$} (e);	
		\path[every node/.style={font=\sffamily\small}]
		(a) edge[color=red, dashed, bend right=10] node [midway] {} (k);		
		
		\path[every node/.style={font=\sffamily\small}]
		(a) edge[color=orange, dashed, bend right=10] node [midway] {} (e);	
		\path[every node/.style={font=\sffamily\small}]
		(a) edge[color=orange, dashed, bend left=10] node [midway] {$\bullet$} (k);

	\node[inner sep=0.5, fill=none, label=center:{$\delta$}] at (150:1.4) {};
	\node[inner sep=0.5, fill=none, label=right:{$\gamma$}] at (90:1.6) {};

	\node[inner sep=0.5, fill=none, label=right:{$\mu$}] at (120:1.5) {};	
	\node[inner sep=0.5, fill=none, label=left:{$\overline{\mu}$}] at (124:0.6) {};
	
	\node[inner sep=0.5, fill=none, label=right:{$\nu$}] at (50:2) {};
	
\end{tikzpicture}\\
(a) & (b)
\end{array}\]
\caption{Here, $1\le i_1 < i_2 < j_1 \le n$ and $j_2 \in \{-1, \ldots, -(n-1)\}$. In $(a)$, we show $\gamma$ and $\delta$ when $i_1 \neq i_2$; the endpoint $j_1$ is allowed to equal $n$. In $(b)$, we show $\gamma$ and $\delta$ when $i_1 = i_2$. In $(b)$, if $j_1 = -j_2$, the equivalence class $[\nu]$ does not appear. }
\label{fig_bad2}
\end{figure}

\textit{Case~\ref{bad3}}: Figure~\ref{fig_bad3} shows the possible configurations of curves that can appear in this case. In the situation of Figure~\ref{fig_bad3}$(a)$, we have the following morphisms
\[\begin{array}{rccrcccccccccc}
p: & V(\gamma) {\twoheadrightarrow} V(\nu) & & q: & V(\nu) \twoheadrightarrow V(\delta)\\
p^\prime: & V(\gamma) {\twoheadrightarrow} V(\mu) & & q^\prime: & V({\mu}) \twoheadrightarrow V(\delta). \\
\end{array}\]
This implies that $q\circ p \neq 0$ so $\text{Hom}(V(\gamma), V(\delta)) \neq 0$. We have also the following nonsplit extension
\[\begin{tikzcd}[ampersand replacement=\&]
0 \ar[r] \& V(\gamma) \ar[r, "{\tiny\left[\begin{array}{c} p \\ p^\prime \end{array}\right]}"]
    \&  V(\nu)\oplus V({\mu}) \ar[r, "{\tiny\left[\begin{array}{r} -q\\ q^\prime \end{array}\right]^T}"]  \& V(\delta) \ar[r] \& 0,
\end{tikzcd}\]
which shows that $\text{Ext}^1(V(\delta), V(\gamma)) \neq 0$.

\begin{figure}
\[\begin{array}{cccccc}
\begin{tikzpicture}
\draw [fill] circle (1pt);

	\foreach \a in {0,30,...,360} {
		\draw[fill] (\a:2cm) circle (1pt); 
		}
	
	\node[inner sep=0.5, fill=none, label=above:{}] (a) at (0:0) {};
	\node[inner sep=0.5, fill=none, label=right:{}] (b) at (30:2) {};
	\node[inner sep=0.5, fill=none, label=above:{}] (c) at (60:2) {};
	\node[inner sep=0.5, fill=none, label=above:{$i_2$}] (d) at (90:2) {};
	\node[inner sep=0.5, fill=none, label=left:{$i_1$}] (e) at (120:2) {};
	\node[inner sep=0.5, fill=none, label=left:{}] (f) at (150:2) {};
	\node[inner sep=0.5, fill=none, label=left:{$1$}] (g) at (180:2) {};
	\node[inner sep=0.5, fill=none, label=left:{$j_2$}] (h) at (210:2) {};
	\node[inner sep=0.5, fill=none, label=below:{$j_1$}] (i) at (240:2) {};
	\node[inner sep=0.5, fill=none, label=right:{}] (j) at (270:2){};
	\node[inner sep=0.5, fill=none, label=below:{}] (k) at (300:2) {};
	\node[inner sep=0.5, fill=none, label=right:{}] (l) at (330:2) {};
	\node[inner sep=0.5, fill=none, label=right:{}] (m) at (360:2) {};

\path[every node/.style={font=\sffamily\small}]
		(e) edge[color=black] node [midway] {} (i);	
		\path[every node/.style={font=\sffamily\small}]
		(d) edge[color=blue] node [midway] {} (h);

\path[every node/.style={font=\sffamily\small}]
		(h) edge[color=OliveGreen, dashed] node [midway] {} (e);
	\path[every node/.style={font=\sffamily\small}]
		(d) edge[color=red, dashed] node [midway] {} (i);

	\node[inner sep=0.5, fill=none, label=center:{$\gamma$}] at (127:1.4) {};
	\node[inner sep=0.5, fill=none, label=left:{$\delta$}] at (99:1.4) {};

	\node[inner sep=0.5, fill=none, label=center:{$\mu$}] at (88:1.5) {};	
	\node[inner sep=0.5, fill=none, label=center:{$\nu$}] at (180:1.7) {};

\end{tikzpicture} &
\begin{tikzpicture}
\draw [fill] circle (1pt);

	\foreach \a in {0,30,...,360} {
		\draw[fill] (\a:2cm) circle (1pt); 
		}
	
	\node[inner sep=0.5, fill=none, label=above:{}] (a) at (0:0) {};
	\node[inner sep=0.5, fill=none, label=right:{}] (b) at (30:2) {};
	\node[inner sep=0.5, fill=none, label=above:{}] (c) at (60:2) {};
	\node[inner sep=0.5, fill=none, label=above:{$i_2$}] (d) at (90:2) {};
	\node[inner sep=0.5, fill=none, label=left:{$i_1$}] (e) at (120:2) {};
	\node[inner sep=0.5, fill=none, label=left:{}] (f) at (150:2) {};
	\node[inner sep=0.5, fill=none, label=left:{$1$}] (g) at (180:2) {};
	\node[inner sep=0.5, fill=none, label=left:{$j_2$}] (h) at (210:2) {};
	\node[inner sep=0.5, fill=none, label=below:{}] (i) at (240:2) {};
	\node[inner sep=0.5, fill=none, label=below:{$j_1$}] (j) at (270:2){};
	\node[inner sep=0.5, fill=none, label=below:{}] (k) at (300:2) {};
	\node[inner sep=0.5, fill=none, label=right:{}] (l) at (330:2) {};
	\node[inner sep=0.5, fill=none, label=right:{}] (m) at (360:2) {};

\path[every node/.style={font=\sffamily\small}]
		(e) edge[color=black] node [midway] {} (j);	
		\path[every node/.style={font=\sffamily\small}]
		(d) edge[color=blue] node [midway] {} (h);

\path[every node/.style={font=\sffamily\small}]
		(h) edge[color=OliveGreen, dashed] node [midway] {} (e);
	\path[every node/.style={font=\sffamily\small}]
		(d) edge[color=red, dashed, bend right=10] node [midway] {} (a);		
		
\path[every node/.style={font=\sffamily\small}]
		(j) edge[color=red, dashed, bend left=10] node [midway] {$\bullet$} (a);	
		
		\path[every node/.style={font=\sffamily\small}]
		(d) edge[color=orange, dashed, bend left=10] node [midway] {$\bullet$} (a);		
		
\path[every node/.style={font=\sffamily\small}]
		(j) edge[color=orange, dashed, bend right=10] node [midway] {} (a);

	\node[inner sep=0.5, fill=none, label=left:{$\gamma$}] at (180:0.5) {};
	\node[inner sep=0.5, fill=none, label=left:{$\delta$}] at (99:1.4) {};

	\node[inner sep=0.5, fill=none, label=center:{$\overline{\mu}$}] at (80:1.5) {};	
	\node[inner sep=0.5, fill=none, label=left:{$\mu$}] at (90:0.5) {};	
	\node[inner sep=0.5, fill=none, label=center:{$\nu$}] at (180:1.7) {};

\end{tikzpicture} &
\begin{tikzpicture}
\draw [fill] circle (1pt);

	\foreach \a in {0,30,...,360} {
		\draw[fill] (\a:2cm) circle (1pt); 
		}
	
	\node[inner sep=0.5, fill=none, label=below:{$-n$}] (a) at (0:0) {};
	\node[inner sep=0.5, fill=none, label=right:{}] (b) at (30:2) {};
	\node[inner sep=0.5, fill=none, label=above:{}] (c) at (60:2) {};
	\node[inner sep=0.5, fill=none, label=above:{$i_2$}] (d) at (90:2) {};
	\node[inner sep=0.5, fill=none, label=left:{$i_1$}] (e) at (120:2) {};
	\node[inner sep=0.5, fill=none, label=left:{}] (f) at (150:2) {};
	\node[inner sep=0.5, fill=none, label=left:{$1$}] (g) at (180:2) {};
	\node[inner sep=0.5, fill=none, label=left:{$j_2$}] (h) at (210:2) {};
	\node[inner sep=0.5, fill=none, label=below:{}] (i) at (240:2) {};
	\node[inner sep=0.5, fill=none, label=right:{}] (j) at (270:2){};
	\node[inner sep=0.5, fill=none, label=below:{}] (k) at (300:2) {};
	\node[inner sep=0.5, fill=none, label=right:{}] (l) at (330:2) {};
	\node[inner sep=0.5, fill=none, label=right:{}] (m) at (360:2) {};

\path[every node/.style={font=\sffamily\small}]
		(e) edge[color=black] node [midway] {$\bullet$} (a);	
		\path[every node/.style={font=\sffamily\small}]
		(d) edge[color=blue] node [midway] {} (h);

\path[every node/.style={font=\sffamily\small}]
		(h) edge[color=OliveGreen, dashed] node [midway] {} (e);
	\path[every node/.style={font=\sffamily\small}]
		(d) edge[color=red, dashed] node [midway] {} (a);

	\node[inner sep=0.5, fill=none, label=center:{$\gamma$}] at (127:1.4) {};
	\node[inner sep=0.5, fill=none, label=left:{$\delta$}] at (99:1.4) {};

	\node[inner sep=0.5, fill=none, label=right:{$\mu$}] at (88:1.5) {};	
	\node[inner sep=0.5, fill=none, label=center:{$\nu$}] at (180:1.7) {};

\end{tikzpicture} \\ 
(a) & (b) & (c)
\end{array}\]
\caption{Here, $1 \le i_1 < i_2 \le n-1$. In $(a)$, we require that $i_2 \neq -j_1$ $-(n-1) \le j_2 < j_1 \le -1$. In addition, it can happen that the centroid is properly contained in the region bounded by $\gamma$, $\delta$, and $\mu$. In $(b)$, we require that $i_2 = -j_1$ and $-(n-1) \le j_2 < j_1$. In $(c)$, we require that $j_2 = -n$ and $-(n-1) \le j_2 < -i_2$.}
\label{fig_bad3}
\end{figure}
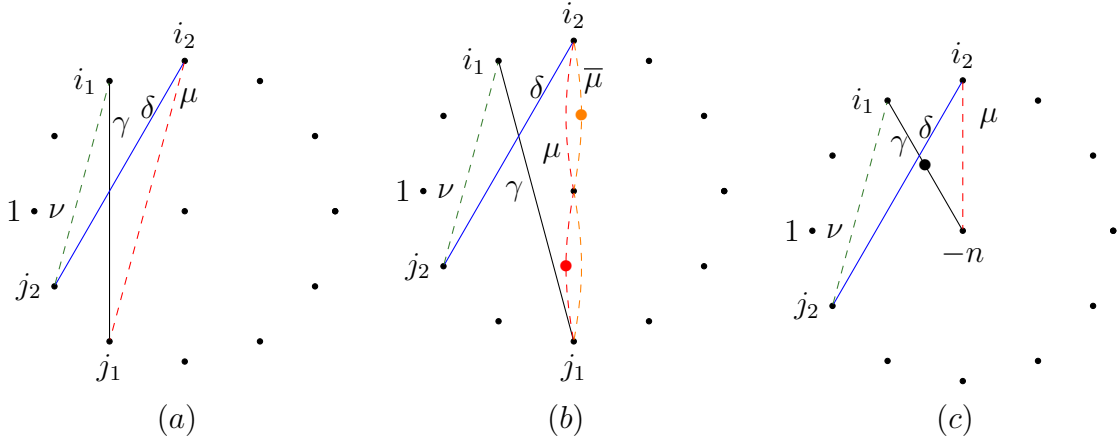

In the situation of Figure~\ref{fig_bad3}$(b)$, we have the following morphisms
 \[\begin{array}{rccrccccccccc}
p: & V(\gamma) {\twoheadrightarrow} V(\nu) & & q: & V(\nu) \twoheadrightarrow V(\delta)\\
q_1: & V(\gamma) {\twoheadrightarrow} V({\mu}) & & \iota: & V({\mu}) \hookrightarrow V(\delta)   \\
\overline{q_1}: & V(\gamma) {\twoheadrightarrow}  V(\overline{\mu}) & & \overline{\iota}: & V(\overline{\mu}) \hookrightarrow V(\delta). \\
\end{array}\]
We obtain that $q\circ p \neq 0$ so $\text{Hom}(V(\gamma),V(\delta)) \neq 0$. We also obtain the nonsplit extension
\[\begin{tikzcd}[ampersand replacement=\&]
0 \ar[r] \& V(\gamma) \ar[r, "{\tiny\left[\begin{array}{c} p \\ q_1 \\ \overline{q_1}  \end{array}\right]}"]
    \&  V(\nu)\oplus V({\mu})\oplus V(\overline{\mu}) \ar[r, "{\tiny\left[\begin{array}{r} q\\ \iota \\ -\overline{\iota}  \end{array}\right]^T}"]  \& V(\delta) \ar[r] \& 0,
\end{tikzcd}\] 
which shows that $\text{Ext}^1(V(\delta), V(\gamma)) \neq 0$.

Lastly, in the situation of Figure~\ref{fig_bad3}$(c)$, we have the following morphisms
\[\begin{array}{rccrcccccccccc}
\iota: & V(\gamma) {\hookrightarrow} V(\nu) & & p^\prime: & V(\nu) \twoheadrightarrow V(\delta)\\
p: & V(\gamma) {\twoheadrightarrow} V(\mu) & & \iota^\prime: & V({\mu}) \hookrightarrow V(\delta). \\
\end{array}\]
We observe that $\iota^\prime\circ p \neq 0$ so $\text{Hom}(V(\gamma),V(\delta)) \neq 0$. We also have the nonsplit extension
\[\begin{tikzcd}[ampersand replacement=\&]
0 \ar[r] \& V(\gamma) \ar[r, "{\tiny\left[\begin{array}{c} \iota \\ p \end{array}\right]}"]
    \&  V(\nu)\oplus V({\mu}) \ar[r, "{\tiny\left[\begin{array}{r} -p^\prime\\ \iota^\prime \end{array}\right]^T}"]  \& V(\delta) \ar[r] \& 0,
\end{tikzcd}\]
which show that $\text{Ext}^1(V(\delta), V(\gamma)) \neq 0$. This completes the proof.\end{proof}

\begin{lem}\label{no_intersection_lemma}
Let $[\gamma]$ and $[\delta]$ be two equivalence classes that do not form a bad pair and whose curves have no common endpoints. Then $(V(\gamma),V(\delta))$ and $(V(\delta), V(\gamma))$ are exceptional sequences.
\end{lem}

\begin{proof}
Assume that $\gamma$ and $\delta$ are the essential curves in the equivalence classes $[\gamma]$ and $[\delta]$, respectively. We prove the result in each of the following cases:
\begin{enumerate}
\item\label{good1} both endpoints of the curves $\gamma$ and $\delta$ are in $\{1, \ldots, n\}$ or the curve $\gamma$ has both its endpoints in $\{1,\ldots, n-1\}$ and the curve $\delta$ has one endpoint in $\{1,\ldots, n-1\}$ and one endpoint is $-n$;
\item\label{good2} both endpoints of the curve $\gamma$ are in $\{1,\ldots, n\}$, one endpoint of $\delta$ is in $\{1, \ldots, n-1\}$, and the other is in $\{-1,\ldots, -(n-1)\}$;
\item\label{good3} the curves $\gamma$ and $\delta$ each have one endpoint in $\{1,\ldots, n-1\}$, the curve $\gamma$ has one endpoint in $\{-1,\ldots, -(n-1)\}$, and the curve $\delta$ has one endpoint in $\{-1,\ldots, -n\}$.
\end{enumerate} 

For Case~\ref{good1}, we show the possible configurations of curves in Figure~\ref{fig_good1}. Similarly, we show the possible configurations of curves for Case~\ref{good2} and Case~\ref{good3} in Figure~\ref{fig_good2} and Figure~\ref{fig_good3}, respectively. As is done in the proof of Lemma~\ref{common_endpt_lemma}, it is routine to verify in each case that \[\langle \textbf{dim}(V(\gamma)), \textbf{dim}(V(\delta)) \rangle = \textbf{dim}(V(\delta)), \textbf{dim}(V(\gamma)) \rangle = 0.\]

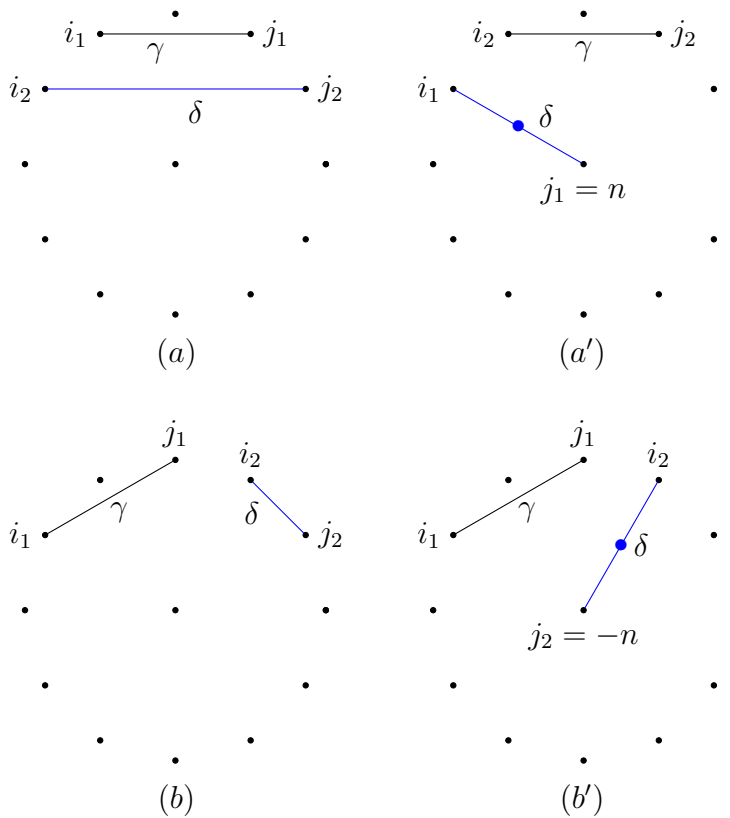
\begin{figure}[!htbp]
\[\begin{array}{ccccc} \begin{tikzpicture}
\draw [fill] circle (1pt);

	\foreach \a in {0,30,...,360} {
		\draw[fill] (\a:2cm) circle (1pt); 
		}
	
	\node[inner sep=0.5, fill=none, label=above:{}] (a) at (0:0) {};
	\node[inner sep=0.5, fill=none, label=right:{$j_2$}] (b) at (30:2) {};
	\node[inner sep=0.5, fill=none, label=right:{$j_1$}] (c) at (60:2) {};
	\node[inner sep=0.5, fill=none, label=below:{}] (d) at (90:2) {};
	\node[inner sep=0.5, fill=none, label=left:{$i_1$}] (e) at (120:2) {};
	\node[inner sep=0.5, fill=none, label=left:{$i_2$}] (f) at (150:2) {};
	\node[inner sep=0.5, fill=none, label=left:{}] (g) at (180:2) {};
	\node[inner sep=0.5, fill=none, label=left:{}] (h) at (210:2) {};
	\node[inner sep=0.5, fill=none, label=below:{}] (i) at (240:2) {};
	\node[inner sep=0.5, fill=none, label=right:{}] (j) at (270:2){};
	\node[inner sep=0.5, fill=none, label=below:{}] (k) at (300:2) {};
	\node[inner sep=0.5, fill=none, label=right:{}] (l) at (330:2) {};
	\node[inner sep=0.5, fill=none, label=right:{}] (m) at (360:2) {};

\path[every node/.style={font=\sffamily\small}]
		(c) edge[color=black] node [midway] {} (e);	
\path[every node/.style={font=\sffamily\small}]
		(b) edge[color=blue] node [midway] {} (f);
		
	\node[inner sep=0.5, fill=none, label=right:{$\delta$}] at (90:0.7) {};
	\node[inner sep=0.5, fill=none, label=center:{$\gamma$}] at (100:1.5) {};
\end{tikzpicture} & & \begin{tikzpicture}
\draw [fill] circle (1pt);

	\foreach \a in {0,30,...,360} {
		\draw[fill] (\a:2cm) circle (1pt); 
		}
	
	\node[inner sep=0.5, fill=none, label=below:{$j_1=n$}] (a) at (0:0) {};
	\node[inner sep=0.5, fill=none, label=above:{}] (b) at (30:2) {};
	\node[inner sep=0.5, fill=none, label=right:{$j_2$}] (c) at (60:2) {};
	\node[inner sep=0.5, fill=none, label=below:{}] (d) at (90:2) {};
	\node[inner sep=0.5, fill=none, label=left:{$i_2$}] (e) at (120:2) {};
	\node[inner sep=0.5, fill=none, label=left:{$i_1$}] (f) at (150:2) {};
	\node[inner sep=0.5, fill=none, label=left:{}] (g) at (180:2) {};
	\node[inner sep=0.5, fill=none, label=left:{}] (h) at (210:2) {};
	\node[inner sep=0.5, fill=none, label=below:{}] (i) at (240:2) {};
	\node[inner sep=0.5, fill=none, label=right:{}] (j) at (270:2){};
	\node[inner sep=0.5, fill=none, label=below:{}] (k) at (300:2) {};
	\node[inner sep=0.5, fill=none, label=right:{}] (l) at (330:2) {};
	\node[inner sep=0.5, fill=none, label=right:{}] (m) at (360:2) {};

\path[every node/.style={font=\sffamily\small}]
		(f) edge[color=blue] node [midway] {$\bullet$} (a);	
\path[every node/.style={font=\sffamily\small}]
		(c) edge[color=black] node [midway] {} (e);
		
	\node[inner sep=0.5, fill=none, label=right:{$\delta$}] at (140:1) {};
	\node[inner sep=0.5, fill=none, label=center:{$\gamma$}] at (90:1.5) {};
\end{tikzpicture} \\ (a) & & (a^\prime) \\ \\
 \begin{tikzpicture}
\draw [fill] circle (1pt);

	\foreach \a in {0,30,...,360} {
		\draw[fill] (\a:2cm) circle (1pt); 
		}
	
	\node[inner sep=0.5, fill=none, label=below:{}] (a) at (0:0) {};
	\node[inner sep=0.5, fill=none, label=right:{$j_2$}] (b) at (30:2) {};
	\node[inner sep=0.5, fill=none, label=above:{$i_2$}] (c) at (60:2) {};
	\node[inner sep=0.5, fill=none, label=above:{$j_1$}] (d) at (90:2) {};
	\node[inner sep=0.5, fill=none, label=above:{}] (e) at (120:2) {};
	\node[inner sep=0.5, fill=none, label=left:{$i_1$}] (f) at (150:2) {};
	\node[inner sep=0.5, fill=none, label=left:{}] (g) at (180:2) {};
	\node[inner sep=0.5, fill=none, label=left:{}] (h) at (210:2) {};
	\node[inner sep=0.5, fill=none, label=below:{}] (i) at (240:2) {};
	\node[inner sep=0.5, fill=none, label=right:{}] (j) at (270:2){};
	\node[inner sep=0.5, fill=none, label=below:{}] (k) at (300:2) {};
	\node[inner sep=0.5, fill=none, label=right:{}] (l) at (330:2) {};
	\node[inner sep=0.5, fill=none, label=right:{}] (m) at (360:2) {};

\path[every node/.style={font=\sffamily\small}]
		(c) edge[color=blue] node [midway] {} (b);	
\path[every node/.style={font=\sffamily\small}]
		(d) edge[color=black] node [midway] {} (f);
		
	\node[inner sep=0.5, fill=none, label=center:{$\gamma$}] at (120:1.5) {};
	\node[inner sep=0.5, fill=none, label=right:{$\delta$}] at (60:1.5) {};
\end{tikzpicture}
& & \begin{tikzpicture}
\draw [fill] circle (1pt);

	\foreach \a in {0,30,...,360} {
		\draw[fill] (\a:2cm) circle (1pt); 
		}
	
	\node[inner sep=0.5, fill=none, label=below:{$j_2=-n$}] (a) at (0:0) {};
	\node[inner sep=0.5, fill=none, label=above:{}] (b) at (30:2) {};
	\node[inner sep=0.5, fill=none, label=above:{$i_2$}] (c) at (60:2) {};
	\node[inner sep=0.5, fill=none, label=above:{$j_1$}] (d) at (90:2) {};
	\node[inner sep=0.5, fill=none, label=above:{}] (e) at (120:2) {};
	\node[inner sep=0.5, fill=none, label=left:{$i_1$}] (f) at (150:2) {};
	\node[inner sep=0.5, fill=none, label=left:{}] (g) at (180:2) {};
	\node[inner sep=0.5, fill=none, label=left:{}] (h) at (210:2) {};
	\node[inner sep=0.5, fill=none, label=below:{}] (i) at (240:2) {};
	\node[inner sep=0.5, fill=none, label=right:{}] (j) at (270:2){};
	\node[inner sep=0.5, fill=none, label=below:{}] (k) at (300:2) {};
	\node[inner sep=0.5, fill=none, label=right:{}] (l) at (330:2) {};
	\node[inner sep=0.5, fill=none, label=right:{}] (m) at (360:2) {};

\path[every node/.style={font=\sffamily\small}]
		(c) edge[color=blue] node [midway] {$\bullet$} (a);	
\path[every node/.style={font=\sffamily\small}]
		(d) edge[color=black] node [midway] {} (f);
		
	\node[inner sep=0.5, fill=none, label=center:{$\gamma$}] at (120:1.5) {};
	\node[inner sep=0.5, fill=none, label=right:{$\delta$}] at (60:1) {};
\end{tikzpicture}\\ (b) & & (b^\prime)  \end{array}\]
\caption{In $(a)$, we have $1 \le i_2 < i_1 < j_1 < j_2 \le n$. In $(a^\prime)$, we have $1 \le i_2 < i_1 < j_1 \le n-1$ and $j_2 = -n$. In $(b)$, we have $1 \le i_1 < j_1 < i_2 < j_2 \le n$. In $(b^\prime)$, we have $1 \le i_1 < j_1 < i_2 \le n-1$ and $j_2 = -n$.}
\label{fig_good1}
\end{figure}

\begin{figure}[!htbp]
\[\begin{array}{cccc}
 \begin{tikzpicture}
\draw [fill] circle (1pt);

	\foreach \a in {0,30,...,360} {
		\draw[fill] (\a:2cm) circle (1pt); 
		}
	
	\node[inner sep=0.5, fill=none, label=above:{}] (a) at (0:0) {};
	\node[inner sep=0.5, fill=none, label=right:{}] (b) at (30:2) {};
	\node[inner sep=0.5, fill=none, label=right:{$j_1$}] (c) at (60:2) {};
	\node[inner sep=0.5, fill=none, label=below:{}] (d) at (90:2) {};
	\node[inner sep=0.5, fill=none, label=left:{$i_1$}] (e) at (120:2) {};
	\node[inner sep=0.5, fill=none, label=left:{$i_2$}] (f) at (150:2) {};
	\node[inner sep=0.5, fill=none, label=left:{$1$}] (g) at (180:2) {};
	\node[inner sep=0.5, fill=none, label=left:{}] (h) at (210:2) {};
	\node[inner sep=0.5, fill=none, label=below:{$j_2$}] (i) at (240:2) {};
	\node[inner sep=0.5, fill=none, label=right:{}] (j) at (270:2){};
	\node[inner sep=0.5, fill=none, label=below:{}] (k) at (300:2) {};
	\node[inner sep=0.5, fill=none, label=right:{}] (l) at (330:2) {};
	\node[inner sep=0.5, fill=none, label=right:{}] (m) at (360:2) {};

\path[every node/.style={font=\sffamily\small}]
		(c) edge[color=black] node [midway] {} (e);	
\path[every node/.style={font=\sffamily\small}]
		(i) edge[color=blue] node [midway] {} (f);
		
	\node[inner sep=0.5, fill=none, label=center:{$\delta$}] at (180:1.3) {};
	\node[inner sep=0.5, fill=none, label=center:{$\gamma$}] at (100:1.5) {};
\end{tikzpicture}& & \begin{tikzpicture}
\draw [fill] circle (1pt);

	\foreach \a in {0,30,...,360} {
		\draw[fill] (\a:2cm) circle (1pt); 
		}
	
	\node[inner sep=0.5, fill=none, label=above:{}] (a) at (0:0) {};
	\node[inner sep=0.5, fill=none, label=above:{}] (b) at (30:2) {};
	\node[inner sep=0.5, fill=none, label=above:{$i_2$}] (c) at (60:2) {};
	\node[inner sep=0.5, fill=none, label=above:{$j_1$}] (d) at (90:2) {};
	\node[inner sep=0.5, fill=none, label=above:{}] (e) at (120:2) {};
	\node[inner sep=0.5, fill=none, label=left:{$i_1$}] (f) at (150:2) {};
	\node[inner sep=0.5, fill=none, label=left:{}] (g) at (180:2) {};
	\node[inner sep=0.5, fill=none, label=left:{$j_2$}] (h) at (210:2) {};
	\node[inner sep=0.5, fill=none, label=below:{}] (i) at (240:2) {};
	\node[inner sep=0.5, fill=none, label=right:{}] (j) at (270:2){};
	\node[inner sep=0.5, fill=none, label=below:{}] (k) at (300:2) {};
	\node[inner sep=0.5, fill=none, label=right:{}] (l) at (330:2) {};
	\node[inner sep=0.5, fill=none, label=right:{}] (m) at (360:2) {};

\path[every node/.style={font=\sffamily\small}]
		(d) edge[color=black] node [midway] {} (f);	
\path[every node/.style={font=\sffamily\small}]
		(c) edge[color=blue] node [midway] {} (h);
		
	\node[inner sep=0.5, fill=none, label=right:{$\delta$}] at (150:1) {};
	\node[inner sep=0.5, fill=none, label=center:{$\gamma$}] at (100:1.5) {};
	\node[inner sep=0.5, fill=none, label=left:{$1$}] at (180:2) {};
\end{tikzpicture} \\
 (a) & & (b)
\end{array}\]
\caption{In $(a)$, we have $1 \le i_2 < i_1 < j_1 \le n$ and $-(n-1) \le j_2 \le -j_1$. In $(b)$, we have $1 \le i_1 < j_1 < i_2 < -j_2 \le n-1$ and $-(n-1) \le j_2 \le -i_2$.}
\label{fig_good2}
\end{figure}
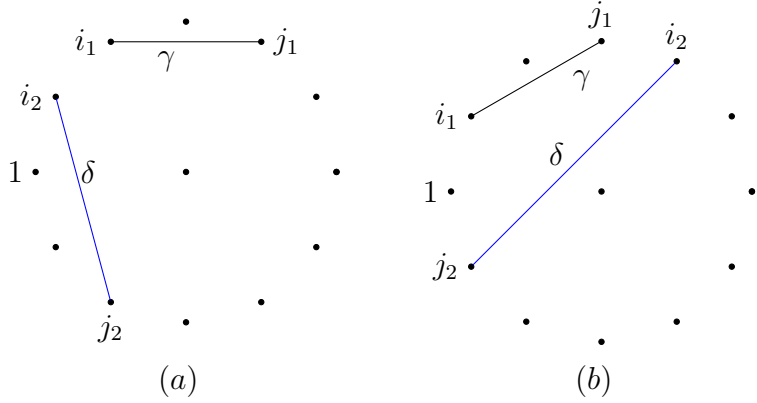

\begin{figure}[!htbp]
\[\begin{tikzpicture}
\draw [fill] circle (1pt);

	\foreach \a in {0,30,...,360} {
		\draw[fill] (\a:2cm) circle (1pt); 
		}
	
	\node[inner sep=0.5, fill=none, label=above:{}] (a) at (0:0) {};
	\node[inner sep=0.5, fill=none, label=above:{}] (b) at (30:2) {};
	\node[inner sep=0.5, fill=none, label=above:{}] (c) at (60:2) {};
	\node[inner sep=0.5, fill=none, label=above:{$i_2$}] (d) at (90:2) {};
	\node[inner sep=0.5, fill=none, label=above:{}] (e) at (120:2) {};
	\node[inner sep=0.5, fill=none, label=left:{$i_1$}] (f) at (150:2) {};
	\node[inner sep=0.5, fill=none, label=left:{}] (g) at (180:2) {};
	\node[inner sep=0.5, fill=none, label=left:{$j_1$}] (h) at (210:2) {};
	\node[inner sep=0.5, fill=none, label=below:{$j_2$}] (i) at (240:2) {};
	\node[inner sep=0.5, fill=none, label=right:{}] (j) at (270:2){};
	\node[inner sep=0.5, fill=none, label=below:{}] (k) at (300:2) {};
	\node[inner sep=0.5, fill=none, label=right:{}] (l) at (330:2) {};
	\node[inner sep=0.5, fill=none, label=right:{}] (m) at (360:2) {};

\path[every node/.style={font=\sffamily\small}]
		(d) edge[color=blue] node [midway] {} (i);	
\path[every node/.style={font=\sffamily\small}]
		(f) edge[color=black] node [midway] {} (h);
		
	\node[inner sep=0.5, fill=none, label=right:{$\delta$}] at (150:1) {};
	\node[inner sep=0.5, fill=none, label=center:{$\gamma$}] at (180:1.5) {};
	\node[inner sep=0.5, fill=none, label=left:{$1$}] at (180:2) {};
\end{tikzpicture} \]
\caption{Here, we have $1 \le i_1 < i_2 < -j_2$ and $-(n-1) \le j_1 < j_2 < -i_2$, or we have that $ 1 \le i_1 < i_2 \le n-1$ and $-(n-1) \le j_1 < -i_2$ and $j_2 = -n$.}
\label{fig_good3}
\end{figure}
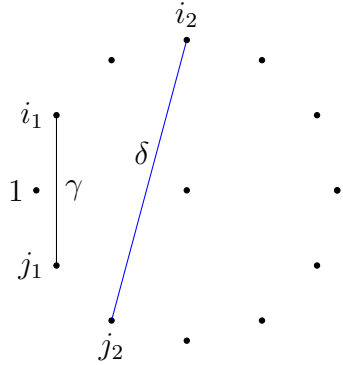
\end{proof}

\begin{proof}[Proof of Theorem~\ref{thm1}] To prove the theorem, it is enough to show that for any exceptional collection of curves $\{[\gamma_i]\}_{i = 1}^k$ the representations $\{V(\gamma_i)\}_{i=1}^k$ are an exceptional collection of representations and vice versa.

Let $\{[\gamma_i]\}_{i=1}^k$ be an exceptional collection of curves. We first show that $\{V(\gamma_i)\}_{i=1}^k$ is an exceptional collection of representations. Since $\{[\gamma_i]\}_{i=1}^k$ contains no bad pairs, Lemmas~\ref{common_endpt_lemma} and \ref{no_intersection_lemma} and Example~\ref{commuting_curves_fig} imply that for any $i,j \in \{1,\ldots, k\}$ at least one of $(V(\gamma_i),V(\gamma_j))$ or $(V(\gamma_j),V(\gamma_i))$ is an exceptional sequence.

We claim that there exists $i_1 \in \{1, \ldots, k\}$ such that $(V(\gamma_{i_1}), V(\gamma_i))$ is an exceptional sequence for all $i \in \{1, \ldots, k\}\backslash\{i_1\}$. This is clear if the curves $\{[\gamma_i]\}_{i = 1}^k$ do not form any cycles. On the other hand, suppose that there is a cycle formed by these curves.  Since $\{[\gamma_i]\}_{i=1}^k$ is an exceptional collection of curves, there is a unique pair of cycles, and it must be one of the configurations of the form shown in Figure~\ref{????}. 

Observe that the collection $\{[\gamma_i]\}_{i = 1}^k\backslash\{[\gamma]\}$ is also an exceptional collection of curves, and it does not have any cycles, where $[\gamma]$ is the equivalence class shown in Figure~\ref{????}. Therefore, there exists $[\delta] \in \{[\gamma_i]\}_{i = 1}^k\backslash\{[\gamma]\}$ such that $(V(\delta), V(\gamma_i))$ is an exceptional sequence for all $[\gamma_i] \in \{[\gamma_i]\}_{i = 1}^k\backslash\{[\gamma]\}$. If $[\delta]$ appears in one of the configurations in Figure~\ref{????}, then $(V(\gamma), V(\delta))$ is an exceptional sequence. So $[\gamma]$ is the desired equivalence class.

If  $[\delta]$ does not appear in Figure~\ref{????}, then $(V(\gamma), V(\delta))$ or $(V(\delta), V(\gamma))$ is an exceptional sequence. In the former case, $[\gamma]$ is the desired equivalence class. In the latter case, $[\delta]$ is the desired equivalence class. This establishes the claim.


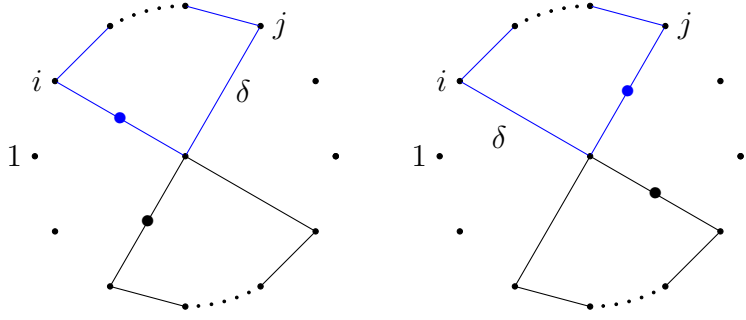
\begin{figure}[!htb]
\centering
\begin{tikzpicture}
\draw [fill] circle (1pt);
	
	\foreach \a in {0,30,...,360} {
		\draw[fill] (\a:2cm) circle (1pt); 
		}
		\foreach \a in {90,95,...,120} {
		\draw[fill] (\a:2cm) circle (0.5pt);
		}
		
		\foreach \a in {270,275,...,300} {
		\draw[fill] (\a:2cm) circle (0.5pt);
		}
\node[inner sep=0.5,fill=none,label=left:{$1$}]at (180:2) {};	

\node[inner sep=0.5,fill=none,label=left:{$i$}]at (150:2) {};	
\node[inner sep=0.5,fill=none,label=right:{$j$}]at (60:2) {};

\node[inner sep=0.5,fill=none,label=right:{{$\delta$}}] at (60:1) {};

	\node[inner sep=0.5, fill=none, label=right:{}] (a) at (60:2) {};
	\node[inner sep=0.5, fill=none, label=right:{}] (b) at (150:2) {};
	\node[inner sep=0.5, fill=none, label=right:{}] (c) at (120:2) {};
	\node[inner sep=0.5, fill=none, label=right:{}] (d) at (90:2) {};
	\node[inner sep=0.5, fill=none, label=below:{}] (e) at (0:0) {};

	\path[every node/.style={font=\sffamily\small}] 
	(a) edge[color=blue] node [right] {} (e);
	\path[every node/.style={font=\sffamily\small}] 
	(b) edge[color=blue] node [midway] {$\bullet$} (e);
	\path[every node/.style={font=\sffamily\small}] 
	(c) edge[color=blue] node [right] {} (b);
	\path[every node/.style={font=\sffamily\small}] 
	(d) edge[color=blue] node [right] {} (a);
	
	\node[inner sep=0.5, fill=none, label=right:{}] (a) at (60:-2) {};
	\node[inner sep=0.5, fill=none, label=right:{}] (b) at (150:-2) {};
	\node[inner sep=0.5, fill=none, label=right:{}] (c) at (120:-2) {};
	\node[inner sep=0.5, fill=none, label=right:{}] (d) at (90:-2) {};
	\node[inner sep=0.5, fill=none, label=below:{}] (e) at (0:0) {};
	
	\path[every node/.style={font=\sffamily\small}] 
	(a) edge[color=black] node [midway] {$\bullet$} (e);
	\path[every node/.style={font=\sffamily\small}] 
	(b) edge[color=black] node [midway] {} (e);
	\path[every node/.style={font=\sffamily\small}] 
	(c) edge[color=black] node [right] {} (b);
	\path[every node/.style={font=\sffamily\small}] 
	(d) edge[color=black] node [right] {} (a);
	
\end{tikzpicture} \ \ \ \ \ 
\begin{tikzpicture}
\draw [fill] circle (1pt);
	
	\foreach \a in {0,30,...,360} {
		\draw[fill] (\a:2cm) circle (1pt); 
		}
		\foreach \a in {90,95,...,120} {
		\draw[fill] (\a:2cm) circle (0.5pt);
		}
		
		\foreach \a in {270,275,...,300} {
		\draw[fill] (\a:2cm) circle (0.5pt);
		}
\node[inner sep=0.5,fill=none,label=left:{$1$}]at (180:2) {};	

\node[inner sep=0.5,fill=none,label=left:{$i$}]at (150:2) {};	
\node[inner sep=0.5,fill=none,label=right:{$j$}]at (60:2) {};

\node[inner sep=0.5,fill=none,label=right:{{$\delta$}}] at (170:1.5) {};

	\node[inner sep=0.5, fill=none, label=right:{}] (a) at (60:2) {};
	\node[inner sep=0.5, fill=none, label=right:{}] (b) at (150:2) {};
	\node[inner sep=0.5, fill=none, label=right:{}] (c) at (120:2) {};
	\node[inner sep=0.5, fill=none, label=right:{}] (d) at (90:2) {};
	\node[inner sep=0.5, fill=none, label=below:{}] (e) at (0:0) {};
	
	\path[every node/.style={font=\sffamily\small}] 
	(a) edge[color=blue] node [midway] {$\bullet$} (e);
	\path[every node/.style={font=\sffamily\small}] 
	(b) edge[color=blue] node [midway] {} (e);
	\path[every node/.style={font=\sffamily\small}] 
	(c) edge[color=blue] node [right] {} (b);
	\path[every node/.style={font=\sffamily\small}] 
	(d) edge[color=blue] node [right] {} (a);
	
	\node[inner sep=0.5, fill=none, label=right:{}] (a) at (60:-2) {};
	\node[inner sep=0.5, fill=none, label=right:{}] (b) at (150:-2) {};
	\node[inner sep=0.5, fill=none, label=right:{}] (c) at (120:-2) {};
	\node[inner sep=0.5, fill=none, label=right:{}] (d) at (90:-2) {};
	\node[inner sep=0.5, fill=none, label=below:{}] (e) at (0:0) {};
	
	\path[every node/.style={font=\sffamily\small}] 
	(a) edge[color=black] node [midway] {} (e);
	\path[every node/.style={font=\sffamily\small}] 
	(b) edge[color=black] node [midway] {$\bullet$} (e);
	\path[every node/.style={font=\sffamily\small}] 
	(c) edge[color=black] node [right] {} (b);
	\path[every node/.style={font=\sffamily\small}] 
	(d) edge[color=black] node [right] {} (a);
	
\end{tikzpicture}
\caption{The two possible types of pairs of cycles that may appear in an exceptional collection of curves. These configurations are considered up to the action of $\varrho$ on all curves in the configuration.}
\label{????}
\end{figure}

Let $V^1 = V(\gamma_{i_1})$. Now, inductively, let ${i_p} \in \{1, \ldots, k\}\backslash\{i_1, \ldots, i_{p-1}\}$ be an index such that the pair $(V(\gamma_{i_p}), V(\gamma_i))$ is an exceptional sequence for all $i \in \{1, \ldots, k\}\backslash\{i_1,\ldots, i_{p-1}\}$. Setting $V^p := V(\gamma_{i_p})$ for all $p \in \{1, \ldots, k\}$, we obtain that $(V^1, \ldots, V^k)$ is an exceptional sequence, as desired.

Next, we show any exceptional collection of representations $\{V^i\}_{i = 1}^k$ gives rise to an exceptional collection of curves. Without loss of generality, we assume that $(V(\gamma_i))_{i = 1}^k$ is an exceptional sequence where $V^i = V(\gamma_i)$ for all $i \in \{1, \ldots, k\}$. This implies that $(V(\gamma_i), V(\gamma_j))$ is an exceptional sequence for all $i$ and $j$ satisfying $i < j$. Therefore, Lemma~\ref{bad_pair_classification} implies that the set $\{[\gamma_i]\}_{i=1}^k$ does not contain any bad pairs.

Suppose the curves $\{[\gamma_i]\}_{i=1}^k$ have a configuration of the form shown in Figure~\ref{fig:forbidden_cconfiguration}. Without loss of generality, assume that the configuration is of the form shown in the left image of Figure~\ref{fig:forbidden_cconfiguration}, and let $\gamma_{i_1}$, $\gamma_{i_2}$, and $\gamma_{i_3}$ be the curves connecting $k_1$ to $n$, $k_2$ to $-n$, and $k_3$ to $n$, respectively. By Lemma~\ref{common_endpt_lemma}, we must have that $(V(\gamma_{i_1}),V(\gamma_{i_3})$, $(V(\gamma_{i_2}), V(\gamma_{i_1}))$, and $(V(\gamma_{i_3}), V(\gamma_{i_2}))$ must be exceptional sequences, and each pair of these representations taken in the opposite order is not an exceptional sequence. However, this implies that these three representations cannot be totally ordered in a way that previous an exceptional sequences consisting of all three. This contradicts that  $(V(\gamma_i))_{i = 1}^k$ is an exceptional sequence. 


Suppose that the curves $\{[\gamma_i]\}_{i=1}^k$ form a cycle that is not of the form shown in Figure~\ref{fig_special_configuration}. Because the collection $\{[\gamma_i]\}_{i=1}^k$ does not contain any bad pairs, it contains a cycle of curves none of which have a dot in its interior. Since this cycle is embedded in the plane, it encloses a region of the plane. Orient this cycle so that one travels counterclockwise around this region. Thus,  $[\gamma_{i_{s+1}}]$ is clockwise from $[\gamma_{i_{s}}]$ for all $s \in \{1, \ldots, p\}$, interpreting the indices cyclically.   This implies that $V(\gamma_{i_s})$ precedes in $V(\gamma_{i_{s+1}})$ in the exceptional sequence for all $s \in \{1, \ldots, p\}$, interpreting the indices cyclically. However, this contradicts that $(V(\gamma_i))_{i = 1}^k$ is an exceptional sequence. Consequently, no such cycle is formed by the curves $\{[\gamma_i]\}_{i=1}^k$. 

Lastly, we show that the curves $\{[\gamma_i]\}_{i=1}^k$ form at most one configuration of the form shown in Figure~\ref{fig_special_configuration}. Suppose two such configurations exist. Therefore, these curves must also form a cycle that is not of the form shown in Figure~\ref{fig_special_configuration} or a configuration of the form shown in Figure~\ref{fig:forbidden_cconfiguration}. This is a contradiction. We conclude that $\{[\gamma_i]\}_{i=1}^k$ is an exceptional collection of curves. This completes the proof.
\end{proof}

\begin{proof}[Proof of Theorem~\ref{thm:2}]
To prove the theorem, it is enough to show that for any exceptional sequence of curves $([\gamma_i])_{i = 1}^k$ the representations $(V(\gamma_i))_{i=1}^k$ are an exceptional sequence of representations and vice versa.

We first show that if $s,t \in \{1,\ldots, k\}$ satisfy $s < t$, then $(V(\gamma_s), V(\gamma_t))$ is an exceptional sequence, which implies that $(V(\gamma_i))_{i=1}^k$  is an exceptional sequence. By Theorem~\ref{thm1}, we know that $\{V(\gamma_i)\}_{i=1}^k$ is an exceptional collection of curves. So $[\gamma_s]$ and $[\gamma_t]$ do not form a bad pair.

If $[\gamma_s]$ and $[\gamma_t]$ have no common endpoints, then Lemma~\ref{no_intersection_lemma} implies that $(V(\gamma_s),V(\gamma_t))$ is an exceptional sequence. 

Next, assume that $[\gamma_s]$ and $[\gamma_t]$ have exactly one common endpoint. Assume that $[\gamma_s]$ and $[\gamma_t]$ share an endpoint $\ell$. Since $([\gamma_i])_{i = 1}^k$ is an exceptional sequence of curves, Lemma~\ref{common_endpt_lemma} implies that $(V(\gamma_s), V(\gamma_t))$ is an exceptional sequence of representations.

Now, assume $[\gamma_s]$ and $[\gamma_t]$ have at least two common endpoints. Since $[\gamma_s]$ and $[\gamma_t]$ do not form a bad pair, it must be the case that $[\gamma_t] = [\overline{\gamma}_s]$. By Example~\ref{commuting_curves_fig}, we know that  $(V(\gamma_s),V(\gamma_t))$ is an exceptional sequence. We obtain that $(V(\gamma_i))_{i=1}^k$ is an exceptional sequence.

Lastly, suppose that $(V^i)_{i=1}^k$ is an exceptional sequence of representations. Write $V^i = V(\gamma_i)$ for each $i \in \{1, \ldots, k\}$. Since $\{[\gamma_i]\}_{i = 1}^k$ is an exceptional collection of curves, there are no bad pairs among the classes $\{[\gamma_i]\}_{i = 1}^k$. We reduce to considering any pair of representations $V(\gamma_s)$ and $V(\gamma_t)$ with $s, t \in \{1, \ldots, k\}$ where $s < t$ and where $(V(\gamma_t), V(\gamma_s))$ is not an exceptional sequence. In this situation, since $\{[\gamma_i]\}_{i = 1}^k$ contains no bad pairs, the classes $[\gamma_s]$ and $[\gamma_t]$ have exactly one common endpoint. From Lemma~\ref{common_endpt_lemma}, we see that $[\gamma_s]$ and $[\gamma_t]$ are ordered in a way that is consistent with the definition  of an exceptional sequence of curves. We obtain that $([\gamma_i])_{i = 1}^k$ is an exceptional sequence of curves.
\end{proof}

\section{Proof of Theorem~\ref{thm:5.1}}\label{sec_thm_3_proof}

To prove Theorem~\ref{thm:5.1}, we recall some basic facts about the Coxeter groups of type $D_n$. The \textit{Grothendieck group} of a quiver $Q$ is the free abelian group $\mathcal{K}_0(Q)$ with basis given by the isomorphism classes $\widetilde{V}$ of representations $V$ of $Q$. The relations are given by $\widetilde{V} = \widetilde{U} + \widetilde{W}$ whenever there is an extension $0 \to U \to V \to W \to 0$. The dimension vector map is a group isomorphism $\textbf{dim}: \mathcal{K}_0(Q) \to \mathbb{Z}^{|Q_0|}$, and the simple representations of $Q$ give a basis of $\mathcal{K}_0(Q)$.

Let $\mathcal{V} = \mathcal{K}_0(Q)\otimes \mathbb{R}$, and endow this vector space with the symmetrized Euler form (i.e., $(\alpha, \beta) := \langle \alpha, \beta\rangle + \langle \beta, \alpha \rangle$ for any $\alpha, \beta \in \mathcal{V}$). The dimension vectors of simple representations of $Q$ give the basis $\{e_i\}_{i = 1}^{|Q_0|}$ of $\mathcal{V}$, via the isomorphism $\mathcal{K}_0(Q) \simeq \mathbb{Z}^{|Q_0|}$. 

A vector $v \in \mathcal{V}$ is a \textit{positive root} if $(v, v) = 2$ and $v$ is a non-negative integer combination of the $e_i$s. Each positive root gives rise to a reflection defined by \[s_v(w) = w - (v,w)v.\] Let $W$ be the group generated by these reflections. It can be shown that the reflections $s_i := s_{e_i}$ generate $W$, and the data $(W,\{s_i\}_{i = 1}^{|Q_0|})$ is a Coxeter system (see \cite[II.5.1]{humphreys1990reflection}).

For the remainder of the paper, $W$ will denote the Coxeter group $W$ that is associated with the quiver $Q^n$. In this case, the indecomposable representations are in bijection with the reflections in $W$ via the map $V \mapsto s_{V} := s_{\textbf{dim}(V)}$. The following lemma is a straightforward consequence of the definition of $V(\gamma)$.

\begin{lem}\label{factorization_lemma}
Let $[\gamma]$ be any equivalence class in $\text{Eq}(\Sigma_n)$ where $\gamma$ is the essential curve with endpoints $i$ and $j$ where $i > 0$. Then the reflection $s_{V(\gamma)} \in W$ may be expressed as follows
\[s_{V(\gamma)} = \left\{\begin{array}{l|l} s_{j-1}s_{j-2}\cdots s_{i+1}s_is_{i+1} \cdots s_{j-2}s_{j-1} & \text{$i < j \le n-1$}\\
s_ns_{n-2}\cdots s_{i+1}s_is_{i+1}\cdots s_{n-2}s_n & \text{$j = n$}\\
s_{n-1}s_{n-2}\cdots s_{i+1}s_is_{i+1}\cdots s_{n-2}s_{n-1} & \text{$j = -n$}\\
s_{-j}\cdots s_{n-2}s_ns_{n-1}s_{n-2}\cdots s_{i+1}s_is_{i+1}\cdots s_{n-2}s_{n-1}s_ns_{n-2}\cdots s_{-j} & \text{$-1 \le j \le -(n-2)$}\\
s_ns_{n-1}s_{n-2}\cdots s_{i+1}s_is_{i+1} \cdots s_{n-2}s_{n-1}s_n & \text{$j = -(n-1)$.}\end{array}\right.\]
\end{lem}

Following \cite{bjorner2006combinatorics}, the group $W$ may also be described as a certain subgroup of the symmetric group $\mathfrak{S}_{[n]^{\pm}}$ where $[n]^{\pm} := \{-1,-2, \ldots, -n, 1, 2, \ldots, n\}$. We review this now. Let $w \in \mathfrak{S}_{[n]^{\pm}}$ be any element with property that $w(-i) = -w(i)$ for any $i \in [n]^{\pm}$. We refer to such an element as a \textit{signed permutation}. We can write signed permutations in \textit{window notation} as $w = [a_1, \ldots, a_n]$ where $w(i) = a_i$. 

Let $\mathfrak{S}_n^D$ denote the group of all signed permutations $w \in \mathfrak{S}_{[n]^{\pm}}$ where the window notation for $w$ has even number of negative values (i.e., $|\{i \mid a_i < 0\}| = 2k$ for some $k \in \mathbb{Z}$ where $w = [a_1, \ldots, a_n]$). The reflections in $\mathfrak{S}_n^D$ are pairs of disjoint transpositions $((i,j)) := (i,j)(-i,-j)$ with $i \neq -j$. In general, we will refer to a product of disjoint cycles $(i_1, \ldots, i_k)(-i_1, \ldots, -i_k)$ as a \textit{paired cycle}, and we refer to a cycle $(i_1, \ldots, i_k, -i_1, \ldots, -i_k)$ as a \textit{balanced cycle}.

Next, by \cite[Proposition 8.2.3]{bjorner2006combinatorics}, there is a group isomorphism $\varphi: W \to \mathfrak{S}_n^D$ defined by 
\[\varphi(s_i) = \left\{\begin{array}{l|l}((i,i+i)) & i \neq n-1, n\\ ((n-1,n)) & i = n-1 \\ ((n-1,-n)) & i = n\end{array} \right.\]
and extended via the group multiplication. The following lemma follows from a case-by-case check using Lemma~\ref{factorization_lemma}.

\begin{lem}
Let $[\gamma]$ be any equivalence class in $\text{Eq}(\Sigma_n)$ where $\gamma$ is the essential curve with endpoints $i$ and $j$ where $i > 0$. Then $\varphi(s_{V(\gamma)}) = ((i,j)).$
\end{lem}

Before proving Theorem~\ref{thm:5.1}, we recall an alternative definition of the lattice of $D_n$ noncrossing partitions. For any $w \in W$, let $\ell(w)$ denote the minimum number $r$ such that $w$ can be expressed as a product of reflections in $W$. Define a poset $T^{D_n}$ on $W$ where $w \le w^\prime$ if $\ell(w^\prime) = \ell(w) + \ell(w^{-1}w^\prime)$. Now, fix the Coxeter element $c=s_1s_2\cdots s_{n-1}s_n$ of $W$. Note that \[\varphi(c) = (1, 2, \ldots, n-1, -1, -2, \ldots, -(n-1))(n,-n).\] 

By \cite[Theorem 1.1]{athanasiadis2004noncrossing}, the lattice of $D_n$ noncrossing partitions is isomorphic to the interval $[1, c] \subseteq T^{D_n}$ via the isomorphism
\[\begin{array}{rcl}
 [1,c] & {\longrightarrow} & NC^D(n)\\
 w & \mapsto & f(\varphi(w)).
\end{array}
\]
The map $f$, as introduced in \cite{athanasiadis2004noncrossing}, is defined by sending $\varphi(w)$ to the partition of $[n]^{\pm}$ whose nonzero blocks are formed by the paired cycles of $\varphi(w)$, and whose zero block (if it exists) is the union of the elements of all balanced cycles of $\varphi(w)$.

\begin{proof}[Proof of Theorem~\ref{thm:5.1}]
Let $(V^1, \ldots, V^k)$ be an exceptional sequence of representations of $Q^n$. By \cite[Lemma 3.9, 3.10]{ingalls2009noncrossing}, the element $s_{V^k}\cdots s_{V^1} \in W$ is a noncrossing partition in $[1,c]$. Moreover, \cite[Lemma 3.11]{ingalls2009noncrossing} implies that $s_{V^i} \cdots s_{V^1} < s_{V^j}\cdots s_{V^1}$ in $[1,c]$ for any $i, j \in \{1, \ldots, k\}$ with $i < j$. When $j = i+1$, the order relation in $[1,c]$ is a covering relation. Consequently, the map \[(V^1, \ldots, V^k) \mapsto (1, s_{V^1}, \ldots, s_{V^k}\cdots s_{V^1})\] defines a bijection between exceptional sequences of representations of $Q^n$ and saturated chains of $[1,c]$ that include the identity.

Combining the maps presented in this section with our bijection from Theorem~\ref{thm:2}, we obtain that the map \[([\gamma_1], \ldots, [\gamma_k]) \mapsto (\{\{i\}\}_{i \in [n]^{\pm}}, (f\circ \varphi)(s_{V(\gamma_1)}),\ldots, (f\circ \varphi)(s_{V(\gamma_k)} \cdots s_{V(\gamma_1)})) \] defines a bijection between exceptional sequences of curves on $\Sigma_n$ and saturated chains of $NC^D(n)$ that include the identity.
\end{proof}




\section*{Acknowledgements}{This project is the result of a CRM-ISM  summer scholarship at Universit\'e du Qu\'ebec \`a Montr\'eal. A. Garver thanks Jacob P. Matherne for interesting discussions on this project.}


\bibliography{sample}{}
\bibliographystyle{plain}


\end{document}